\documentclass[10pt]{article}
\def\commentflag{3}
\def\ct{2}
\usepackage{titlesec}
\usepackage{amssymb,amsmath,amsfonts,amsthm,enumerate,color,mathrsfs,hyperref,amssymb,extarrows}
\usepackage[shortlabels]{enumitem}
\usepackage[toc,title,titletoc,header]{appendix}

\usepackage{graphicx}
\usepackage{indentfirst}
\usepackage{multicol}
\usepackage{booktabs}
\usepackage{float} 
\usepackage{subfigure}

\newcommand{\zz}[1]{{\color{black}{#1}}}

\newcommand{\fn}[1]{\footnote{ {\color{blue}{#1}} }}

\newtheorem{theorem}{Theorem}[section]
\newtheorem{lemma}[theorem]{Lemma}
\newtheorem{proposition}[theorem]{Proposition}

\newtheorem{corollary}[theorem]{Corollary}

\newtheorem{remark}[theorem]{Remark}
\numberwithin{equation}{section}

\def \sss{\scriptscriptstyle}

\newcommand{\norm}[1]{\left\lVert #1 \right\rVert}

\def\eq{\Longleftrightarrow}
\def\q {\quad}

\def \l{\langle}
\def \r{\rangle}

\def\bb{\begin{equation}
  \left\{
   \begin{array}{l} }
\def\ee{   \end{array}
  \right.
  \end{equation}}

\def\mm{ \left[
 \begin{matrix}}
\def\nn{\end{matrix} \right] } 

\def\p{\partial}
\def \dd{\cdot}
\def \t{\times}

\def\n{\nu}
\def \w {\widetilde}
\def \h{\hat}

\def\d{\delta}
\def \vp{\varphi}
\def \na{\nabla}
\def \ep{\varepsilon}
\def \lad{\lambda}

\def\re{\mathfrak{Re}}
\def\im{\mathfrak{Im}}

\def \curl{\nabla \times}

\def \ccurl{{\rm  curl}}
\def \ddiv {{\rm div}}

\def \sgd  {\na_{S}}

\def \sdiv {\sdiv_S}

\def \td{T_D^k}
\def \kd{K_D^k}
\def \md{M_{\p D}^k}
\def \wg{\w{\gamma}_n^{-1}}

\def \b{\overline}

\def \NN{\mathbb{N}}

\def \R{\mathbb{R}}
\def \C{\mathbb{C}}

\def \S{\mathbb{S}}
\def \T{\mathbb{T}}
\def \TT{\mathcal{T}}
\def \rr{\mathcal{R}}

\def \P{\mathbb{P}}
\def \xx {\mathbb{X}}

\def \FF{{\rm F}}

\def \gk{g(x,y,k)}

\def \Z{\mathbb{Z}}
\def \I{\mathbb{I}}

\def \bvi{{ \boldsymbol \vp^j_{\lad_i}}}

\def \np{\mathcal{K}^{k,*}_{\p D}}
\def \ss{\mathcal{S}}

\def \A{\mathcal{A}}

\def \diag{{\bf diag}}

\def \ynm {Y_n^m}
\def \unm {U_n^m}
\def \vnm {V_n^m}
\def \ete {E^{TE}_{n,m}}
\def \etm {E^{TM}_{n,m}}
\def \wete {\w{E}^{TE}_{n,m}}
\def \wetm {\w{E}^{TM}_{n,m}}
\def \anm {\alpha_{n,m}}
\def \bnm {\beta_{n,m}}
\def \aanm{a^m_n}
\def \bbnm{b^m_n}
\def \gnm {\gamma_{n,m}}
\def \enm {\eta_{n,m}}
\def \hh{\mathcal{H}}
\def \jj{\mathcal{J}}

\def \hzz {H_0(\ddiv 0, D)}
\def \hbb {H_{0}^{-1/2}(\p D)}

\def \pd {\mathbb{P}_{{\rm d}}} 
\def \pw {\mathbb{P}_{{\rm w}}}

\title{Super-resolution in recovering embedded electromagnetic
sources \\ in high contrast media}
\begin{document}
\author{
Habib Ammari\footnote{Department of Mathematics, ETH Z\"{u}rich, R\"{a}mistrasse 101, CH-8092 Z\"{u}rich, Switzerland. 
The work of this author is partially supported by the
Swiss National Science Foundation (SNSF) grant 200021-172483.
(habib.ammari@math.ethz.ch).}
\and Bowen Li\footnote{Department of Mathematics, The Chinese University of Hong Kong, Shatin, N.T., Hong Kong. (bwli@math.cuhk.edu.hk).}
\and Jun Zou\footnote{Department of Mathematics, The Chinese University of Hong Kong, Shatin, N.T., Hong Kong.
The work of this author was
substantially supported by Hong Kong RGC General Research Fund (Project  14306718)
and NSFC/Hong Kong RGC Joint Research Scheme 2016/17 (Project N\_CUHK437/16).
(zou@math.cuhk.edu.hk).}
}
\date{}
\maketitle

\begin{abstract}
The purpose of this work is to provide a rigorous mathematical analysis of the expected super-resolution phenomenon in the time-reversal imaging of electromagnetic (EM) radiating sources embedded in a high contrast medium. It is known that the resolution limit is essentially determined by the sharpness of the imaginary part of the EM Green's tensor for \zz{the associated background}.
We first establish the close connection between the resolution and the material parameters and the resolvent of the electric integral operator, via the Lippmann-Schwinger representation formula. We then present 
an insightful characterization of the spectral structure of the integral operator for a general bounded domain and derive the pole-pencil decomposition of its resolvent in the high contrast regime. 
For the special case of a spherical domain, we provide some quantitative asymptotic behavior 
of the eigenvalues and eigenfunctions. These mathematical findings shall enable us to provide
a concise and rigorous illustration of the super-resolution in the EM source reconstruction in high contrast media. 
Some numerical examples are also presented to verify our main theoretical results. 
\end{abstract}

\section{Introduction}
In this work, we study the potential super-resolution phenomenon when using the time-reversal imaging method to reconstruct the EM sources embedded in general media with high refractive indices. Among the various imaging algorithms, the time-reversal approach is one of the most direct and simplest ones. Its principle is to exploit the reciprocity of wave propagation. Intuitively, we retrace the path of the wave observed in the far field backwards in chronology to find the location of its generating source \cite{wahab2014electromagnetic,wahab2014electromlossy,chen2013reverseac,chen2013reverse}. For a far-field imaging 
system using the time-reversal method, we know from the Helmholtz-Kirchhoff integral that its resolution is limited by the imaginary part of the Green's function of the wave equations associated with the background medium \cite{ammari2015mathematical,ammari2015super}. It is connected with the so-called Abbe diffraction limit (half of the operating wavelength) via the concept of full width at half maximum (FWHM) \cite{ammari2017sub,ammari2016mathematics}.  In a more precise way, the sharper the imaginary part of the Green's function, the smaller the full width at
its half maximum and the smaller scale the imaging system can resolve.

Over the past several decades, intensive efforts have been made to explore the potential of breaking the diffraction limit 
in two-fold: generating a better raw images, and recovering the finer details of raw images by
post-imaging processes. In this work, our discussion shall be restricted to the first procedure, that is, how to physically improve the resolution by obtaining the better 
a priori information. The Abbe diffraction limit actually results from the fact that the information about subwavelength details of the profile is carried out by the evanescent components  of  the scattered field that is basically unmeasurable in the far field \cite{bao2013near,bao2014near}, \zz{see also Proposition \ref{prop:asylarge}}. To break the resolution barrier, we may need to
capture the subwavelength information. It has been demonstrated in many different settings that using resonant media is a promising and feasible choice, e.g., the plasmonic nanoparticles \cite{ammari2017mathematicalscalar,ammari2016plasmaxwell,ammari2016surface}, the bubbly media \cite{ammari2018minnaert,ammari2017sub}, the Helmholtz resonators \cite{ammari2015mathematical}, and the high contrast media \cite{ammari2014medium,ammari2015super,ammari2018super}. Under specific circumstances, these resonant media can excite the resonances and serve as an amplifier that increases the strength of the subwavelength information of the sources encoded in the measured data.  In general, they are mathematically equivalent to eigenvalue problems \cite{ammari2015super,ammari2018minnaert,ammari2017mathematicalscalar}. 
It was demonstrated in \cite{ammari2017mathematicalscalar} that 
the surface plasmon resonance can be treated as an eigenvalue problem of the Neumann-Poincar\'{e} operator, 
which was further used to analyze the imaginary part of the Green's function and the possibility of achieving
the super-resolution by using plasmonic nanoparticles. For the bubbly media, it was shown in \cite{ammari2017sub} that 
the  super-focusing of acoustic waves can be obtained at frequencies near the Minnaert resonance. 
The inverse source problem was investigated in \cite{ammari2015super} 
for the Helmholtz equation and the super-resolution was explained 
based on the resonance expansion of the Green's function associated with the medium 
with respect to the generalized eigenfunctions of the Riesz potential $\kd$ (cf.\,\eqref{def:kd}). 
As a complement of the work \cite{ammari2015super}, the imaging of the target of high contrast was studied in \cite{ammari2018super} 
for the Helmholtz system
and the experimentally observed super-resolution was illustrated 
via the concept of scattering coefficients. In this work, we \zz{consider the three-dimensional EM wave governed by the full Maxwell equations, and, with the help of an electric integral operator $\td$, a solid mathematical foundation is provided for the the expected super-resolution phenomenon in the time-reversal reconstruction of EM sources embedded in a high contrast
medium. We  
also develop some analytical tools very different from the acoustic cases to discuss several critical issues that were not covered in \cite{ammari2015super,ammari2018super}.}

The contributions of this work are three-fold. Firstly, we derive the Lippmann-Schwinger equation to reveal the relations between the medium (shape and refractive indices) and its associated EM Green's tensor (cf.\,\eqref{eq:greeninho}), \zz{of which the explicit formula is not available.} \zz{It is worth emphasizing that this derivation is not as trivial and standard as one might think, and, in fact, our arguments and analysis are very different from the ones in \cite{ammari2015super} 
for the Helmholtz equation and are much more involved. The main difficulty in our case arises from the strong singularity of the EM Green's tensor so the standard approach (see, e.g., \cite{colton2012inverse,ammari2015super}) that works for the functions with $L^2$-regularity is not applicable.} To deal with this problem, we deliberately choose a smooth cutoff function to separate the singular part
from the Green's tensor $G$ so that the remaining regular part can be represented by the Lippmann-Schwinger equation. Since the singular term is explicitly constructed, our decomposition (see Theorems \ref{thm:resoinhogr_1} and \ref{thm:resoinhogr_2}) may also have potential applications in the numerical computation of $G$.  
Secondly, as we shall demonstrate, the mechanism underlying the super-resolution in resonant media is closely related to the spectral analysis of $\td$, which is still far from being complete. For the case of the electric permittivity being smooth enough on the whole space, the integral operator involved in the Lippmann-Schwinger equation is compact and well-studied \cite{colton2012inverse,costabel2010volume}. When the material coefficients 
have jumps across the medium interfaces, the integral operator is not compact and its spectral study is largely open. In \cite{costabel2015volume}, the authors investigated the essential spectrum of the integral operators arising from the EM scattering on the Lipschitz domain in two dimensions and gave a relatively complete characterization in various cases, which extended their earlier results in \cite{costabel2010volume,costabel2012essential} where only the smooth domain was considered. We refer the readers to \cite{rahola2000eigenvalues,budko2006spectrum} for the numerical study of the spectrum of EM volume integral operators. 
To explore the spectral properties of the integral operator $\td$ in three dimensions, we first show that all the eigenvalues of $\td$, except $-1$, of which the corresponding eigenspace consists of the nonradiating sources, lie in the upper-half plane of $\C$; see Theorem \ref{prop:diseig}. Then, by using the Helmholtz decomposition of $L^2$-vector fields, we obtain a characterization of the essential spectrum of $\td$ in a more concise and constructive manner than the existing ones \cite{costabel2012essential, costabel2015volume}. Combining the characterization with the analytic Fredholm theory, we further characterize its eigenvalues of finite type, and give the relation among these eigenvalues, the eigenvalues (point spectrum) and the essential spectrum in Theorem \ref{thm:mainspec}. To the best of our knowledge, it is the first time that the relations between the various types of spectra of $\td$ are clearly characterized in the literature. These results, along with the fundamental properties of Riesz projections, allow us to write the pole-pencil decomposition of the resolvent of $\td$. 
After that, we present more quantitative results for the case of a spherical domain. We rigorously establish the asymptotic forms of the eigenvalues of the integral operator, and prove that these complex eigenvalues are rapidly tending to the real axis in Theorem \ref{thm:asyeigenvalue}. We also observe that along these eigenvalue sequences, there is a localization phenomenon for the associated eigenfunctions \cite{grebenkov2013geometrical,nguyen2013localization}, with 
a mathematical illustration provided in Theorem \ref{thm:loceig}. In Appendix \ref{app:B}, we provide another possible perspective to investigate the spectral properties of $\td$ by regarding it as a quasi-Hermitian operator. 

Our third contribution is that by applying the pole-pencil decomposition to the Lippmann-Schwinger representation of the Green's tensor, we write the resonance expansion (eigenfunction expansion) for the imaginary part of 
the Green's tensor, and find that both eigenvalues and eigenfunctions are responsible for the super-resolution in the reconstruction of the EM embedded sources in the high contrast setting. Precisely, the localized eigenfunctions are highly oscillating and can encode the subwavelength information of the sources. Such information is further amplified when the high contrast approaches some resonant values, and then is back-propagated to reconstruct the subwavelength details of the sources.

The remainder of this work is organized as follows. In Section \ref{sec:iescp}, we first give a brief review of the resolution of the time-reversal method for the inverse source problem and then derive the Lippmann-Schwinger representation of the EM Green's tensor.  In Section \ref{sec:savio}, we investigate the spectral structure of the involved volume integral operator on a general domain (cf.\,\eqref{def:td}) and obtain the pole-pencil decomposition of its resolvent near the small regular value. 
We then proceed to provide more quantitative analysis of spectral properties for the spherical domain. With these mathematical findings, we provide a full explanation for the super-resolution in high contrast media in Section \ref{sec:super-resolu}. 
In addition, we will \zz{present the numerical evidences in the case of a spherical region to validate} our main theoretical results. \zz{Some details and} other useful and interesting results are given in Appendices \ref{app:A}, \ref{app:B} and \ref{app:C}.

We shall use some standard notations for the Sobolev spaces (see \cite{monk2003finite}) 
throughout this work. For a vector $x \in \R^3$, we denote its transport by $x^t$ and its polar form by $(|x|,\h{x})$ with $\h{x}:=x/|x| \in S^2$, 
where $S^2$ is the two dimensional unit sphere in $\R^3$.
We denote the inner product and outer product for two vector $u,v \in \R^3$ by $u^t \dd v$ and $u \t v$ respectively. We also need the tensor product operation  $\otimes$ of two vectors, i.e., given two vectors $u \in \R^n$ and $v \in \R^m$, $u \otimes v$ is a $n \times m$ matrix given by $(u \otimes v)_{ij} = u_i v_j$. And we always let vector operators act on matrices column by column. For a Banach space $X$ and its topological dual $X'$, we introduce the dual pairing $\l l,x\r_{X} : = l(x)$. 
We use $\oplus_\perp$ to denote the orthogonal sum in a Hilbert space, while the direct sum in a Banach space is denoted by $\oplus$.

\section{Resolution of imaging EM embedded sources}\label{sec:iescp}
In this section, we shall first introduce the time-reversal reconstruction of EM sources embedded in a high contrast medium and then review its resolution analysis. 
The main purpose of this section is to work out the explicit relation between the resolution limit and the contrast between the refractive indices of the dielectric inclusion and its surrounding medium.

Let us start with the introduction of some notation, definitions and conventions in this work. We consider a dielectric inclusion $D$ embedded in the free space $\R^3$, where $D$ is a bounded connected open set with a smooth boundary $\p D$ and the exterior unit normal vector $\n$. We assume the refractive index $n(x) \in L^\infty(\R^3)$ of the form:
\begin{equation*}
    n(x) = 1+ \tau \chi_D(x)\,,
\end{equation*}
where $\tau \gg 1$ is a positive real constant and $\chi_D$ is the characteristic function of $D$. Let $k$ and $k_\tau := k \sqrt{1 + \tau}$ be the wave numbers in the free space and in the medium $D$, respectively. Then we introduce the fundamental solution of the differential operator $-(\Delta + k^2)$ in $\R^3$: $ g(x,y,k) := \frac{e^{ik|x-y|}}{4 \pi |x-y|}$, $k \ge 0$. We define the Riesz potential $\kd$: 
\begin{equation}\label{def:kd}
    \kd[\varphi] = \int_D g(x,y,k) \vp(y)dy\quad \text{for}\ \vp \in L^2(D,\R^3)\,,
\end{equation}
which is a bounded linear operator from $L^2(D,\R^3)$ to $H^2_{loc}(\R^3,\R^3)$.
This further allows us to introduce the electric volume integral operator $\td$:
\begin{equation} \label{def:td}
    \td[\varphi] = (k^2 + \nabla \ddiv) \kd[\varphi]\in H_{loc}(\ccurl,\R^3)\quad \text{for}\ \vp \in L^2(D,\R^3)\,, 
\end{equation}
which satisfies
\begin{align}  \label{eq:maxwellint}
      \curl \curl \td[\varphi]- k^2 \td[\varphi] = k^2 \varphi \chi_D \q \text{in}\ \R^3\,,
\end{align}
in the variational sense,  together with the outgoing radiation condition:
\begin{equation} \label{mod:out}
    |x|\left(\curl \td[\varphi](x) \t \h{x} - i k \td[\varphi](x)\right) \to 0 \q \text{as} \ |x|\to \infty\,. 
\end{equation}
We say that a $L^2$-vector field $E$ solving the homogeneous Maxwell equations is radiating if it satisfies the radiation condition \eqref{mod:out} in the far-field, and of which we define the far-field pattern $E_\infty(\h{x})\in L_T^2(S^2)$ by the asymptotic form:
\begin{equation} \label{def:ffp}
    E_\infty(\h{x}) = \frac{e^{ik|x|}}{|x|} E_\infty(\h{x}) + O\left(\frac{1}{|x|^2}\right)\q \text{as} \ |x| \to \infty\,.
\end{equation}
The following surface integral operators are also needed:
\begin{align}\label{eq:deflayer}
    \ss^k_{\p D}[\vp] = \int_{\p D} g(x,y,k) \vp((y)d\sigma(y), \quad  \np[\vp] = \int_{\p D} \frac{\p}{\p \n_x} g(x,y,k) \vp((y)d\sigma(y)\quad \text{for}\  \vp \in H^{-\frac{1}{2}}(\p D)\,.
\end{align}
 We recall the normal trace formula for the gradient of $\ss_{\p D}^k$:
\begin{equation} \label{eq:jumprelation}
   \gamma_n\left(\na\ss^k_{\p D}[\vp]\right) = \left(\frac{1}{2} + \np\right)[\vp](x)\,, \q x \in \p D\,,
\end{equation}
where $\gamma_n [\dd] = \n^t \dd \dd$ is the normal trace mapping which is well-defined on the space $H(\ddiv,D)$.
For the case where the density function $\vp$ in $\ss^k_{\p D}$ is the tangent vector fields from $H_T^{-1/2}(\ddiv,\p D)$, we denote the operator by $\A_{\p D}^k$ instead in order to avoid any confusion. When $k = 0$, we omit the superscript $k$ in the above definitions for simplicity, e.g., we write $\ss_{\p D}$ for $\ss_{\p D}^0$. We are now ready to state the inverse source problem of our interest in this work, and analyze the resolution of the time-reversal reconstruction of the EM 
embedded sources.

Consider the following forward source problem associated with the medium $D$:
\begin{equation} \label{mod:eq}
\left\{ 
\begin{aligned}
   &\curl  \curl E (x) - k^2 n(x) E(x) = f(x)\,, \q x \in \R^3\,,   \\
     & E\ \text{satisfies the outgoing radiation condition \eqref{mod:out}}\,,
\end{aligned} \right.
\end{equation}
where $f \in L^2(D,\R^3)$ is the electric radiating source in the sense that $E$ has a nontrivial far-field pattern \cite{albanese2006inverse}.  The corresponding inverse source problem is aimed at reconstructing the source $f$ by using the electric field data $E_{{\rm meas}}(x)$ collected on the far-field measurement surface $\p B(0,\h{R})$, where the radius $\h{R}$ is large enough and $B(0,\h{R})$ contains $D$. In the distribution sense, the measured data $E_{{\rm meas}}(x)$ on $\p B(0,\h{R})$ can be written as 
\begin{equation} \label{eq:measureddata}
    E_{{\rm meas}}(x) = \int_{D} G(x,y,k)f(y) dy\,, \quad x \in \p B(0,\h{R})\,,
\end{equation}
where $G(x,y,k)$ is the Green's tensor of Maxwell's equations for the inhomogeneous background, defined by 
\begin{equation} \label{eq:greeninho}
    \curl  \curl G(x,y,k) - k^2 n(x) G(x,y,k) = \d(x- y)\I\,, \q x \in \R^3\,, \ y \in  \R^3\backslash \p D\,,
\end{equation}
such that each column of $G$ satisfies the outgoing radiation condition \eqref{mod:out}. Here, $\I$ is the $3\t 3$ identity matrix.  The existence of $G$ can be rigorously justified by the boundary integral equations (cf.\ \eqref{eq:repgre:1}-\eqref{eq:repgre:2}). In our following representation, $G(x,y,k)$ will usually occur with a unit polarization vector $p\in S^2$, i.e., $G(x,y,k)p$, physically denoting the electric field generated by the point dipole source $\d(x-y)p$ located at $y$,
and we will not give descriptions for the other similar  notations  if there is no ambiguity.

To re-emit the measured field $E_{{\rm meas}}(x)$ in \eqref{eq:measureddata} back to the source, we multiply it by 
$\b{G}$ (phase conjugation is the frequency domain counterpart of time reversal), which immediately leads us to the imaging functional:
\begin{equation} \label{eq:imiisp}
    I(z) = \int_{\p B(0,\h{R})} \b{G(z,x,k)} E_{{\rm meas}}(x) d\sigma(x)\,,
\end{equation}
where $z$ is any sampling point taken from the sampling region $\Omega$ which is a bounded domain satisfying $D \subset \Omega \subset B(0,R)$. The resolution of the above imaging functional is a standard consequence of the following corollary of the well-known Helmholtz-Kirchhoff identity\cite{chen2013reverse,ito2013direct}: for any $p,q\in S^2$, 
 \begin{equation} \label{eq:hkidencor}
    k \int_{\p B(0,\h{R})} (\overline{G(\xi,x,k)}q)^t \dd G(\xi,z,k)p d\sigma(\xi) = q^t\dd \im G(x,z,k)p + O\left(\frac{1}{\h{R}}\right)\,, \q \forall \ x,z \in \Omega \backslash \p D\,.
\end{equation}
To see this, we substitute \eqref{eq:measureddata} into \eqref{eq:imiisp}, and then readily obtain from \eqref{eq:hkidencor} that for an arbitrary probing direction $q \in S^2$, it holds that
\begin{align*}
   q^t \dd I(z) = \int_{\p B(0,\h{R})}q^t \dd \overline{G(z,x,k)}E_{{\rm meas}}(x) d\sigma(x) & = \int_{D}\int_{\p B(0,\h{R})} q^t \dd  \b{G(z,x,k)}G(x,y,k)f(y) d\sigma(x)  dy \\
    &  =\frac{1}{k} \int_{D} q^t \dd \im G(z,y,k)f(y) dy + O\left(\frac{1}{\h{R}}\right)\,,
\end{align*}
where we have used the reciprocity of the Green's tensor: $G(x,y,k)^t = G(y,x,k)$. Thus, we have that $I(z)$ can be approximated by
\begin{equation*}
     \h{I}(z) = \frac{1}{k} \int_D \im G(z,y,k)f(y)dy\,,\quad z \in \Omega\,,
\end{equation*}
when $\h{R}$ tends to infinity. To investigate the properties of $\h{I}$, it suffices to consider the imaginary part of the Green's tensor (with a polarization vector $p$):
\begin{equation*}
    \im G(z,z_0,k)p\,, \quad z_0 \in D\,, \ p \in S^2\,,
\end{equation*}
which is proportional to the raw image $I(z)$ of the point dipole source $f(y) = \d_{z_0}(y)p$ asymptotically. It
is worth emphasizing that $\im G$, unlike the acoustic case, is anisotropic in the sense that $q^t \dd \im G p$ may present  different features for different probing directions $q \in S^2$ and polarization directions $p \in S^2$, and hence yields a direction dependent diffraction barrier.   \zz{But we can still expect a better resolution in the image of $f$ obtained from the approximate functional $\h{I}(z)$, if $\im G(z,z_0,k)p$ exhibits subwavelength peaks.}

To figure out how the high contrast $\tau$ influences the behavior of the imaginary part of the Green's tensor, the Lippmann-Schwinger formulation may be adopted, 
as it was suggested in \cite{ammari2015super} for the acoustic case. However, it is not a trivial task 
to derive the Lippmann-Schwinger equation here as in \cite{ammari2015super} due to the strong singularity 
of the current Green's tensor $G(x,y,k)$ associated with the Maxwell equations for the inhomogeneous background.
We observe that $\im Gp$ does not satisfy the outgoing radiation condition \eqref{mod:out} although it obeys 
\begin{equation*} 
    \curl  \curl \im G(x,y,k)p - k^2 n(x) \im G(x,y,k)p = 0\,, \q x \in \R^3\,, \ y \in  \R^3\backslash \p D\,.
\end{equation*} 
Thus, we need to to deal directly with $G(z,z_0,k)p$ that solves the equation:
\begin{equation} \label{eq:fundapolar}
    \curl  \curl  G(z,z_0,k)p - k^2 n(x) G(z,z_0,k)p = \d_{z_0}(z)p\,, \quad z_0\in D, \ z \in \R^3 \backslash \p D\,,
\end{equation}
or equivalently,
\begin{align}
     \curl  \curl  \left[G(z,z_0,k) - G_0(z,z_0,k)\right]p -& k^2 \left[G(z,z_0,k) - G_0(z,z_0,k)\right]p \notag \\&= k^2 \tau \chi_D G(z,z_0,k)p\,, \quad z_0\in D, \ z \in \R^3 \backslash \p D\,, \label{eq:fundapolar_1}
\end{align}
where 
\begin{equation}\label{eq:greenfree}
    G_0(x,y,k) := \left(\I+\frac{1}{k^2}\na \ddiv\right)\gk \I\,
\end{equation}
is the Green's tensor of Maxwell equations for the free space with wave number $k$. \zz{
By \eqref{eq:maxwellint} and \eqref{eq:fundapolar_1}, the integral equation for $G$ may be formally formulated as}
\begin{equation*}
     G(z,z_0,k)p - G_0(z,z_0,k)p  = \tau \td\left[G(\dd,z_0,k)p\right](z), \quad z\in D\,.
\end{equation*}
Nevertheless, there is a strong singularity of $G(z,z_0,k)$ near $z_0$ (cf.\,\eqref{eq:singpriori}), 
resulting in the fact that $G(z,z_0,k)p \notin L^2(D,\R^3)$ and the evaluation of 
$\td\left[G(\dd,z_0,k)\right](z)$ makes no sense.

To address this issue, we need an a priori information on the singularity of Green's tensor $G$, which we shall observe
from the boundary integral equation for $G$.  With the help of the integral operator $\A_{\p D}^k$ introduced 
earlier in this section, we assume that $G(x,y,k)p$ has the following ansatz: for $y \in D$,
\begin{equation} \label{eq:repgre:1}
G(x,y,k)p = \begin{cases}
G_0(x,y,k_\tau)p + \curl\A_{\p D}^{k_\tau} [\phi](x) + \curl \curl \A_{\p D}^{k_\tau}[\psi](x)\,, &\mbox {$x \in D$}\,,
\\ 
\curl \A_{\p D}^k[\phi](x) + \curl \curl \A_{\p D}^k[\psi](x)\,, & \mbox {$x \in \R^3 \backslash \bar{D}$}\,,
\end{cases}
\end{equation}
and for $y \in \R^3\backslash \bar{D}$,
\begin{equation} \label{eq:repgre:2}
G(x,y,k)p = 
\begin{cases}
\curl\A_{\p D}^{k_\tau} [\phi](x) + \curl \curl \A_{\p D}^{k_\tau}[\psi](x)\,, & \mbox {$x \in D$}\,, \\ 
G_0(x,y,k)p +\curl \A_{\p D}^k[\phi](x) + \curl \curl \A_{\p D}^k[\psi](x)\,, & \mbox {$x \in \R^3 \backslash \bar{D}$}\,.
\end{cases}
\end{equation}
The densities $\phi,\psi \in H_T^{-1/2}(\ddiv,\p D)$ in \eqref{eq:repgre:1} and \eqref{eq:repgre:2} can be found by solving a boundary integral equation built via the trace formulas related to $\A_{\p D}^{k}$ \cite{ammari2013enhancement,ammari2018mathematical}. By \eqref{eq:repgre:1}, we readily see that near $z_0 \in D$, $G(z,z_0,k)p$ has
the same singularity as $G_0(z,z_0,k_\tau)p$ \zz{in the sense that}
\begin{equation} \label{eq:singpriori}
   G(z,z_0,k)p - G_0(z,z_0,k_\tau)p \in L^2(D,\R^3)  \,.
\end{equation}

We are now prepared to derive the Lippmann-Schwinger representation of $G$ in terms of $\td$ and $\tau$. The key idea here is to split $G$ into a singular term with compact support in $D$ and a regular remainder, and then establish the integral equation for the regular part instead. To do so, \zz{we construct a smooth cutoff function $\w{\chi}_{z_0}(z)$ with a compact support in $D$ satisfying $$
\w{\chi}_{z_0}(z) \equiv 1\  \text{on a small ball}\  B(z_0,r) \subset D\,,
$$ 
and define
\begin{equation} \label{def:tilg}
\w{g}(z,z_0,k) :=  \w{\chi}_{z_0}(z)g(z,z_0,k)\,,\q z \in \R^3\,,
\end{equation}
which helps us to separate the singularity indicated in \eqref{eq:singpriori} locally.} It follows that $\nabla_z \ddiv_z (\w{g}(z,z_0,k)p)$ is a distribution on $\R^3$ with its support and singular support, respectively, given by the compact set  ${\rm supp}(\w{\chi}{z_0})$ and the single point $\{z_0\}$.
We now write $G(z,z_0,k)p$ as
\begin{equation} \label{eq:decomgginho}
G(z,z_0,k)p = G_0(z,z_0,k)p - \frac{\tau}{k_\tau^2}\nabla_z \ddiv_z (\w{g}(z,z_0,k)p) + V(z,z_0,k)p\,,\q  z\in \R^3\,,
\end{equation}
where \zz{$V(\dd,z_0,k)p|_D$ defined by the above formula is an $L^2$-vector field, by \eqref{eq:singpriori} and \eqref{def:tilg}.} Substituting \eqref{eq:decomgginho} back into \eqref{eq:fundapolar}, we can find, by a direct computation, that $V(z,z_0,k)p$ satisfies
\begin{align} 
    \curl \curl V(z,z_0,k)p &- k^2 n(z) V(z,z_0,k)p \notag \\
    & =  \tau  k^2 \chi_D(z) (G_0(z,z_0,k)p - \frac{1}{k^2} \nabla_z \ddiv_z (\w{g}(z,z_0,k)p))\,, \label{eq:eqforvp}
\end{align}
where we have used the fact that $G_0$ is the fundamental solution to the homogeneous Maxwell equations and a simple but important observation that 
 \begin{equation*}
     k^2n(x)\frac{\tau}{k_\tau^2}\nabla_z \ddiv_z (\w{g}(z,z_0,k)p) = \tau \nabla_z \ddiv_z (\w{g}(z,z_0,k)p)\,, \q z \in \R^3\,.
 \end{equation*}
The above observation also suggests the reasons why it is necessary to restrict the singularity in the domain $D$.  Note that the source term in the right-hand side of \eqref{eq:eqforvp} is an $L^2$-vector field. We define a matrix function
\begin{equation} \label{def:gp}
   \w{G}(z,z_0,k):= G_0(z,z_0,k) - \frac{1}{k^2} \nabla_z \ddiv_z \left(\w{g}(z,z_0,k)\I\right)\,, \q z,z_0 \in D\,.
\end{equation}
Then the corresponding Lippmann-Schwinger equation for $Vp$ reads as follows:
\begin{equation*}
     V(z,z_0,k)p = \tau \td[\w{G}(\dd,z_0,k)p + V(\dd,z_0,k)p](z)\,, \q z \in D\,.
\end{equation*}
If $1 - \tau \td$ is invertible (as we shall see in Theorem \ref{prop:diseig}, this is always the case for a high contrast $\tau$), we further have 
\begin{align}
     V(z,z_0,k)p &= (1-\tau \td)^{-1}(\tau \td - 1+1)[\w{G}(\dd,z_0,k)p](z) \notag\\
     & = (1-\tau \td)^{-1}[\w{G}(\dd,z_0,k)p](z) - \w{G}(z,z_0,k)p\,, \q z\in D \,.\label{eq:reprevp}
\end{align}
Then it follows from the decomposition \eqref{eq:decomgginho}, the definition of $\w{G}$ in \eqref{def:gp} and the relation $k_\tau = k \sqrt{1 + \tau}$ that
\begin{align*}
    G(z,z_0,k)p & = \w{G}(z,z_0,k)p + (\frac{1}{k^2}-\frac{\tau}{k_\tau^2}) \nabla_z \ddiv_z (\w{g}(z,z_0,k)p) + V(z,z_0,k)p\\
    & =  \w{G}(z,z_0,k)p + \frac{1}{k_\tau^2} \nabla_z \ddiv_z (\w{g}(z,z_0,k)p) + V(z,z_0,k)p\,, \quad z, z_0\in D\,.
\end{align*}

Combining this decomposition with \eqref{eq:reprevp}, we arrive at the main result of this section.
\begin{theorem} \label{thm:resoinhogr_1}
The Green's tensor of the Maxwell equations \zz{\eqref{eq:fundapolar}} with a polarization vector $p \in S^2$, has the following representation: 
\begin{equation}  \label{eq:resoinhogr_1}
    G(z,z_0,k)p =  \frac{1}{k_\tau^2} \nabla_z \ddiv_z (\w{g}(z,z_0,k)p) + (1-\tau \td)^{-1}[\w{G}(z,z_0,k)p](z)\,, \quad z, z_0\in D\,,
\end{equation}
where $\w{g}$ and $\w{G}$ are given by
\eqref{def:tilg} and \eqref{def:gp}, respectively.
\end{theorem}
In the above construction, the definitions of $\w{g}$ and  $\w{G}$ depend on the position of $z_0$ and the explicit choice of the cutoff function $\w{\chi}_{z_0}(z)$. If we re-define $\w{g}$ and $\w{G}$ in
\eqref{def:tilg} and \eqref{def:gp} as 
\begin{equation} \label{def:tilg_2}
    \w{g}(z,z',k) = \w{\chi}_{z_0}(z)g(z,z',k),\q z\in \R^3, \ z' \in B(z_0,r)\,,
\end{equation}
and 
\begin{equation} \label{def:gp_2}
    \w{G}(z,z',k) = G_0(z,z',k) - \frac{1}{k^2} \na_z \ddiv_z (\w{g}(z,z',k)\I),\q z\in \R^3, \ z' \in B(z_0,r)\,,
\end{equation}
respectively, and revisit the proof of Theorem \ref{thm:resoinhogr_1} carefully, we can find the same representation of $G(z,z',k)p$ as the one in \eqref{eq:resoinhogr_1} for $z \in D$ and $z' \in B(z_0,r)$ \zz{but with $\w{g}$ and $\w{G}$ replaced by the ones in \eqref{def:tilg_2} and \eqref{def:gp_2}.} More generally, given an arbitrary compact subset $D'$ of $D$, we may replace the cutoff function $\w{\chi}_{z_0}(z)$ in \eqref{def:tilg_2} by another smooth cutoff function $\w{\chi}_{D'}$ such that $\w{\chi}_{D'}(z) \equiv 1$ on a small neighborhood of $D'$. Then, by a very similar argument as above, 
we can derive an improved variant of Theorem \ref{thm:resoinhogr_1}. 
\begin{theorem} \label{thm:resoinhogr_2}
Given a compact subset $D'$ of $D$, let $\w{g}$ be given by \eqref{def:tilg_2} with $\w{\chi}_{z_0}(z)$ replaced by the smooth cutoff function $\w{\chi}_{D'}(z)$ associated with $D'$, and let $\w{G}$ be defined as in \eqref{def:gp_2} with the newly defined $\w{g}$. Then the following decomposition of the Green's tensor $G(z,z',k)$ (cf.\,\eqref{eq:greeninho}) holds,
\begin{equation}  \label{eq:resoinhogr_2}
    G(z,z',k) = \frac{1}{k_\tau^2} \nabla_z \ddiv_z (\w{g}(z,z',k)\I) + (1-\tau \td)^{-1}[\w{G}(\dd,z',k)](z)\,, \q z\in D, \ z' \in D'\,.
\end{equation}
\end{theorem}
We can clearly see from \eqref{eq:resoinhogr_2} (or  \eqref{eq:resoinhogr_1}) how the high contrast $\tau$
affects the behavior of $G$. In the high contrast regime, i.e., $\tau \gg 1$, the first term of \eqref{eq:resoinhogr_2} involves
the contrast $\tau$ in an explicit way, and we can find that its imaginary part is of order $\tau^{-1}$ and thereby  negligible since $\im \w{g}(z,z',k)$ is a sufficiently smooth function.
At the same time, the second term in \eqref{eq:resoinhogr_2} is strongly influenced by the property of operator $(\tau^{-1}-\td)^{-1}$. If there are some poles of the resolvent of $\td$ near $\tau^{-1}$, we may expect that the term $(1-\tau \td)^{-1}[\w{G}(\dd, z',k)](z)$ blows up and hence $\im G$ exhibits a sharper peak than the one in the homogeneous space. These observations lead us to the investigations of the spectral structure as well as the resolvent of $\td$ in the next section, which serves as 
the mathematical preparations for a complete study of 
the possibility of achieving the super-resolution in high contrast media in Section \ref{sec:super-resolu}.

\section{Spectral analysis of the volume integral operator} \label{sec:savio}

For a bounded linear operator $A$ on a complex Banach space, we denote by $\sigma(A)$ its spectrum, by $\sigma_p(A)$ its eigenvalues (point spectrum), and by $(\lad - A)^{-1}$ the resolvent, which is an analytic operator-valued function defined on the resolvent set $\rho(A) := \C \backslash \sigma(A)$. We refer to the elements in $\rho(A)$ as the regular values of $A$. We have seen in Section \ref{sec:iescp} that the resolution limit in the EM inverse source problem is closely related to the behaviour of the resolvent $(\lad - \td)^{-1}$ near the small regular value $\tau^{-1}\ll 1$. 

\subsection{Spectral structure} \label{subsec:specstr}

\zz{In this subsection, we are going to first consider the distribution of eigenvalues of $\td$ and then give  characterizations of the essential spectrum and eigenvalues of finite type (their definitions will be given after Corollary \ref{cor:eigespe-1}). These  results are  fundamental to the pole-pencil decomposition of the resolvent $(\lad - \td)^{-1}$ that shall be derived in Section \ref{subsec:ppdec}.} We start with an easily observed but quite important lemma for our later use.
\begin{lemma} \label{lem:eqtrass}
For the integral operator $\td$ defined by \eqref{def:td}, we have $0 \notin \sigma_p(\td)$. Moreover, the eigenvalue equation $(\lad - \td)[\vp] = 0$ has nontrivial solutions for some $\lad \in \C$ \zz{(i.e., $\lad \in \sigma_p(\td)$)} if and only if the following transmission problem has a nontrivial radiating solution $u \in H_{loc}(\ccurl,\R^3)$,
\begin{equation} \label{eq:transim}
    \curl \curl u - k^2 u = \frac{k^2}{\lad}u \chi_D \q \text{in}\ \R^3\,.
\end{equation}
In this case, the solution $u$ to \eqref{eq:transim}, restricted on $D$, is an eigenfunction of $\td$ associated with $\lad$. 
\end{lemma}

\begin{proof}
Suppose $(\lad,\vp)$ is the eigenpair of $\td$, i.e., $\td[\vp] = \lad \vp$, $\vp \neq 0$, which directly yields, by \eqref{eq:maxwellint}, 
\begin{equation} \label{eq:midforu}
    (\curl \curl - k^2)\td[\vp] = (\curl \curl - k^2)\lad \vp = k^2 \vp \chi_D \q \text{in} \ \R^3. 
\end{equation}
We readily see that if $\lad = 0$, then $\vp = 0$ on $D$, from which it follows that $0 \notin \sigma_p(\td)$ and $\lad$ in \eqref{eq:midforu} does not vanish. Since $\vp$ is the eigenfunction of $\td$ with eigenvalue $\lad$, we can write the right-hand side of \eqref{eq:midforu} as $k^2\td[\vp/\lad]\chi_D$ and then conclude that $\td[\vp]$ is a nontrivial solution of \eqref{eq:transim}. Conversely, if $u$ is a nontrivial solution of \eqref{eq:transim}, by the
uniqueness of a solution to the Maxwell source problem and \eqref{eq:maxwellint}, we have $u = \td[u/\lad]$,  which also implies that $u|_D$ is an eigenfunction of $\td$ associated with $\lad$.  
\end{proof}

We denote the interior wave number $k\sqrt{1 + \lad^{-1}}$ in \eqref{eq:transim} by $k_\lad$. \zz{Here and throughout this work, we consider the principal branch of $\sqrt{\dd}$ with the branch cut given by $(-\infty,0]$.}
It should be stressed that the equation \eqref{eq:transim} is defined on the whole space $\R^3$ and understood in the variational sense. This fact immediately yields $\curl u \in H_{loc}(\ccurl,\R^3)$, and hence 
$\curl u \in H_{loc}^1(\R^3,\R^3)$ by noting that $\ddiv (\curl u) = 0$ and making use of the embedding theorem (cf.\,\cite[Theorem 2.5]{amrouche1998vector}). These facts can also be verified by the integral representation of $u$, i.e., $u = \td[u/\lad]$. We now give the first main result of this subsection, concerning an a priori characterization of the distribution of the eigenvalues and eigenspaces of $\td$. \zz{The proof follows a similar spirit of 
the one for proving the uniqueness of a solution to the direct acoustic scattering problem (cf.\,\cite[Theorem 2.14]{colton2012inverse}) but pays a special attention to the ranges of the eigenvalues and the topology of the domain.
}

\begin{theorem} \label{prop:diseig}
For a bounded smooth domain $D$, we have that if $\lad \in \sigma_p(\td)\backslash\{-1\}$, then $\im \lad > 0$. Suppose that $\R^3\backslash \bar{D}$ is connected. We have that if $\lad = -1$ is an eigenvalue of $\td$, then the associated eigenspace must be contained in $\na H_0^1(D)$. 
\end{theorem}

\begin{proof}
We assume that $u \in H_{loc}(\ccurl,\R^3)$ is a radiating solution to \eqref{eq:transim}, or equivalently, the following system:
\begin{equation} \label{eq:trans12}
    \left\{ \begin{array}{ll}
     \curl \curl u - k_\lad^2 u = 0    & \text{in}\ D\,, \\
     \curl \curl u - k^2 u = 0    & \text{in}\ \R^3 \backslash \bar{D}\,, \\
     \left[\n \t u\right] = 0, \quad \left[\n \t \curl u\right] = 0 &  \text{on}\ \p D\,, 
    \end{array}\right.
\end{equation}
where $\lad \neq 0$ is a complex number with $\im \lad \le 0$. We shall prove that if $\lad \neq -1$ (equivalently, $k_\lad \neq 0$), $u$ must be zero everywhere; if $\lad  = -1$, then $u \in \na H_0^1(D)$, provided that the open set $\R^3\backslash \bar{D}$ is connected. For this purpose, choose an open ball $B(0,R)$ centered at the origin with large enough radius $R$ such that $\bar{D}\subset B(0,R)$, and multiply both sides of the second equation in the system \eqref{eq:trans12} by the test function $\bar{u}$. Then a direct integration by parts on $B(0,R)\backslash \bar{D}$ gives us
\begin{align}
    0 &= \int_{B(0,R)\backslash \bar{D}}\curl \curl u \dd \bar{u}-k^2 u \dd \bar{u} dx \notag \\ 
     & = \int_{B(0,R)\backslash \bar{D}} |\curl u|^2  - k^2 |u|^2 dx + \int_{\p B(0,R)} \h{x} \t \curl u \dd \bar{u}d\sigma(x) - \int_{\p D} \n \t \curl u \dd \bar{u}d\sigma(x) \notag\\
     & = \int_{B(0,R)\backslash \bar{D}} |\curl u|^2  - k^2 |u|^2 dx -ik \int_{\p B(0,R)}|u|^2  d\sigma(x) + O\left(\frac{1}{R}\right) - \int_{\p D} \n \t \curl u \dd \bar{u}d\sigma(x)\,, \label{eq:trans12pro_1}
\end{align}
where we have used the radiation condition \eqref{mod:out} 
and the fact that $\curl u \in H^1_{loc}(\R^3,\R^3)$. By taking the imaginary parts of both sides  of \eqref{eq:trans12pro_1} and letting $R$ tends to infinity, we have 
\begin{equation} \label{eq:trans12pro_2}
   \im \int_{\p D} \n \t \curl u \dd \bar{u}d\sigma(x) = -k\int_{S^2} |u_\infty|^2 d\sigma(\h{x})\le 0\,.
\end{equation}
Here, $u_\infty$ is the far-field pattern of $u$ given by \eqref{def:ffp}. We now consider the field inside the domain.  Similarly, with the help of an integration by parts over $D$ and the first equation in \eqref{eq:trans12}, we obtain
\begin{equation} \label{eq:trans12pro_3}
     -\int_D   |\curl u|^2 - k_\lad^2 |u|^2 dx  =  \int_{\p D} \n \t \curl u \dd \bar{u} d \sigma(x)\,,
\end{equation}
and its imaginary part 
\begin{equation} \label{eq:trans12pro_4}
    \im \int_{\p D} \n \t \curl u \dd \bar{u} d \sigma(x) = \im \int_D k^2 \frac{1}{\lad} |u|^2 dx \,.
\end{equation}
Noting that $\im \lad^{-1} = - \im (\lad/|\lad|^2)\ge 0$, we readily have $$\im \int_{\p D} \n \t \curl u \dd \bar{u} d \sigma(x) = 0\,,$$ 
by \eqref{eq:trans12pro_2} and \eqref{eq:trans12pro_4}, since the tangential traces of $u$ and $\curl u$ are continuous. Then, we see from the above formula and \eqref{eq:trans12pro_2} that the far-field pattern $u_\infty$ vanishes, and thus $u$ vanishes in the unbounded connected component of $\R^3\backslash \bar{D}$ by  
Rellich's lemma (cf.\,\cite[Theorem 6.10]{colton2012inverse}).
Therefore, it follows that
\begin{equation} \label{eq:trans12pro_5}
    \n \t u = 0 \,, \  \n \t \curl u = 0 \quad \text{on}\  \Gamma_0\,,
\end{equation}
where $\Gamma_0$ is the boundary of the unbounded component of $\R^3\backslash \bar{D}$. 

To complete the proof, let us first consider the simple case: $\lad \neq -1$, where the interior wave number $k_\lad$ does not vanish. The desired result that  $u = 0$ in $D$ directly follows from \eqref{eq:trans12pro_5} and the Holmgren's Theorem (cf.\,\cite[Theorem 6.5]{colton2012inverse}). We now consider the other case where $\lad = -1$ under the condition that $\R^3\backslash \bar{D}$ is connected. In this case, we only have $\curl u = 0$ in $D$, i.e., $u \in H_0(\ccurl 0, D)$, from \eqref{eq:trans12pro_3} and the observation $\Gamma_0 = \p D$. 
Recalling \eqref{eq:kerofcurl}, we have the following characterization of $H_0(\ccurl 0, D)$:
\begin{equation*}
    H_0(\ccurl 0, D) = \na H_0^1(D)\,,
\end{equation*}
since the $\R^3\backslash \bar{D}$ is connected and thus the corresponding normal cohomology space $K_N(D)$ is trivial. Therefore, we can conclude $u = \na p$ for some $p\in H_0^1(D)$ if $\lad = -1$ is an eigenvalue, and complete the proof. 
\end{proof}
The above theorem does not tell us whether $\lad  = - 1$ is an eigenvalue or not. However, if we extend an $L^2$-field $u$ from $\na H_0^1(D)$, or more generally, $H_0(\ccurl0, D)$, by zero outside the domain $D$, i.e., $\chi_D u$, we can find that it solves the system \eqref{eq:trans12} for $\lad = -1$, which indicates that $\lad = -1$ is indeed an eigenvalue of $\td$. Thus, we actually have the following corollary.
\begin{corollary} \label{cor:eigespe-1}
For a bounded smooth domain $D$, $\lad = -1$ is always an eigenvalue of $\td$ with the associated eigenspace containing $H_0(\ccurl0,D)$. If $\R^3 \backslash \bar{D}$ is connected, then the eiganspace is equal to $\na H_0^1(D)$.
\end{corollary}

To proceed, we need the following concepts about the spectrum of a bounded linear operator $A$. We say that $\lad \in \sigma(A)$ is an eigenvalue of finite type if and only if $\lad$ is an isolated point in $\sigma(A)$ and the corresponding Riesz Projection $P_{\lad}$: 
\begin{equation} 
    P_\lad(A) = \frac{1}{2\pi i}\int_{\Gamma}(z - A)^{-1}d z,
\end{equation}
is a finite rank operator, where $\Gamma$ is a Cauchy contour in $\C$ enclosing only the eigenvalue $\lad$ among $\sigma(A)$, and the definition does not depend on the choice of $\Gamma$. The other concept is the essential spectrum $\sigma_{ess}(A)$ defined by
\begin{equation*}
    \sigma_{ess}(A) = \{\lad \in \C\,;\ \lad \I - A \ \text{is not Fredholm operator}\}\,.
\end{equation*}
\zz{Inspired by the work \cite{costabel2010volume} where the strongly singular volume integral equation associated with the EM scattering problem was transformed to a coupled surface-volume system involving only weakly singular kernels by introducing an additional variable on the boundary via an integration by parts, here we exploit the Helmholtz decomposition of $L^2$-vector fields to obtain another operator matrix similar to the one in \cite{costabel2010volume} but with fully decoupled unknown variables.
This newly derived system enables us to see a clear and insightful spectral structure of $\td$.}

We now recall from Theorem \ref{lem:decoofl2} the Helmholtz decomposition of $L^2$-vector fields:
\begin{align} \label{eq:helmdecomp}
    L^2(D,\R^3) = \nabla H^1_0(D) \oplus_\perp H_0(\ddiv 0, D) \oplus_\perp W\,,
\end{align}
where $W$ is the function space consisting of $H^1$-harmonic functions and $H_0(\ddiv 0, D) = \ccurl \w{X}_N^0 \oplus_\perp K_T(D)$. Denote by $\mathbb{P}_0$, $\pd$ and $\pw$ the projections from $L^2(D,\R^3)$ to $\nabla H^1_0(D)$, $H_0(\ddiv 0, D)$ and $W$, respectively. In Appendix \ref{app:A}, we show how these subspaces are connected with the divergence, curl and normal trace of a vector field. In particular, we have $\mathbb{P}_0u = - \na \S \ddiv u$  and $\pw u = \wg \gamma_n(u+\na \S \ddiv u)$; see Appendix \ref{app:A} for the definitions of operators $\S$ and $\wg$. For our subsequent analysis, we introduce a product space: 
\begin{equation*}
     \xx:= \na H_0^1(D) \t \hzz \t H_{0}^{-\frac{1}{2}}(\p D)\,,
\end{equation*}
equipped with the norm $\norm{\FF}_{\xx}:= \norm{f_1}_{L^2(D)} + \norm{f_2}_{L^2(D)} + \norm{f_3}_{\hbb}$ for $\FF = (f_1,f_2,f_3)\in \xx$, which is isomorphic to $L^2(D,\R^3)$  via the isomorphism $\Xi : f \to \Xi[f] = (\mathbb{P}_0 f, \pd f, \w{\gamma}_n \pw f)$. By using the isomorphism $\Xi$, \zz{we define an operator $\TT_D^k$ on $\xx$ by}
\begin{equation} \label{def:matrixtd}
    \TT_D^k :=  \Xi\td\Xi^{-1}\,,
\end{equation}
which \zz{is similar to $\td$ and hence} has the same spectral properties as $\td$. \zz{We remark that} the inverse of $\Xi$ is given by 
$\Xi^{-1}(f_1,f_2,f_3) = f_1 + f_2 + \wg f_3$.

\zz{We proceed to consider the spectral analysis of $\TT_D^k$.} We first observe that $\nabla H_0^1(D)$ and divergence-free vector fields $H(\ddiv 0, D)$ are $\td$-invariant spaces. In fact, for $\phi \in H_0^1(D)$, we have
\begin{align} \label{eq:invanah}
    T^k_D[\nabla \phi] = k^2 \nabla K^k_D[\phi] + \nabla \Delta K^k_D[\phi] = - \nabla \phi\,,
\end{align}
which can be verified \zz{by using integration by parts with} the fact that $\phi$ has zero trace on $\p D$. On the other hand, by a density argument and the fact that $\ddiv: L^2(D,\R^3) \to H^{-1}(D)$, we have 
\begin{equation*} 
    \ddiv T^k_D[\phi] = - \ddiv \phi \q \text{for}\ \phi \in L^2(D,\R^3)\,.
\end{equation*}
By these observations and the definition of
$\TT_D^k$ (cf.\,\eqref{def:matrixtd}), we can write the operator matrix $\mathcal{T}_D^k$ as follows:
\begin{equation} \label{eq:matrixtd}
  \TT_D^k = \mm - 1 & 0 & 0 \\
     0 & \pd \td & \pd \td \wg \\
     0 &  \gamma_n \td  & \gamma_n \td \wg \nn.
\end{equation}
To further analyze the properties of $\TT_D^k$, we need to work out explicit formulas for the operators involved in \eqref{eq:matrixtd}, which are only defined in an abstract way. To do so, a direct calculation gives us that 
\begin{align} \label{eq:tdhdiv}
    T^k_D[\vp] = k^2 K^k_D[\vp] - \nabla \ss^k_{\p D}[\vp \dd \n] = k^2 K^k_D[\pd \vp + \pw \vp] - \nabla \ss^k_{\p D}[\gamma_n \pw \vp]
\end{align}  
holds for $\vp \in H(\ddiv0,D)$. Then, we take the normal trace on both sides  of \eqref{eq:tdhdiv} and find 
\begin{align} \label{eq:tdhdivnor}
    \gamma_n \td [\vp] & = k^2\gamma_n \kd[\pd \vp + \pw \vp] - (\frac{1}{2} + \np) [\gamma_n \pw \vp] \q \text{for}\ \vp \in H(\ddiv0,D)\,,
\end{align}
where we have used the normal trace formula \eqref{eq:jumprelation} for $\na\ss^{k}_{\p D}$. By \eqref{eq:tdhdiv} and \eqref{eq:tdhdivnor}, we readily have 
\begin{equation} \label{eq:explicitmtd}
\left\{ \begin{aligned}
  &\pd \td [\dd] = k^2 \pd \kd[\dd]\,, && \gamma_n \td[\dd] = k^2 \gamma_n K^k_D[\dd]   && \text{on} \ \hzz\,, \\
  &\pd \td \wg [\dd] = k^2 \pd K^k_D \wg[\dd] - \pd \nabla \ss^k_{\p D}[\dd]\,, &&  \gamma_n \td \wg[\dd] = k^2\gamma_n \kd \wg [\dd] - (\frac{1}{2} + \np) [\dd]       && \text{on} \ \hbb\,.
\end{aligned}\right.
\end{equation}

We are now in a position to prove the following lemma. 
\begin{lemma} \label{lem:speccomp}
$\rr_D^k := \TT_D^k - \diag(-1,0,-\frac{1}{2})$ is a compact operator on $\xx$.
\end{lemma}
\begin{proof}
To prove the compactness of $\rr^k_D$ on the product space $\xx$, it suffices to show that each block in $\rr^k_D$ is compact. By the mapping property of $\kd$ and Rellich's lemma for Sobolev spaces, we can obtain that $\pd \kd$ and $\gamma_n \td$ are compact operators from $\hzz$ to $\hzz$ and $\hbb$, \zz{namely, the operators $(\rr_D^k)_{2,2}$ and $(\rr_D^k)_{3,2}$ are compact (cf.\,\eqref{eq:explicitmtd}).} Meanwhile, a further fact that $\np$ is compact gives us the compactness of  \zz{$(\rr_D^k)_{3,3} = \gamma_n\td \wg + 1/2$} on $\hbb$, by \eqref{eq:explicitmtd}. To show that  \zz{$(\rr_D^k)_{2,3} = \pd \td \wg$} is compact from $\hbb$ to $\hzz$, we write it, by using \eqref{eq:explicitmtd}, as
\begin{equation*}
\pd \td \wg[\dd] = 
(k^2 \pd K^k_D \wg - \pd \nabla (\ss^k_{\p D}-\ss_{\p D}))[\dd] - \pd \nabla \ss_{\p D}[\dd]\,,
\end{equation*}
where the first term is obviously compact, and the second term actually vanishes due to the fact that $\na \ss_{\p D}[\dd]\in W$. The proof is complete.
\end{proof}
By Lemma \eqref{lem:speccomp} and the fact that 
the essential spectrum is stable under a compact perturbation \cite{gohberg1990classes}, we directly have the characterization of the essential spectrum \cite{costabel2012essential}:
\begin{equation*}
    \sigma_{ess}(\td) = \sigma_{ess}(\TT_D^k)  = \sigma_{ess}(\diag(-1,0,-\frac{1}{2})) =  \{-1,0,-\frac{1}{2}\}\,, 
\end{equation*}
and $\lad-\td$ is an analytic Fredholm operator function with index zero on $\C\backslash \sigma_{ess}$ as a consequence of the definition of essential spectrum \zz{and the fact that the Fredholm index ${\rm ind}(\lad - T_D^k)$ is a constant on a connected open set.} Then, by using the analytic Fredholm theory \cite{gohberg1990classes} and Theorem \ref{prop:diseig}, we can conclude that $(\lad-\td)^{-1}$ is extended to a meromorphic function on $\C\backslash \sigma_{ess}(\td)$ with its poles \zz{being} a discrete and countable bounded set given by $\sigma_p(\td) \backslash \sigma_{ess}(\td)$, and for some $\lad_0 \in \sigma_p(\td) \backslash \sigma_{ess}(\td)$ and $\lad$ in a sufficiently small neighborhood of $\lad_0$, $(\lad-\td)^{-1}$ has the following Laurent expansion:
\begin{equation} \label{eq:lauexp}
    (\lad-\td)^{-1} = \sum_{n=-q(\lad_0)}^\infty (\lad - \lad_0)^n T_n\,,
\end{equation}
where $T_0$ is Fredholm operator with index zero, and $T_i$, $-q(\lad_0)\le i \le -1$, are finite rank operators \zz{with $q(\lad_0)$ being a positive integer.} 

From now on we shall denote the set of all the eigenvalues of finite type of $\td$ by $\sigma_f(\td)$.
To better understand this set, we recall 
the following fundamental  property  concerning the Riesz projection (cf.\,\cite[Thm 2.2]{gohberg1990classes}).
\begin{lemma}\label{lem:fpriproj}
For a bounded linear operator $A$ on a Banach space $X$, let $\sigma$ be an isolated part of $\sigma(A)$ and $P_\sigma(A)$ is the associated Riesz projection. Then both ${\rm im} P_\sigma(A)$ and ${\rm ker} P_\sigma(A)$ are the invariant subspaces of $A$ with $\sigma(A|_{{\rm im} P_\sigma}) = \sigma$ and   $\sigma(A|_{{\rm ker} P_\sigma(A)}) = \sigma(A)\backslash \sigma$. Moreover, $X$ has the direct sum decomposition: $X = {\rm im} P_\sigma(A) \oplus {\rm ker} P_\sigma(A)$.
\end{lemma}
From Lemma \ref{lem:fpriproj} it immediately follows that $\sigma_f(\td)$ is a subset of $\sigma_p(\td)$. Conversely, note from \eqref{eq:lauexp} that for $\lad_0 \in \sigma_p(\td) \backslash \sigma_{ess}(\td)$,
\begin{equation*}
    P_{\lad_0}(\td) = \frac{1}{2 \pi i}\int_{\Gamma} (\lad - \lad_0)^{-1}T_{-1} d \lad = T_{-1}
\end{equation*}
is a finite rank operator. By this fact, together with the definition of eigenvalues of finite type and $\sigma_f(\td) \subset \sigma_p(\td)$, we readily have 
\begin{equation} \label{eq:relaspe}
    \sigma_p(\td)\backslash\sigma_{ess}(\td) = \sigma_f(\td)\backslash\sigma_{ess}(\td)\,.
\end{equation}
In fact, we can obtain a sharper version of \eqref{eq:relaspe} by some further observations. We first note from Lemma \ref{lem:eqtrass} and Theorem \ref{prop:diseig} that $\{0,-\frac{1}{2}\} \not\subset \sigma_p(\td)$ \zz{and further that}
\begin{equation} \label{eq:relaspe_2}
    \sigma_p(\td)\backslash\sigma_{ess}(\td) = \sigma_p(\td)\backslash \{-1\}\subset\{\lad \in \C\,;\ \im \lad > 0\}\,.
\end{equation}
To consider the relation between $\sigma_f(\td)$ and $\sigma_{ess}(\td)$, we need a general result from 
\cite[Lemma 4.3.17]{davies2007linear}.
\begin{lemma}
Let $A$ be a bounded linear operator, and let $\lad_0$ be an isolated point in $\sigma(A)$. Then we have $\lad_0 \in \sigma_{ess}(A)$ if and only if
the Riesz projection $P_{\lad_0}(A)$ has an infinite-dimensional range. In particular, we have 
\begin{equation*}
    \sigma_{ess}(A) \bigcap \sigma_{f}(A) = \emptyset\,. 
\end{equation*}
\end{lemma}
This lemma, along with \eqref{eq:relaspe} and \eqref{eq:relaspe_2}, allows us to conclude that
\begin{equation*}
    \sigma_p(\td)\backslash \{-1\} = \sigma_f(\td)\,.
\end{equation*}

With all the above arguments, we actually have proved our second main  result of this subsection.
\begin{theorem} \label{thm:mainspec}
The spectrum $\sigma(\td)$ is a disjoint union of essential spectrum and eigenvalues of finite type, i.e.,
\begin{equation*}
    \sigma (\td) = \sigma_{ess} (\td)  \bigcup \sigma_f(\td)\,,
\end{equation*}
where $\sigma_{ess}(\td)$ and $\sigma_f(\td)$ are given by
\begin{equation*}
    \sigma_{ess}(\td) = \{-1,0,-\frac{1}{2}\}\,,\quad  \sigma_{f}(\td) =   \sigma_p(\td)\backslash \{-1\}\subset\{\lad \in \C\,;\ \im \lad > 0\}\,,
\end{equation*}
and $\sigma_{ess}(\td)$ \zz{gives all} the possible accumulation points of $\sigma_f(\td)$. Furthermore, $(\lad - \td)^{-1}$ is a meromorphic function on $\C\backslash\sigma_{ess}(\td)$ with a discrete set of poles given by $\sigma_f(\td)$.  
\end{theorem}

\begin{remark} \label{rem:nonradi_1}
This remark is to emphasize the special roles of eigenvalue $-1$ and its eigenspace, and to connect it with the nonradiating sources. We have observed in Corollary \ref{cor:eigespe-1} that $H_0(\ccurl 0,D)$ is a $\td$-invariant subspace with $\sigma(\td|_{H_0(\ccurl 0,D)}) = \{-1\}$, which can also be obtained by a direct calculation as in \eqref{eq:invanah}. In fact, we have 
\begin{equation*}
    \td[\vp] = \ccurl \kd[\ccurl \vp] - \ccurl \A_{\p D}^k[\n \t \vp]-\vp \chi_D\q \text{for}\ \vp \in H(\ccurl ,D)\,.
\end{equation*}
Hence, the space $H_0(\ccurl0,D)$ also corresponds to the nonradiating sources in the sense that $\td[\vp]$ for $\vp \in H_0(\ccurl0,D)$ vanishes in the far field since $\td[\vp] = - \vp \chi_D$. A more general version of this fact has actually been included in the proof of Theorem \ref{prop:diseig} implicitly. We have proved therein that if $u$ is the eigenfunction of $\td$ with eigenvalue $-1$, then $\td[u]$ has a vanishing far-field pattern. We refer the readers to \cite{bao2019stability} for the detailed characterization of nonradiating sources for Maxwell's equations in the homogeneous space.
\end{remark}

\subsection{Pole-pencil decomposition}  \label{subsec:ppdec}
To fully understand the structure of $(\lad - \td)^{-1}$, we may need to perform the full expansion of a vector field with respect to eigenfunctions and generalized eigenfunctions of $\td$ as the one given in \cite{ammari2015super} for the Helmholtz equation. Nevertheless, such a full expansion does not work here since we do not know whether the set of eigenfunctions and generalized eigenfunctions is complete in the space $L^2(D,\R^3)$. To circumvent this technical barrier, we develop 
a new pole-pencil decomposition (local expansion) in this subsection for the resolvent $(\lad - \td)^{-1}$ near
the reciprocal of the contrast $\tau$ instead, which relies on the concept of eigenvalues of finite type and Theorem \ref{thm:mainspec}. 

For our purpose, we define an $\ep$-neighborhood of $\tau^{-1}$ in $\sigma(\td)$:
\begin{equation} \label{eq:localneiandpole}
\sigma := B(\tau^{-1},\ep) \cap \sigma(\td)\,,
\end{equation}
where $\ep$ is a given small enough constant. By the fact from Theorem \ref{thm:mainspec} that $\sigma_f(\td)$ is discrete, we readily see that $\sigma$ must be a finite set of eigenvalues of finite type of $\td$, i.e.,
\begin{equation*}
    \sigma =  \cup_{i \in I}\{\lad_i\} = \{\lad_i\,;\ \lad_i \in B(\tau^{-1},\ep) \cap \sigma_f(\td) \} \,,
\end{equation*}
where $I \subset \NN$ is a finite index set. Without loss of generality, we assume that $\sigma$ is a nonempty set. In view of the facts that $\na H_0^1(D)$ is an invariant space of $\td$ and $\sigma(\td|_{\na H_0^1(D)}) = \{-1\}$ is disjoint from $\sigma$, it suffices to consider the resolvent of the restriction of $\td$ on $H(\ddiv0,D)$ to derive the pole-pencil decomposition of $(\lad -\td)^{-1}$. In the remainder of this subsection, we simply denote $\td|_{H(\ddiv0,D)}$ by $\w{T}_D^k$. To proceed, we first note from \eqref{eq:matrixtd} and Lemma \ref{lem:speccomp} that Theorem \ref{thm:mainspec} still holds with $\td$ replaced by $\w{T}_D^k$ except $$\sigma_{ess}(\w{T}_D^k) = \{0,-1/2\}\q \text{and}\q \sigma_{f}(\w{T}_D^k) = \sigma_p(\w{T}_D^k)\,.$$ 
It follows that both $\sigma$ and its complement $\zeta : = \sigma(\w{T}_D^k)\backslash \sigma$ are closed subsets of $\sigma(\w{T}_D^k)$, which allows us to choose a Cauchy contour $\Gamma$ in $\rho(\w{T}_D^k)$ around $\sigma$ separating $\sigma$ from $\zeta$, and define the Riesz projection corresponding to $\sigma$: 
\begin{equation} \label{def:rp}
    P_\sigma := \frac{1}{2\pi i}\int_{\Gamma}(\lad - \w{T}_D^k)^{-1}d \lad = \sum_{i \in I}P_{\lad_i}.
\end{equation}
The Riesz projection corresponding to $\zeta$ can be introduced similarly. By Lemma \ref{lem:fpriproj}, $H(\ddiv0,D)$ can be decomposed into two invariant subspaces of $\w{T}_D^k$ (and also $\td$):
\begin{equation} \label{eq:specinvadecomp}
    H(\ddiv0,D) = {\rm im} P_\sigma \oplus {\rm ker} P_\sigma,
\end{equation}
with ${\rm ker} P_\sigma = {\rm im} P_\zeta$, and it holds that 
\begin{equation*}
    \sigma(\td|_{{\rm im} P_\sigma}) = \sigma = \cup_{i \in I}\{\lad_i\},\quad   \sigma(\td|_{{\rm ker} P_\sigma}) = \sigma(\w{T}_D^k)\backslash \cup_{i \in I}\{\lad_i\}\,.
\end{equation*}
This decomposition \eqref{eq:specinvadecomp}, along with the Helmholtz decomposition \eqref{eq:helmdecomp}, gives us the following $\td$-invariant subspace decomposition of $L^2$-vector fields:
\begin{equation*}
    L^2(D,\R^3) = \nabla H_0^1(D,\R^3) \oplus_\perp ({\rm im} P_\sigma \oplus {\rm im} P_\zeta)\,.
\end{equation*}
On the associated product space: $\nabla H_0^1(D,\R^3) \times {\rm im} P_\sigma  \times {\rm im} P_\zeta$, the operator $\lad - \td$ with $\lad \in \C$ has a diagonal representation: $\diag(\lad + 1, \lad - T^k_\sigma , \lad - T_\zeta^k)$, where $T^k_\sigma$ and $T_\zeta^k$ are shorthand notations of $\td|_{{\rm im} P_\sigma}$ and $\td|_{{\rm im} P_\zeta}$ respectively. With the help of these notations, we arrive at the following representation of the solution to $(\lad - \td)[\vp] = f$ for $f \in L^2(D,\R^3)$ and $\lad \in B(\tau^{-1},\ep)\backslash \sigma$:
\begin{equation} \label{eq:prepolepen}
    \vp = \frac{1}{\lad + 1}\mathbb{P}_0 f + (\lad - T^k_\sigma)^{-1}P_\sigma f + (\lad - T_\zeta^k)^{-1}P_\zeta f\,.
\end{equation}

To further understand the behavior of $(\lad - \td)^{-1}$ locally, we recall 
from the definitions of $\sigma$ and $P_\sigma$ that ${\rm im} P_\sigma$ is of finite-dimensional and $\td|_{{\rm im} P_\sigma}$ is an operator acting on a finite-dimensional vector space with 
eigenvalues $\{\lad_i\}_{i \in I}$.
By the Jordan theory to the finite-dimensional linear operator, there exists a basis such that the matrix representation of $\td|_{{\rm im} P_\sigma}$ has a Jordan canonical form, that is, the representation matrix is a block diagonal one consisting of elementary Jordan blocks:
$$
J = \left[\begin{array}{cccc}
\lad &1 &  &  \\
 &\lad & \ddots  & \\
& &\ddots &1\\
& &  &\lad \\
\end{array}\right].
$$
More precisely, suppose that $\lad_i$ has geometric multiplicity $N_i$, and then the associated Jordan matrix $J_{\lad_i}$ will have the form:
$J_{\lad_i} = \diag(J^1_{\lad_i},\cdots,J^{N_i}_{\lad_i})$, where $J^j_{\lad^i},\ 1\le j \le N_i$ are the elementary Jordan blocks. Suppose also that for each Jordan block $J^j_{\lad^i}$, there is a Jordan chain ${ \boldsymbol \vp^j_{\lad_i}} := (\vp^{j,0}_{\lad_i},\vp^{j,1}_{\lad_i},\cdots,\vp^{j,n_{ij}-1}_{\lad_i})$, $\mathbb{N} \ni n_{ij}\ge 1$, an ordered collection of linearly independent generalized eigenfunctions, such that $J^j_{\lad^i}$ is the representation matrix of $\td$ restricted on $E^j_{\lad_i}$:
\begin{equation*} 
    \td|_{E^j_{\lad_i}}\bvi = \bvi J^j_{\lad_i},
\end{equation*}
where $E^j_{\lad_i}$ is the invariant subspace of $\td$ spanned by the Jordan chain $\bvi$. Without loss of generality, we assume $\norm{\vp_{\sss\lad_i}^{\sss j,s}}_{L^2(D)} = 1$, for $i\in I$, $1\le j \le N_i$, $0 \le s \le n_{ij} - 1$ in the rest of the exposition.  With $E^j_{\lad_i}$, we can write the following invariant subspace decomposition of ${\rm im} P_\sigma$:
\begin{equation*}
    {\rm im} P_\sigma = \oplus_{i \in I} \oplus^{N_i}_{j = 1} E^j_{\lad_i}.
\end{equation*}
In our notation, the eigenspace corresponding to $\lad_i$ is spanned by $\{\vp^{j,0}_{\lad_i}\}_{j=1}^{N_i}$ with dimension $N_i$ while the generalized eigenspace is given by $\oplus^{N_i}_{j=1}E^j_{\lad_i}$ with dimension $\sum^{N_i}_{j=1} n_{ij}$ (the algebraic multiplicity of $\lad_i$).
For vector $\vp \in E^j_{\lad_i}$, denote by $(\vp)_{\sss\bvi} = ((\vp)_{\sss\bvi}(0),(\vp)_{\sss\bvi}(1),\cdots(\vp)_{\sss\bvi}(n_{ij}-1)) \in \R^{n_{ij}}$ the coefficients in the expansion of $\vp$ with respect to the basis  $ \{\vp_{\lad_i}^{j,s}\}_{s=0}^{n_{ij}-1}$, i.e., 
\begin{equation} \label{eq:expbasis}
    \vp = \bvi \dd (\vp)_{\sss \bvi}:= \sum^{n_{ij}-1}_{k=0} (\vp)_{\sss\bvi}(k) \vp^{j,k}_{\lad_i}\,.
\end{equation}

With the help of these notions and \eqref{eq:prepolepen}, we arrive at the pole pencil decomposition of $(\lad - \td)^{-1}$.
\begin{theorem}  \label{thm:ppe}
The resolvent $(\lad - \td)^{-1}$ on $ B(\tau^{-1},\ep)\backslash \sigma$ has the following pole pencil decomposition:
\begin{equation} \label{for:poledecom}
    (\lad - \td)^{-1}[\dd]= \frac{1}{\lad + 1}\mathbb{P}_0[\dd] + \sum_{i \in I}\sum^{N_i}_{j = 1} \bvi \dd (\lad - J^j_{\lad_i})^{-1}(P^j_{\lad_i} [\dd])_{\sss\bvi} + (\lad - T_\zeta^k)^{-1}P_\zeta [\dd].
\end{equation}
Here, $P^j_{\lad_i}:=P^j_i P_{\lad_i}$ is the composition of projections $P_i^j$ and $P_{\lad_i}$, where $P_i^j(i \in I, 1\le j \le N_i )$ are finite-dimensional projections from ${\rm im}P_{\lad_i}$ to $E_{\lad_i^j}$. 
\end{theorem}
By the above theorem, we clearly see that the behavior of $(\lad - \td)^{-1}$ is essentially determined by its principal part: $\sum_{i \in I}\sum^{N_i}_{j = 1} \bvi \dd (\lad - J^j_{\lad_i})^{-1}(P^j_{\lad_i} [\dd])_{\sss\bvi}$ in the sense that it contains all the singularity of $(\lad - \td)^{-1}$ on $B(\tau^{-1},\ep)$ while the remainder term $(\lad + 1)^{-1}\mathbb{P}_0 + (\lad - T_\zeta^k)^{-1}P_\zeta$ is an analytic operator function  on $B(\tau^{-1},\ep)$. In fact, if $\sigma$ has only one element $\lad_i$, the principal part here exactly matches the one in the Laurent series of $(\lad - \td)^{-1}$ \eqref{eq:lauexp} near the pole $\lad_i$:
\begin{equation} \label{eq:singmatch}
    \sum^{N_i}_{j = 1} \bvi \dd (\lad - J^j_{\lad_i})^{-1}(P^j_{\lad_i} [\dd])_{\sss\bvi} = \sum_{n = -q(\lad_i)}^{-1} (\lad-\lad_i)^n T_n\,.
\end{equation}
We also note that $(\lad - J^j_{\lad_i})^{-1}$ has the following explicit form:
$$
(\lad - J^j_{\lad_i})^{-1} = \left[\begin{array}{cccc} 
(\lad-\lad_i)^{-1} &(\lad-\lad_i)^{-2} & \cdots  & (\lad-\lad_i)^{-n_{ij}} \\
 &(\lad-\lad_i)^{-1} & \ddots  & \vdots\\
& &\ddots &(\lad-\lad_i)^{-2}\\
& &  &(\lad-\lad_i)^{-1} \\
\end{array}\right],
$$
which readily gives us that the order $q(\lad_i)$ of the pole $\lad_i$ is determined by
\begin{equation} \label{eq:mindeg}
    q(\lad_i) = \max_{1\le j \le N_i}n_{ij}\,.
\end{equation}
Hence, we may expect that there is a blow-up of $(\lad - \td)^{-1}$ near the pole $\lad_i$ with order of $1/|\lad-\lad_i|^{q(\lad_i)}$. In fact, we have the following local resolvent estimate (see Theorem \ref{thm:localesttd}) directly from \eqref{eq:lauexp} and the estimate for $\norm{(\lad-J^j_{\lad_i})^{-1}}$: 
\begin{equation} \label{eq:estjordan}
    \norm{(\lad-J^j_{\lad_i})^{-1}} \le C \frac{1}{|\lad-\lad_i|^{n_{ij}}}\,, 
\end{equation}
where $\lad$ is in a small neighborhood of $\lad_i$ and $C$ is a generic constant depending on $n_{ij}$ and the aforementioned neighborhood of $\lad_i$. Note that we do not indicate the matrix norm that is used due to the norm equivalence property on a finite-dimensional space.
\begin{theorem} \label{thm:localesttd}
Suppose that $B(\tau^{-1}, \ep)$ and $\sigma$ are given as in \eqref{eq:localneiandpole}.  
There exists a constant depending on $\ep$ and the pole set $\sigma$ such that the following estimate holds for $f \in L^2(D,\R^3)$ and $\lad \in B(\tau^{-1},\ep)\backslash \sigma$,
\begin{equation*}
     \norm{(\lad- \td)^{-1}f}_{L^2(D)} \le C \sum_{i \in I}\frac{1}{|\lad- \lad_i|^{q(\lad_i)}}\norm{f}_{L^2(D)}\,,
\end{equation*}
where $q(\lad_i)$ is given by \eqref{eq:mindeg}.
\end{theorem}
This subsection ends with two remarks for a further discussion of the resolvent estimate of $\td$. 
\begin{remark}
In \cite{erxiong1994bounds}, the author gives the following bound for the smallest singular value of a 
$n \times n$ Jordan block $J$ with $\lad$ being its diagonal elements:
\begin{equation*}
    (\frac{n+1}{n})^n \frac{|\lad|^n}{n+1} \le \min_{1 \le j \le n}s_j(J) < \frac{|\lad|}{n} \q \text{for}\ 0 < |\lad| < \frac{n}{n+1}\,,
\end{equation*}
where ${s_j(A)}_{j=1}^n$ denote the singular values for a general $n\times n$ matrix $A$. The above estimate further gives us a sharper estimate for the induced $2$-norm of the resolvent of $J_{\lad_i}^j$ than \eqref{eq:estjordan}: 
 \begin{align*}
    \norm{(\lad - J^j_{\lad_i})^{-1}}_2 & = \max_{1\le j \le n_{ij}} s_j((\lad - J^j_{\lad_i})^{-1}) = \frac{1}{\min_{1\le j \le n_{ij}}s_j((\lad - J^j_{\lad_i}))} \le (\frac{n_{ij}}{n_{ij}+1})^{n_{ij}} \frac{n_{ij}+1}{|\lad-\lad_i|^{n_{ij}}}\,,
\end{align*}
when $0 < |\lad - \lad_j| \le n_{ij}/(n_{ij}+1)$. It allows us to derive new local resolvent estimate for $\td$:
\begin{equation*}
     \norm{(\lad- \td)^{-1}f}_{L^2(D)} \le  C\sum_{i \in I}\sum^{N_i}_{j = 1} \sqrt{n_{ij}} (\frac{n_{ij}}{n_{ij}+1})^{n_{ij}} \frac{n_{ij}+1}{|\lad-\lad_i|^{n_{ij}}} \norm{f}_{L^2(D)}\,,
\end{equation*}
for a generic constant $C$ and $\lad \in B(\tau^{-1}, \ep)$, which seems to be  a little bit shaper than the one in Theorem \ref{thm:localesttd} but actually does not provide us new information on the singularity of $(\lad-\td)^{-1}$ and its blow-up rate near the regular value $\tau^{-1}$.
\end{remark}

\begin{remark}
In general, it is very difficult to obtain a sharp global estimate for the resolvent $(\lad-\td)^{-1}$ of the non-selfadjoint and non-compact operator $\td$. \zz{Nevertheless, by} noting that $\td$ is a quasi-Hermitian operator, we can apply a general result to $\td$ to obtain its resolvent estimate. We put the detailed analysis and some relevant definitions in Appendix \ref{app:B}. 
\end{remark}

We have observed from Theorems \ref{prop:diseig} and \ref{thm:mainspec} that $\tau^{-1} - \td$ is invertible, and then Theorems \ref{thm:ppe} and \ref{thm:localesttd} permit us to write
\begin{equation} \label{eq:ppres}
    (\tau^{-1}-\td)^{-1} \sim \sum_{i \in I}\sum^{N_i}_{j = 1} \bvi \dd (\tau^{-1} - J^j_{\lad_i})^{-1}(P^j_{\lad_i} [\dd])_{\sss\bvi}\,,
\end{equation}
and to see that the behavior of $(\tau^{-1}-\td)^{-1}$ is indeed significantly influenced by the poles of resolvent of $\td$ near $\tau^{-1}$ and their associated eigenstructures, \zz{as it is suggested at the end of Section \ref{sec:iescp}.}

\if \commentflag = \ct
this is for the proof the resolvent (no use) 
\begin{proof}
Since $(\lad + 1)^{-1}\mathbb{P}_0 + (\lad - T_\zeta^k)^{-1}P_\zeta$ is an analytic function on $\bar{B}(\tau^{-1},\ep)$, we have find a constant $C_\ep$ depending on $\ep$ such that for $f \in L^2(D,\R^3)$,
\begin{equation*}
    \norm{(\lad + 1)^{-1}\mathbb{P}_0 f + (\lad - T_\zeta^k)^{-1}P_\zeta f}_{L^2(D)} \le C_\ep \norm{f}_{L^2(D)}\,,  
\end{equation*}
which, along with \eqref{for:poledecom}, imply that
\begin{equation} \label{inpr:local_1}
    \norm{(\lad- \td)^{-1}f}_{L^2(D)} \le C_\ep \norm{f}_{L^2(D)} + \sum_{i \in I}\sum^{N_i}_{j = 1} \norm{\bvi \dd (\lad - J^j_{\lad_i})^{-1}(P_{\lad_i} f)_{\sss\bvi}}_{L^2(D)}\,.
\end{equation}
By Lemma \ref{lem:localjor}, we  have 
\begin{align}
    \norm{(\lad - J^j_{\lad_i})^{-1}}_2 & = \max_{1\le j \le n_{ij}} s_j((\lad - J^j_{\lad_i})^{-1}) = \frac{1}{\min_{1\le j \le n_{ij}}s_j((\lad - J^j_{\lad_i}))} \notag\\ 
    & \le (\frac{n_{ij}}{n_{ij}+1})^{n_{ij}} \frac{n_{ij}+1}{|\lad-\lad_i|^{n_{ij}}}\,. \label{inpr:local_2}
\end{align}
Therefore, for each $i \in I$, the following estimate holds, 
\begin{align}
     &\sum^{N_i}_{j = 1} \norm{\bvi \dd (\lad - J^j_{\lad_i})^{-1}(P_{\lad_i} f)_{\sss\bvi}}_{L^2(D)} \notag \\
     \le & \sum^{N_i}_{j = 1} \sqrt{n_{ij}} \norm{(\lad - J^j_{\lad_i})^{-1}}_2\norm{(P_{\lad_i} f)_{\sss\bvi}}_{2} \notag \\
     \le & C\sum^{N_i}_{j = 1} \sqrt{n_{ij}} (\frac{n_{ij}}{n_{ij}+1})^{n_{ij}} \frac{n_{ij}+1}{|\lad-\lad_i|^{n_{ij}}} \norm{f}_{L^2(D)}\,,\label{inpr:local_3}
\end{align}
where we use the norm equivalence property on a finite- dimensional space and estimate \eqref{inpr:local_2}. By \eqref{inpr:local_1} and \eqref{inpr:local_3}, it follows that
\begin{align*}
    \norm{(\lad- \td)^{-1}f}_{L^2(D)} \le C_\ep \norm{f}_{L^2(D)} + C\sum_{i \in I}\sum^{N_i}_{j = 1} \sqrt{n_{ij}} (\frac{n_{ij}}{n_{ij}+1})^{n_{ij}} \frac{n_{ij}+1}{|\lad-\lad_i|^{n_{ij}}} \norm{f}_{L^2(D)}\,.
\end{align*}
\end{proof}
\fi

\if \commentflag = \ct 
\textbf{Resolvent estimate for $\td$.} 
Our final goal of this section is to derive the resolvent estimate for $\td$ undee the restriction that $|\lad|\ll 1$. By Theorem \ref{thm:mainspec} and norm equivalence \eqref{eq:normequ}, we have
\begin{equation*}
    \norm{(\lad - T_D^k)^{-1}}_{B(L^2)} \approx \norm{(\lad - \TT_D^k)^{-1}}_{B(\xx)}.
\end{equation*}
Recalling \eqref{eq:similar}, we can further derive
\begin{align}
\norm{(\lad - \TT_D^k)^{-1}[\FF]}_{B(\xx)} & = \frac{1}{1 + \lad} \norm{f_1}_{H_0^1} + \norm{(\lad -  \w{\TT}^k_D)^{-1} \w{\FF}}_{\w{X}_0 \times H_0^{-\frac{1}{2}}(\p D)} 
\notag \\ & \le \norm{f_1}_{H_0^1} + C_\P \norm{\P (\lad - \w{\TT}_D^k)^{-1}\P^{-1}\P \w{\FF}}_{\w{X}_0 \times H_0^{-\frac{1}{2}}(\p D)} \notag \\
& \le \norm{f_1}_{H_0^1} + C_\P \norm{(\lad - \T_D^k)^{-1}\P \w{\FF}}_{\w{X}_0 \times H_0^{-\frac{1}{2}}(\p D)} \notag \\
& \le C_\P\norm{(\lambda - k^2 \T \kd \T^{-1})^{-1}}_{B(\w{X}_0)} \norm{\FF}_{\xx}, \label{esti:prioriess}
\end{align}
where $C_\P$ is a generic constant depending on $\P$, and $\w{\FF}:=(f_2,f_3)\in \w{X}_0 \times H_0^{-\frac{1}{2}}(\p D)$ is the last two components associated with $\FF \in \xx$.  The last inequality in \eqref{esti:prioriess} holds since we are considering the high contrast regime, which allow us to assume $\norm{(\lambda - k^2 \T \kd \T^{-1})^{-1}}_{B(\w{X}_0)}$ is much larger than $\norm{(\lad+\frac{1}{2}-\md)^{-1}}_{B(H_0^{-\frac{1}{2}}(\p D))}$. Therefore, we only need to estimate the resolvent $(\lad - \T k^2 \kd \T^{-1})^{-1}$ on $\w{X}_0$ to conclude our final result, Theorem \ref{thm:resolventoftd}. To this purpose, we recall the definition of Hilbert-Schmidt operator and a corresponding resolvent estimate. Given a separable Hilbert space $H$, a linear operator $K \in B(H)$ is Hilbert-Schmidt if for any orthonormal basis $\{e_n\}^\infty_{n=1}$ of $H$, we have 
\begin{equation*}
\sum_{n = 1}^\infty \norm{Ke_n}^2_{H} < \infty .
\end{equation*}
Define Hilbert-Schmidt norm by $\norm{K}^2_{HS} = \sum^\infty_{n = 1} \norm{Ke_n}^2_{H} < \infty$, which can be shown to be independent of the choice of orthonormal basis. In particular, if $H = L^2(D)$, then for any Hilbert-Schmidt operator $K$, we can construct a Kernel $K(x,y)\in L^2(D \times D)$ such that 
\begin{equation*}
    K[f](x) = \int_D K(x,y)f(y)dy, \ \text{for} \ f\in L^2(D).
\end{equation*}
We then have an equivalent definition of Hilbert-Schmidt norm:
\begin{equation*}
    \norm{K}^2_{HS} = \int_D \int_D |K(x,y)|^2 dx dy.
\end{equation*}
Moreover, for the resolvent $(\lad-K)^{-1}$, the following norm estimate holds. 
\begin{lemma} \label{lem:generalesti}
For Hilbert-Schmidt operator $K$, we have following resolvent estimate,
\begin{equation*}
    \norm{(\lad - K)^{-1}}_{H} \le \frac{1}{d(\lad,\sigma(K))}exp(\frac{1}{2} \frac{\norm{K}^2_{HS}}{d(\lad,\sigma(K))^2}+\frac{1}{2}),
\end{equation*}
where $d(\lad,\sigma(K))$ is defined by 
\begin{equation*}
    d(\lad,\sigma(K)):= \inf_{z \in \sigma(K)}|\lad - z|
\end{equation*}
\end{lemma}
We are now prepared to give the result on the estimate of $\norm{(\lambda - k^2 \T \kd \T^{-1})^{-1}}_{B(\w{X}_0)}$.
\begin{lemma}
$\T \kd \T^{-1}$ is a Hilbert-Schmidt operator on $\w{X}_0$ with the following resolvent estimate:
\begin{equation*}
    \norm{(\lambda - \T \kd \T^{-1})^{-1}}_{B(\w{X}_0)} \le \frac{1}{d(\lad,\sigma(\T \kd \T^{-1}))}exp( \frac{diam(D)|D|}{8\pi d(\lad,\sigma(\T \kd \T^{-1}))^2}+\frac{1}{2}),
\end{equation*}
where $diam(D)$ is the diameter of $D$, given by $\sup_{x,y \in D}|x-y|$, and $|D|$ is the volume of $D$.
\end{lemma}

\begin{proof}
We first note $\kd$ is a Hilbert-Schmidt operator with norm:
\begin{equation} \label{auxeq:hskdnorm}
    \norm{\kd}^2_{HS} = \frac{1}{16 \pi^2} \int_D \int_D \frac{1}{|x-y|^2}dxdy \le \frac{1}{4 \pi}diam(D) |D|.
\end{equation}
Consider Hilbert space $\w{X}_0$ and choose a orthonormal basis $\{e_n\}^\infty_{n=1}$.  By the orthogonal relation: $$(\ccurl e_i,\ccurl e_j)_{L^2(D)} = \d_{ij},$$ it is easy to observe that the orthonormal set $\{\ccurl e_n\}^\infty_{n=1}$ with respect to the $L^2$ -inner product can be extended to an orthonormal basis in $L^2(D,\R^3)$, which, combined with \eqref{auxeq:hskdnorm}, leads us to
\begin{equation} \label{auxeq:hsnorm}
    \sum^{\infty}_{n=1}\norm{\kd \ccurl e_n}^2_{L^2(D)} \le \norm{\kd}^2_{HS}.
\end{equation}
Note that 
\begin{align*}
    \norm{\ccurl\T \kd \T^{-1} e_n}_{L^2(D)} = \norm{\ccurl\T \ccurl \kd \ccurl e_n}_{L^2(D)} \le \norm{\kd \ccurl e_n}_{L^2(D)},
\end{align*}
which directly gives us
\begin{equation*}
    \sum^\infty_{n=1}  \norm{\ccurl\T \kd \T^{-1} e_n}_{L^2(D)} \le \norm{\kd}^2_{HS}.
\end{equation*}
By the definition of Hilbert-Schmidt operator and \eqref{auxeq:hskdnorm}, we have $\T \kd \T^{-1}$ is Hilbert-Schmidt with
\begin{equation*}
    \norm{\T \kd \T^{-1}}_{HS} \le \frac{1}{4 \pi}diam(D) |D|.
\end{equation*}
Then, Lemma \ref{lem:generalesti} helps us to conclude. 
\end{proof}
As a corollary, we obtain the resolvent estimate for $\td$.
\begin{theorem} \label{thm:resolventoftd}
For $\td$ and $\lad \in \rho(\td)$, it holds that 
\begin{equation*}
    \norm{(\lad-\td)^{-1}} \le \frac{C_\P}{d(\lad,\sigma(\T \kd \T^{-1}))}exp( \frac{diam(D)|D|}{8\pi d(\lad,\sigma(\T \kd \T^{-1}))^2}+\frac{1}{2}).
\end{equation*}
\end{theorem}
\fi

\subsection{Spherical region} \label{subsec:sphere}
\zz{In view of the formula \eqref{eq:ppres}, both eigenvalues and eigenfunctions can play a crucial role in 
the local behavior of $(\lad - \td)^{-1}$ near the very small regular value $\tau^{-1}$, which motivates us to quantitatively investigate the asymptotic behaviors of eigenvalues and eigenfunctions of the operator $\td$ as $\lad \to 0$ to further explore the mechanism lying behind the super-resolution. In this subsection, we consider the spectral properties of $\td$ on the unit ball $D = B(0,1)$ in $\R^3$, where the Mie scattering theory is applicable.}

 We have seen in Lemma \ref{lem:eqtrass} that solving the eigenvalue equation $(\lad - \td)[\vp] = 0$ is equivalent to finding $\lad$ and the associated nontrivial radiating solution to the transmission problem:
\begin{equation} \label{eq:eqtrassrep}
    \curl \curl E - k^2 E = \frac{k^2}{\lad} E \chi_D.
\end{equation}
In this subsection, we assume $\lad \neq -1$ so that the wave number  $k_\lad = k\sqrt{1+\lad^{-1}}$ inside the domain will never vanish, see Remark \ref{rem:eig-1ball}, also Remark \ref{rem:nonradi_1} for a discussion of the case of $\lad = -1$. By the Mie theory, any solution $E$ of the time-harmonic Maxwell equations $\curl \curl E - k^2 E = 0$ in the far field can be represented in the following series form:
\begin{equation} \label{eq:repoute}
     E(x) = \sum^\infty_{n = 1} \sum^n_{m = -n}
    \gnm \ete(k,x) + \enm \etm(k,x)\,,
\end{equation}
where the complex coefficients $\gnm$ and $\enm$  are to be determined and $\ete$ and $\etm$ are vector wave functions defined in the Appendix \ref{app:C_1}. Similarly, any solution $E$ to the Maxwell equations $\curl \curl E - k_\lad^2 E = 0$ near $0$ has the following representation:
\begin{equation} \label{eq:repinte}
   E(x) = \sum^\infty_{n = 1} \sum^n_{m = -n}  \anm \wete(k_\lad, x) + \bnm \wetm(k_\lad, x)\,,
\end{equation}
with undetermined coefficients $\anm,\bnm \in \C$, see \eqref{eq:enete} and \eqref{eq:enetm} for the definitions of $\wete$ and $\wetm$. To establish the equations for eigenvalues $\lad$, we match the Cauchy data $(\h{x}\t E, \h{x} \t \curl E)$ of  \eqref{eq:repoute} and \eqref{eq:repinte} on the boundary $\p B(0,1)$. By the trace formulas of multipole fields \eqref{eq:trramp} and \eqref{eq:trenmp}, and recalling that $\{\unm\}$ and $\{\vnm\}$ is an orthonormal basis of $L^2_T(S^2)$, matching Cauchy data reduces the original eigenvalue problem to solving infinite linear systems:
\begin{equation*}
    [\h{x} \t E(x)] = 0 \eq 
   \left\{ \begin{array}{l}
   \gnm h_n^{(1)}(k ) = \anm j_n(k_\lad) \\
  \enm \hh_n(k ) = \bnm \frac{k}{k_\lad}    \jj_n(k_\lad )         
    \end{array} \right. n = 1,2,\cdots,\ m=-n,\cdots,n\,, 
\end{equation*}
and 
\begin{equation*}
       [\h{x} \t \curl E (x)] = 0 \eq 
   \left\{ \begin{array}{l}
   \gnm \hh_n(k) = \anm \jj_n(k_\lad) \\
  \enm kh^{(1)}_n(k) = \bnm k_\lad j_n(k_\lad ) 
    \end{array} \right. n = 1,2,\cdots,\ m=-n,\cdots,n\,,
\end{equation*}
which can be reformulated into the following independent equations with the undetermined coefficients as unknowns:
\begin{equation}\label{eq:coeigenmode1}
\mm 
j_n(k_\lad ) & -h_n^{(1)}(k ) \\ 
\jj_n(k_\lad) & -\hh_n(k) 
\nn 
 \mm \anm \\ \gnm \nn  = 0\,,\  n = 1,2,\cdots,\ m=-n,\cdots,n\,,
\end{equation}
and
\begin{equation}\label{eq:coeigenmode2}
\mm \frac{k}{k_\lad}\jj_n(k_\lad ) & -\hh_n(k ) \\ k_\lad j_n(k_\lad )  & -kh^{(1)}_n(k)  \nn \mm \bnm \\ \enm \nn = 0\,, \ n = 1,2,\cdots,\ m=-n,\cdots,n\,.
\end{equation}

\ifnum \commentflag = \ct
\begin{align*}
    & \frac{\curl E(x)}{i k } = \sum^\infty_{n = 1} \sum^n_{m = -n}  \anm \frac{\curl \wete(x)}{i k }  + \bnm \wete(x)\,,\forall \ x\in B(0,R) \\
    & \frac{\curl E(x)}{i k} = \sum^\infty_{n = 1} \sum^n_{m = -n}
    \gnm \frac{\curl \ete(x)}{i k} + \enm \ete(x)\,, \forall \ x \in \R^3\backslash \bar{B}(0,R)
\end{align*}
\fi
\ifnum \commentflag = \ct
We match the coefficients, with respect to the vector spherical harmonics, of tangential trace $\h{x} \t E(x)$ on $\p  B(0,R)$. 
\begin{align*}
   & \h{x} \t \ete(x) =  \sqrt{n(n+1)} h_n^{(1)}(k |x|)\unm(\h{x}) \\
   & \h{x} \t \etm(x) = - \frac{\sqrt{n(n+1)}}{ik |x|} \vnm(\h{x})\hh_n(w|x|) 
\end{align*}
\begin{align*}
    & \h{x} \t \wete (x) = j_n(k\sqrt{1+1/\lad} |x|) \sqrt{n(n+1)} \unm(\h{x})  \\
    & \hat{x} \t \wetm (x) = -\frac{\sqrt{n(n+1)}}{ik (1 + 1/\lad)|x|} \jj_n(k\sqrt{1+1/\lad}|x|)\vnm(\h{x})
\end{align*}
We match the coefficients, with respect to the vector spherical harmonics, of tangential trace $\h{x} \t H(x)$ on $\p  B(0,R)$. 
\begin{align*}
   & \h{x} \t \ete(x) =  \sqrt{n(n+1)} h_n^{(1)}(k |x|)\unm(\h{x}) \\
   & \h{x} \t \frac{\curl \ete}{i k} 
   =  \frac{\sqrt{n(n+1)}}{ik |x|} \vnm(\h{x})\hh_n(w|x|) 
\end{align*}
\begin{align*}
    & \h{x} \t \wete (x) = j_n(k\sqrt{1+1/\lad} |x|) \sqrt{n(n+1)} \unm(\h{x})  \\
    & \hat{x} \t \frac{\curl \wete}{ik} = \frac{\sqrt{n(n+1)}}{ik |x|} \jj_n(k\sqrt{1+1/\lad}|x|)\vnm(\h{x})
\end{align*}
We hence have
\begin{align*}
    \eta_{n,m}h_n^{(1)}(k R) = \beta_{n,m}j_n(k \sqrt{1+1/\lad}R) \\
    \gnm \hh_n(k R) = \anm \jj_n(k \sqrt{1+1/\lad}R)
\end{align*}
\fi 

We readily observe that the coefficient matrices in above linear systems do not depend on the index $m$, and the equation \eqref{eq:eqtrassrep} has nontrivial solutions for $\lad \in \sigma_p(\td)\backslash\{0\}$ if and only if \eqref{eq:coeigenmode1} or \eqref{eq:coeigenmode2} has nonzero solutions for some index $n \in \NN^+$, or equivalently, the determinants of the associated coefficient matrices are zero:
\begin{equation} \label{eq:coeff1}
 h_n^{(1)}(k) \jj_n(k_\lad ) -  j_n(k_\lad ) \hh_n(k ) = 0\,,
\end{equation}
or 
\begin{equation} \label{eq:coeff2}
\frac{k^2}{k_\lad^2} h_n^{(1)}(k) \jj_n(k_\lad) - j_n(k_\lad ) \hh_n(k) = 0\,.
\end{equation}
To proceed, let us focus on the first case, i.e., equations \eqref{eq:coeigenmode1} and \eqref{eq:coeff1}. We note from the fact that all the zeros of $j_n(z)(n \in \NN^+)$, expect the possible point $z = 0$, are simple \cite{watson1995treatise} that $j_n(k_\lad)$ and $\jj_n(k_\lad)$ cannot vanish simultaneously. Neither can $h_n^{(1)}(k)$ and $\hh_n(k)$ by a similar observation. Then all the nontrivial solutions of \eqref{eq:coeigenmode1} have the form: $(\anm,\gnm) = c_{n,m}(\alpha_n,\gamma_n)$ with $\alpha_n,\gamma_n \neq 0$ and $c_{n,m}\in \C\backslash \{0\}$. Therefore, for $\lad$ such that \eqref{eq:coeff1} holds for some index $n$, there is an associated subspace spanned by the eigenfunctions $\{\wete\}_{m=-n}^{m=n}$. If the same $\lad$ happens to satisfy \eqref{eq:coeigenmode1} for index $n' \neq n$ or \eqref{eq:coeigenmode2} for index $n''$, we can find another (sub)eigenspace spanned by $\{\w{E}^{TE}_{n',m}\}_{m=-n'}^{m=n'}$ or $\{\w{E}^{TM}_{n'',m}\}_{m=-n''}^{m=n''}$, which is orthogonal to the aforementioned one. Moreover, the geometric multiplicity of $\lad$ is the sum of the dimensions of these subspaces, which must be finite, since all the eigenvalues of $\td$ except $-1$ are  eigenvalues of finite type, see Theorem \ref{thm:mainspec}. The same arguments can be applied to the system \eqref{eq:coeigenmode2}, as well as the equation \eqref{eq:coeff2}.
We summarize the above facts in the following theorem. 
\begin{theorem}
Denote by $\sigma^1_n$ and $\sigma^2_n$ the sets of $\lad$ such that \eqref{eq:coeff1} and \eqref{eq:coeff2} holds respectively, then we have the set of eigenvalues  of finite type of $\td$ for a spherical region $B(0,1)$ is given by 
\begin{equation*}
    \sigma_f(\td) = \sigma_p(\td) \backslash\{-1\} = \cup_{n=1}^\infty (\sigma^1_n \cup \sigma^2_n)\,.
\end{equation*}
And for each $\lad \in \sigma_f(\td)$, the finite-dimensional eigenspace is spanned by 
\begin{equation*}
    \cup_{i = 1}^2\cup_{n \in \Lambda_i}\cup_{m=-n}^n \w{E}^i_{n,m}(k_\lad,x),
\end{equation*}
where $\Lambda_i$, $i=1,2$, is a finite subset of $\NN^+$ such that $\lad \in \sigma_n^i$ for $n \in \Lambda_i$. Here, $\w{E}^i_{n,m}(k_\lad,x)$ , $i = 1,2$, denote the eigenfunctions $\wete(k_\lad,x)$ and $\wetm(k_\lad,x)$, respectively. 
\end{theorem}

\begin{remark} \label{rem:eig-1ball}
As we have seen in Corollary \ref{cor:eigespe-1} and Remark \ref{rem:nonradi_1}, the eigenspace of eigenvalue $\lad = 1$ is given by $\na H_0^1(D)$, which are the nonradiating sources. For the case of the domain $B(0,1)$, it is spanned by the gradient of eigenfunctions $u_n$ of the Dirichlet Laplacian, that is,
\begin{equation*}
    \left\{
    \begin{array}{ll}
  \Delta u_{n} = -k_n^2 u_{n}  & \text{in} \  B(0,1)\,, \\
  u_n = 0  & \text{on} \ \p B(0,1)\,. \end{array}\right.
\end{equation*}
 The explicit formulas of the Dirichlet eigenvalues $k_n$ and eigenfunctions $u_n$ are available in \cite{folland1995introduction}. It is also worth mentioning that in the above argument, we have actually proved that all of these eigenfunctions: $\wete$ and $\wetm$ are the radiating sources, since both solution spaces of \eqref{eq:coeigenmode1} and \eqref{eq:coeigenmode2} are one-dimensional and  spanned by some vector $p \in \C^2$ with non-vanishing components $p_1,p_2$, i.e., $p_1,p_2 \neq 0$.
\end{remark}

\subsubsection{Asymptotic behavior of eigenvalues}
This subsection is devoted to the understanding of the distribution of eigenvalues in $\sigma_n^i$ for $i = 1,2$, namely, the eigenvalues of $\td$. For this purpose, it suffices to investigate the zeros of $f_n^i(z)$ for $i=1,2$ on $\C\backslash\{0\}$, where $f_n^i(z)$ are introduced by the right-hand side of \eqref{eq:coeff1} and \eqref{eq:coeff2} by setting $z = k_\lad$, i.e.,
\begin{align}
& f_n^1(z) =  h_n^{(1)}(k) \jj_n(z) -  j_n(z) \hh_n(k), \label{def:fn1}\\
& f_n^2(z) =  \frac{k^2}{z^2} h_n^{(1)}(k) \jj_n(z) - j_n(z) \hh_n(k). \label{def:fn2}
\end{align}
We readily see from the analyticity of $z^{-n}j_{n}(z)$ on $\C$ that $f_n^i$ for $i = 1,2$ and $n \in \NN^+$ are analytic on the whole complex plane $\C$, except that $f_1^2(z)$ is a meromorphic function on $\C$ with $0$ being its only simple pole. By the symmetry property of $j_n(z)$: $j_n(-z) = (-1)^nj_n(z)$ \cite{olver2010nist}, we have
\begin{equation*}
    \jj_n(-z) = j_n(-z)+ (-z)j_n'(-z) = (-1)^nj_n(z) + (-1)^n zj_n'(z) = (-1)^n \jj_n(z)\,,
\end{equation*}
which directly gives us the following lemma.
\begin{lemma} \label{lem:symf_12}
For $f_n^i(z)$, $n\in \NN^+,i=1,2$ defined by \eqref{def:fn1} and \eqref{def:fn2}, the following symmetry properties hold,
\begin{equation*}
     f_n^1(-z) = (-1)^n f_n^1(z)\,,\q f_n^2(-z) = (-1)^n f_n^2(z)\,.
\end{equation*}
\end{lemma}
As a consequence of Lemma \ref{lem:symf_12}, the zeros of $f_n^i$ are symmetric with respect to the origin. To obtain an intuition about the behavior of those zeros, we numerically compute the zeros of $f_n^i$ for $i =1,2$ and different values of $n$ in the right half plane $\{z\in \C\,;\  -\frac{\pi}{2} < \arg(z)\le \frac{\pi}{2}\}$, by Muller's method \cite{ammari2018mathematical}. As we can observe in Figure \ref{fig:zeros}, the zeros of $f_n^i(z)$ are complex and lie in the lower half-plane. This fact has been theoretically justified by Theorem \ref{prop:diseig}. Also the overall magnitudes of their imaginary parts rapidly decrease as the value of $n$ increases. Moreover, it is remarkable to note that for fixed $i$ and $n$, there is a sequence of zeros of $f_n^i$ tending exponentially fast to the real axis. It motivates us to investigate the asymptotic behavior of zeros of $f_n^i(z)$ as $|z| \to \infty$.

\begin{figure}[!htbp]
\includegraphics[clip,width=0.35\textwidth]{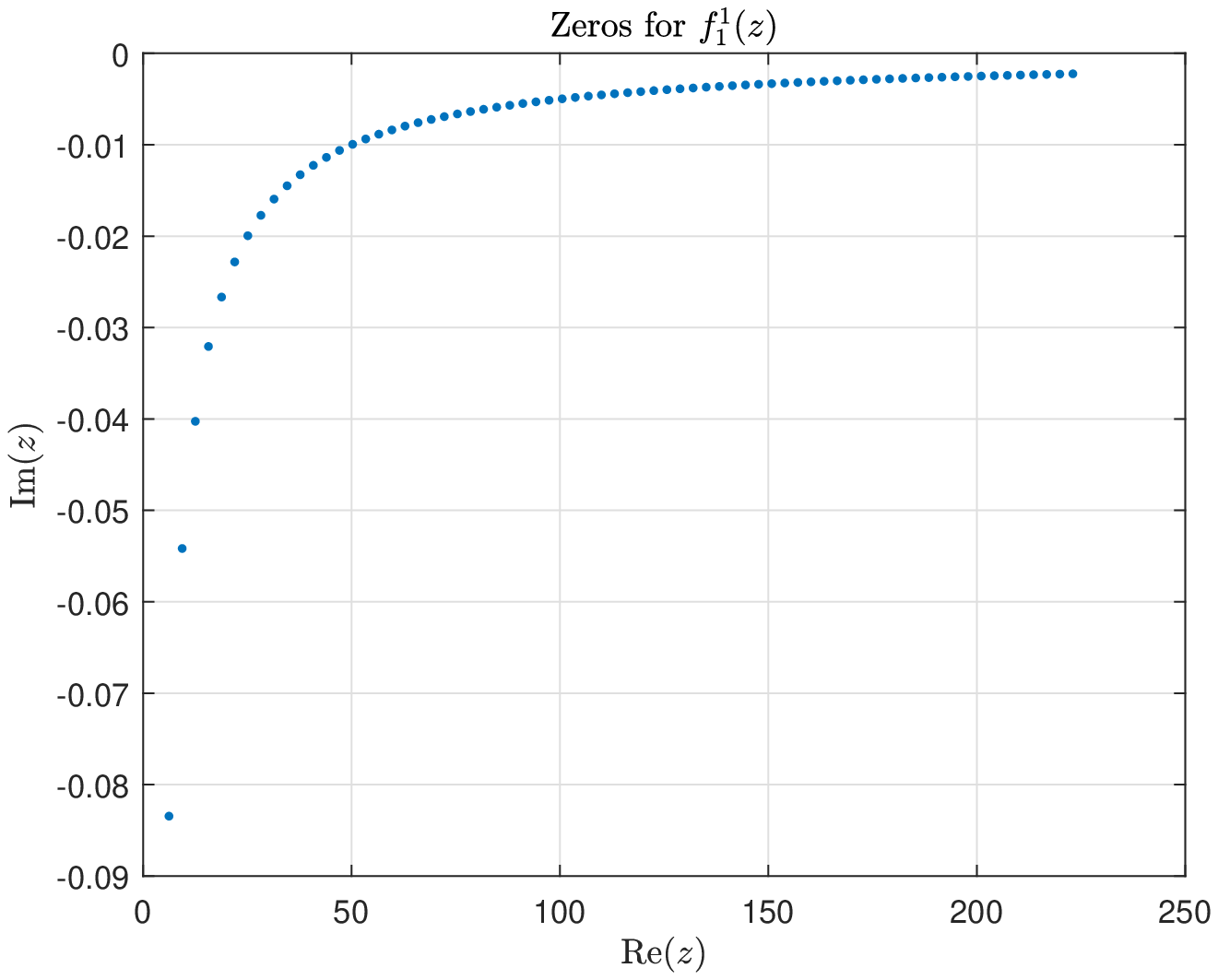} \hskip -0.6cm
\includegraphics[clip,width=0.35\textwidth]{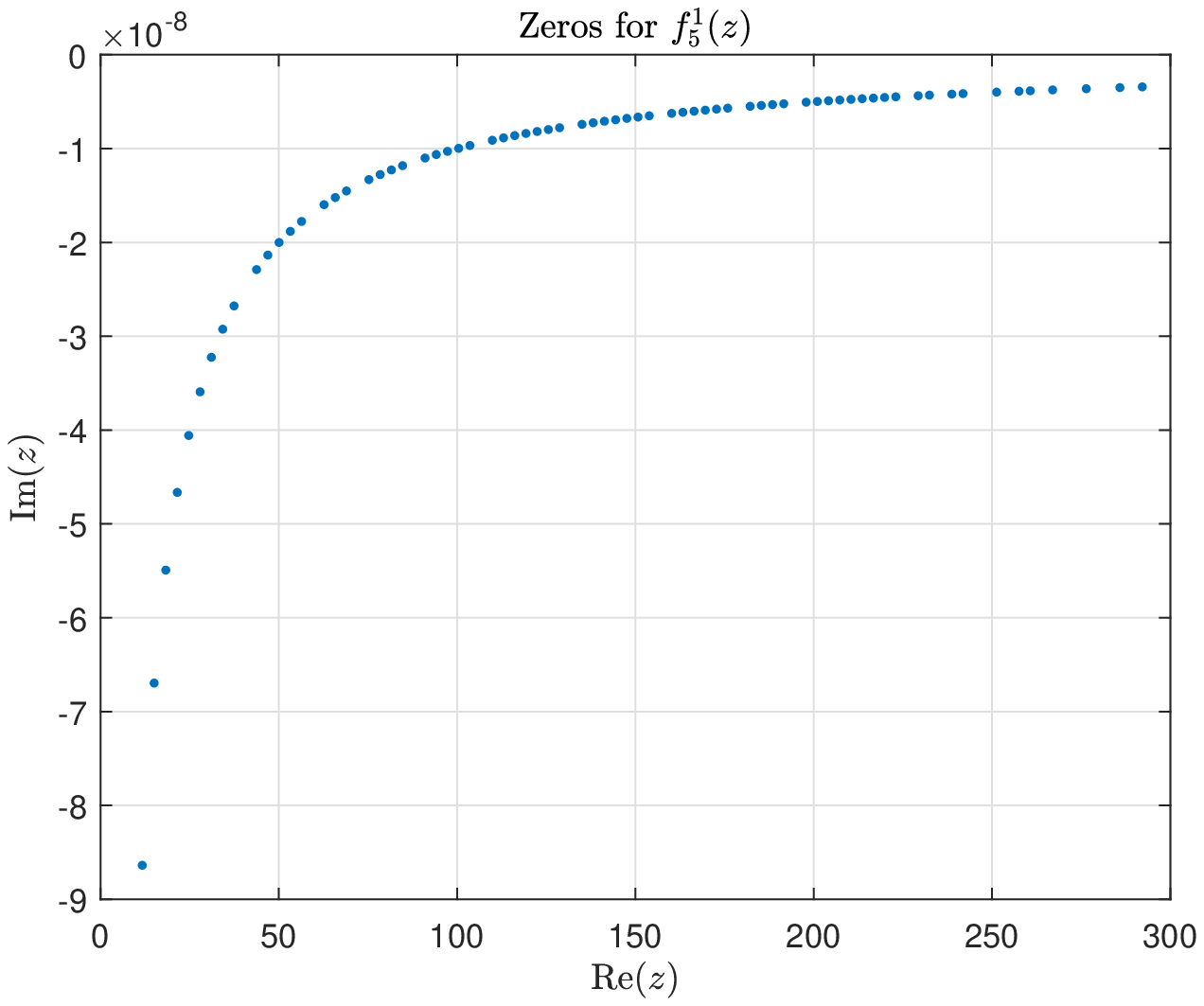} \hskip -0.6cm
\includegraphics[clip,width=0.35\textwidth]{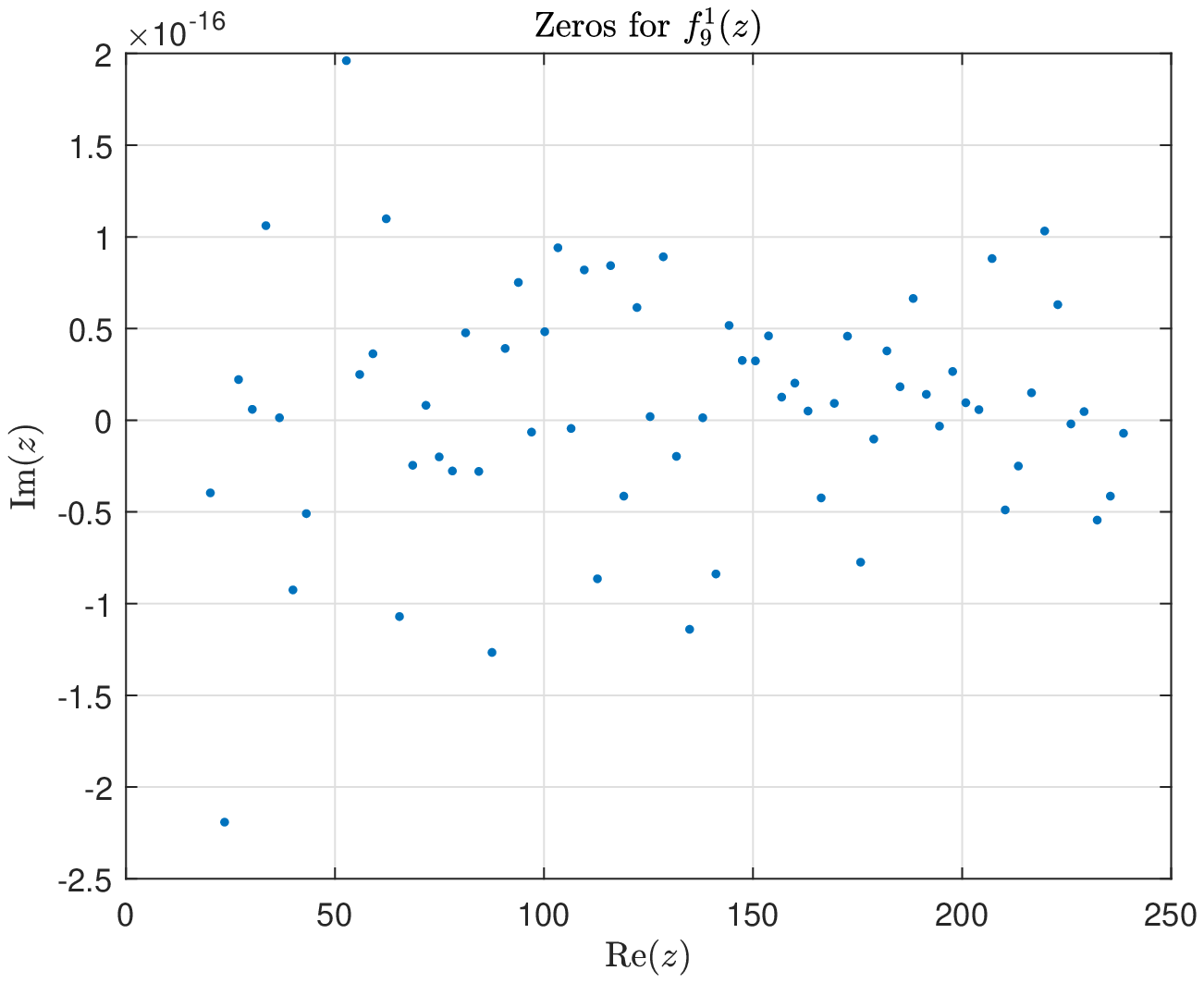} \\
\includegraphics[clip,width=0.35\textwidth]{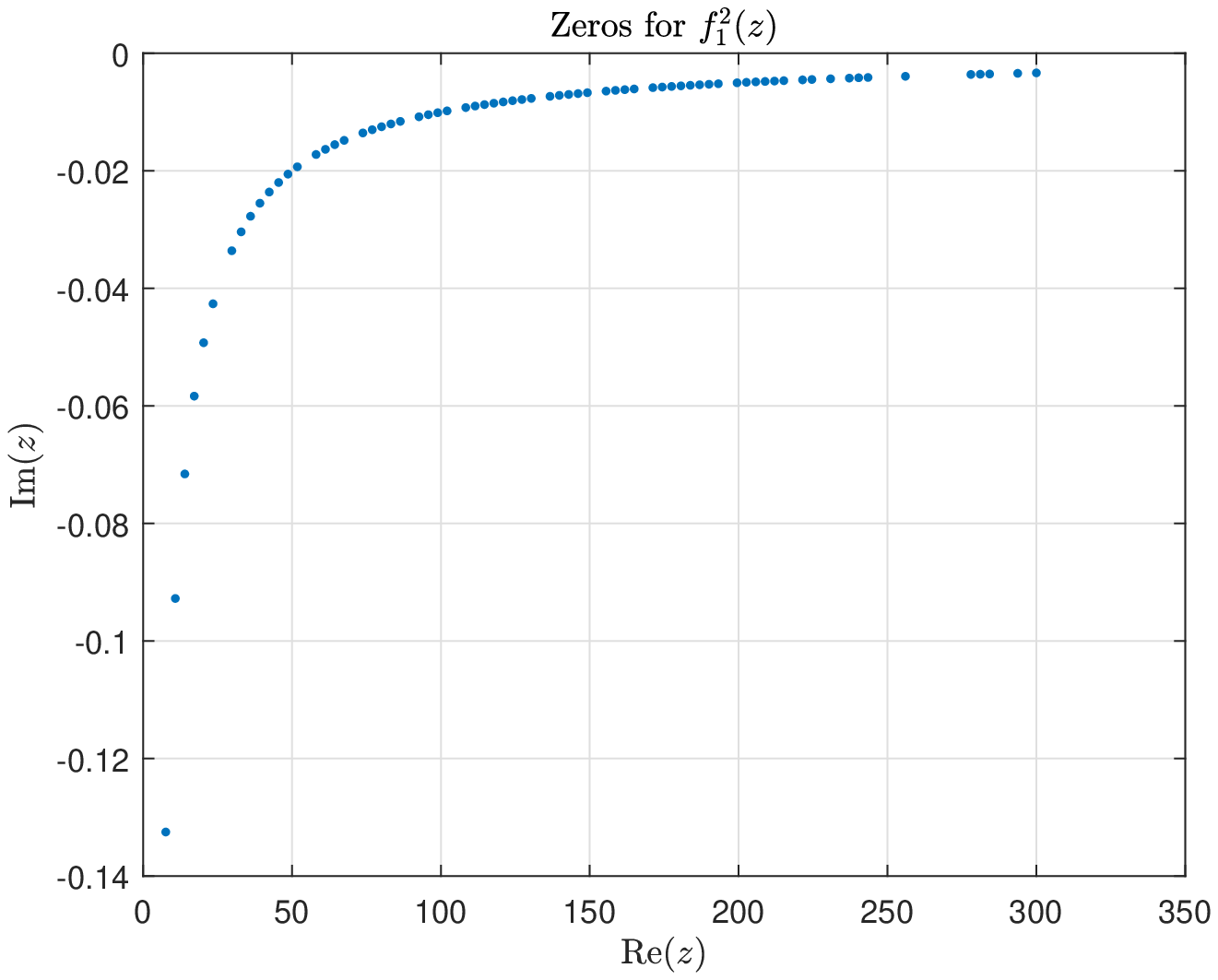} \hskip -0.6cm
\includegraphics[clip,width=0.35\textwidth]{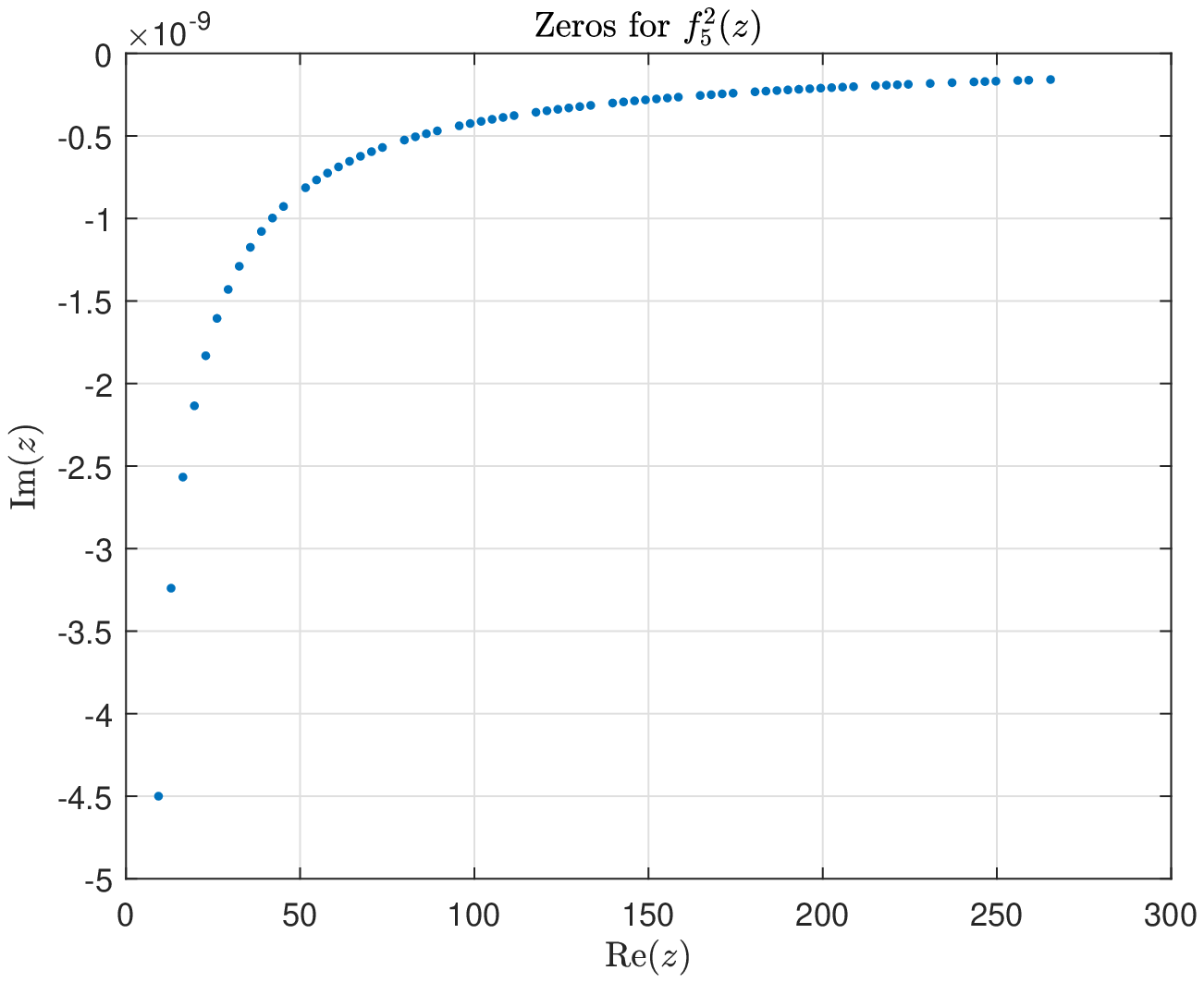} \hskip -0.6cm
\includegraphics[clip,width=0.35\textwidth]{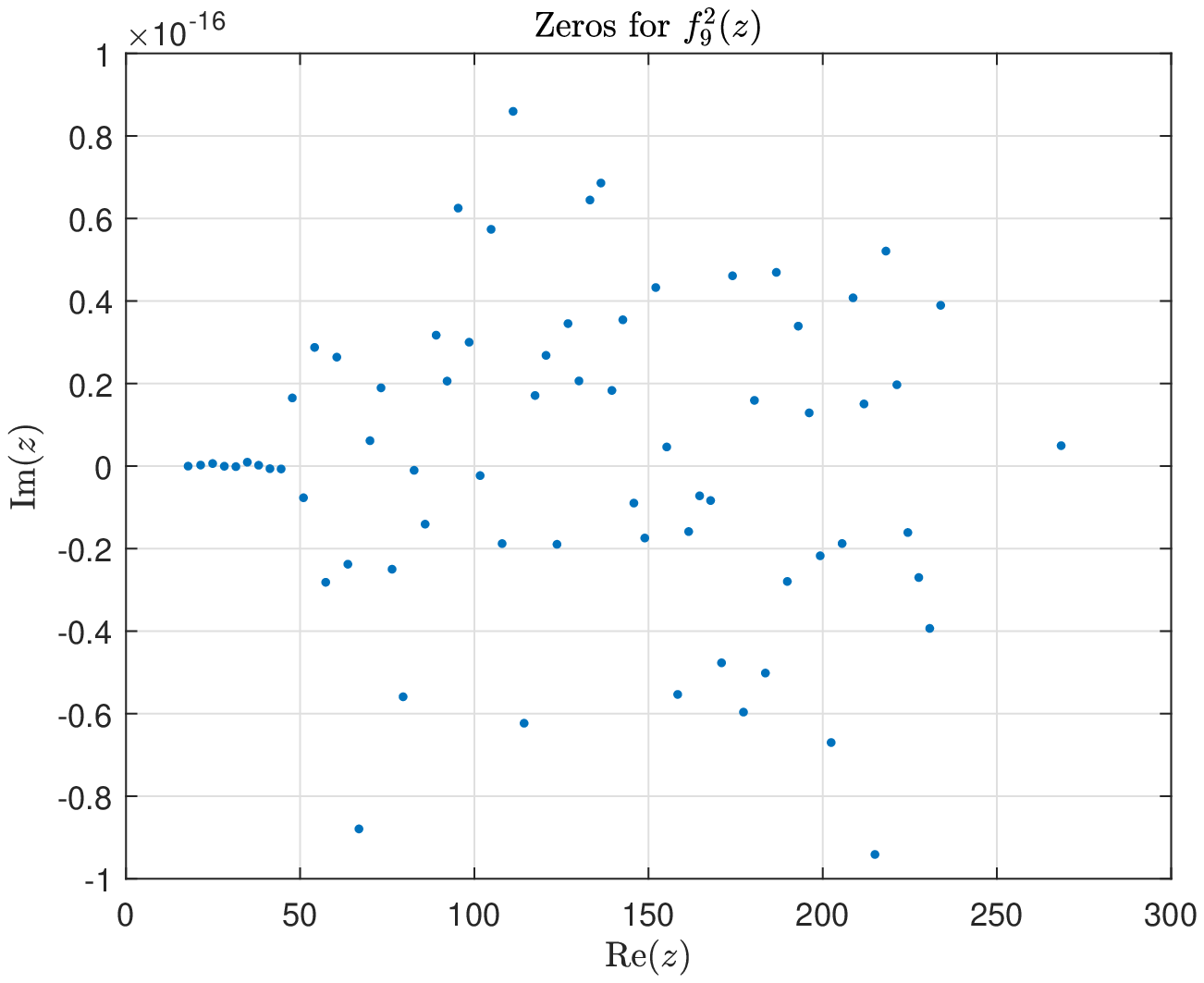}
\caption{\label{fig:zeros} 70 zeros of $f_n^i(z)$ for $i=1$ (the first row), $i = 2$ (the second row), and $n=1,5,9$ (from left to right) in the right half plane: $\{z\in \C\,;\  -\frac{\pi}{2} < \arg(z)\le \frac{\pi}{2}\}$.}
\end{figure}

For this, we first consider $f_n^1(z)$ and see the following asymptotics
from \eqref{eq:aymjn} and \eqref{eq:aymjjn} that for $|\arg(z)| < \pi$,  
\begin{align} \label{eq:asyf1}
    f_n^1(z) &= h_n^{(1)}(k) \cos\left(z - \frac{n\pi}{2}\right)- \frac{1}{z} \hh_n(k)\cos\left(z - \frac{n \pi}{2} - \frac{\pi}{2}\right) + e^{|\im z|} O\left(\frac{1}{|z|}\right)\notag \\
  & = h_n^{(1)}(k) \cos\left(z- \frac{n\pi}{2}\right) +  e^{|\im z|} O\left(\frac{1}{|z|}\right) \q \text{as}\   |z|\to \infty \,,
\end{align}
\zz{where we have also utilized the fact that both $h_n^{(1)}(z)$ and $\hh_n(z)$ do not have real zeros.
} 
In view of \eqref{eq:obserrem}, we can find generic positive constants $C_1$, $C_2$ and $C_3$ depending on $n$ such that
\begin{equation*}
    |f_n^1(z)| \ge |h_n^{1}(k)|\frac{e^{|\im z|}-1}{2} - e^{|\im z|} C_1 \frac{1}{|\im z|} \ge C_2\,,
\end{equation*}
when $|\im z| \ge C_3$. Combining the above estimate with the symmetry of the zeros,
it readily follows that the zeros of $f_n^1(z)$ must lie in the strip:
\begin{equation*}
 \{z\in \C\,;\ |\im z|\le C_3\}\,.
\end{equation*}
In this region, the remainder term  $e^{|\im z|}O(|z|^{-1})$ in \eqref{eq:asyf1} converges to zero as $|z| \to \infty$. Since all the zeros of the entire function $\cos(z-\frac{n \pi }{2})$ are real and simple, given by 
\begin{equation} \label{eq:zerocos}
    \w{z}_{n,l} = \frac{(1+2l+n)\pi}{2}\,,\q l \in \NN\,,
\end{equation}
we foresee that there are zeros of $f_n^1(z)$ lying near $\w{z}_{n,l}$ when $|z|$ is large enough, which is indeed the case, by a direct application of Rouch\'{e}'s theorem and inverse function theorem. To see this, we define the entire function $g_n(z) = h_n^{(1)}(k)\cos(z-\frac{n\pi}{2})$ on the complex plane $\C$, which has the minimal period $2\pi$ in the sense that if $\alpha \in \C$ satisfies $g_n(z + \alpha) = g_n(z)$ for all $z$, then $\alpha = 2 \pi m$ for some integer $m$. Noting that $g'_n(\w{z}_{n,l}) \neq 0$ for $l \in \NN$, by the inverse function theorem, we can find an open neighborhood $V_n$ of $\w{z}_{n,0}$ and an open neighborhood $W_n$ of
the origin such that $g(V_n + l\pi) = (-1)^l W_n$ and $g$ is an analytic isomorphism from the neighborhood $V_n + l \pi$  of $\w{z}_{n,l}$ to the neighborhood $(-1)^lW_n$ of $0$ for each $l \in \Z$, where we also use the periodicity and symmetry of $g_n(z)$ : $g_n(z + l\pi) = (-1)^l g_n(z)$, $l \in \Z$. 
We denote by $r_n$ the radius of the largest ball contained in $V_n$ with center at the $\w{z}_{n,0}$,  and define 
\begin{equation*}
    M_n := \inf_{z \in \p B(\w{z}_{n,l},r_n)}|g_n(z)|\,,
\end{equation*}
which is independent of the value of $l$. When $l$ is large enough, we can guarantee 
\begin{equation*}
    \sup_{z \in \p B(\w{z}_{n,l},r_n)}|f^1_n(z)-g_n(z)| <  M_n\,,
\end{equation*}
by using the asymptotic expansion \eqref{eq:asyf1}. Then, the Rouch\'{e}'s theorem helps us to conclude that in the region $B(\w{z}_{n,l},r_n)\subset V_n + l \pi$, $f_n^1(z)$ has a simple zero denoted by $z_{n,l}$. It then directly follows that  $g_n(z_{n,l}) \in (-1)^l W_n$ and 
\begin{equation} \label{eq:appznl_1}
  0 = g_n(\w{z}_{n,l}) = f_n^1(z_{n,l}) = g_n(z_{n,l}) + O\left(\frac{1}{|z_{n,l}|}\right)\,.
\end{equation}
Hence we have, by using \eqref{eq:appznl_1} and
the local invertibility of $g'_n$,
\begin{equation*}
\frac{|z_{n,l}-\w{z}_{n,l}|}{|g_n(z_{n,l})-g_n(\w{z}_{n,l})|} = \frac{|z_{n,l}-\w{z}_{n,l}|}{|g_n(z_{n,l}))|} \le \sup_{\xi \in (-1)^lW_n }|(g_n^{-1})'(\xi)| = \sup_{z\in V_n} |g_n'(z)|^{-1} < + \infty\,,
\end{equation*}
which immediately implies 
\begin{equation} \label{eq:localre}
    |z_{n,l}-\w{z}_{n,l}| \le C_n \frac{1}{|z_{n,l}|} \le C_n |l|^{-1}, 
\end{equation}
for large enough $l$, where $C_n$ denotes a generic constant depending on $n$ and may have different \zz{values in the following}. Further, considering the fact that a non-constant analytic function on the closure of a bounded domain can only have finite zeros, we arrive at the following result.
\begin{lemma}
The zeros of $f_n^1(z)$ are symmetric with respect to the origin and contained in the strip: $\{z\in \C\,;\ |\im z|\le C\}$ for some constant $C$. Let $\{z^1_{n,l}\}_{l\in \NN}$ denote the zeros with $-\frac{\pi}{2} < \arg(z)\le \frac{\pi}{2}$. Then $\{z^1_{n,l}\}$ has the following estimate:
\begin{equation} \label{eq:est_zeros1}
    |z^1_{n,l}-\w{z}_{n,l}^1| \le C_n l^{-1} \q \text{for all}\ l \in \NN^+\,,
\end{equation}
where $\{\w{z}_{n,l}^1\}$ is given by \eqref{eq:zerocos}.
\end{lemma}

Recall that what we are truly interested in is $\lad^1_{n,l} :=  k^2/((z^1_{n,l})^2-k^2)\in \sigma_n^1$. We translate the above lemma  with respect to $z^1_{n,l}$ to $\lad^1_{n,l}$ and obtain
\begin{equation*}
\left|\lad^1_{n,l} - \frac{4k^2}{(1+2l+n)^2\pi^2 - 4k^2}\right|\le C_n |l|^{-4}\,,
\end{equation*}
by applying the mean-value theorem to the one-dimensional function $h(t) = 
k^2/((\w{z}^1_{n,l}+t(z^1_{n,l}-\w{z}^1_{n,l}))^2-k^2)$ on $[0,1]$.  This estimate can be further simplified as follows: 
\begin{equation*}
\left|\lad^1_{n,l} - \frac{4k^2}{(1+2l+n)^2\pi^2}\right|\le C_n |l|^{-4}\q \text{as} \ l\to +\infty \,.
\end{equation*}

For our second case, by a very similar argument applied to $zf_n^2(z)$, which has the same zeros away from the origin as $f_n^2(z)$ and satisfies the following asymptotic form:
\begin{align}  \label{eq:asyf2}
  z f_n^2(z) &= \frac{k^2}{z} h_n^{(1)}(k) \cos\left(z - \frac{n\pi}{2}\right)-  \hh_n(k)\cos\left(z - \frac{n \pi}{2} - \frac{\pi}{2}\right) + e^{|\im z|} O\left(\frac{1}{|z|}\right)\notag \\
  & =-  \hh_n(k)\cos\left(z - \frac{n \pi}{2} - \frac{\pi}{2}\right) + e^{|\im z|} O\left(\frac{1}{|z|}\right)\q \text{as} \ |z|\to \infty \,,
\end{align}
we can obtain that the zeros $\{z^2_{n,l}\}$ of $f_n^2(z)$ in the right half plane satisfy the estimate:
\begin{equation} \label{eq:est_zeros2}
     \left|z^1_{n,2}-\w{z}_{n,l}^2\right| \le C_n l^{-1} \q \text{with}\ \w{z}_{n,l}^2 := \frac{(2l + n)\pi}{2}\q \text{for all}\ l \in \NN^+\,, 
\end{equation}
and the associated $\{\lad^2_{n,l}\} \subset \sigma_n^2$ have the asymptotics:
\begin{equation*}
    \left|\lad_{n,l}-\frac{4k^2}{(2l+n)^2\pi^2}\right|\le C_n |l|^{-4}\q \text{as} \ l\to +\infty \,.
\end{equation*}

We now give the main result of this subsection.
\begin{theorem} \label{thm:asyeigenvalue}
Let $\{\lad_{n,l}^i\}_{l \in \NN}$ be the eigenvalues in $\sigma_n^i$ for $i = 1,2$ and $n \in \NN^+$. Then, when $l \to +\infty$, the following asymptotic estimates hold,
\begin{equation} \label{eq:asyeigenvalue}
    \left|\lad^1_{n,l} - \frac{4k^2}{(1+2l+n)^2\pi^2}\right| = O(l^{-4})\,,\q  \left|\lad^2_{n,l}-\frac{4k^2}{(2l+n)^2\pi^2}\right| = O(l^{-4})\,.
\end{equation}
\end{theorem}
\zz{We refer the readers to Theorem \ref{thm:eigfree_region} for an interesting related result.}


\subsubsection{Asymptotic behavior and localization of eigenfunctions}
Theorem \ref{thm:asyeigenvalue} has clearly described asymptotic behaviors of the eigenvalues in $\sigma_n^i$, $i = 1,2$. We see from \eqref{eq:asyeigenvalue} and \eqref{eq:ppres} that when $l$ is large enough,  $\lad_{n,l}^1$ and $\lad_{n,l}^2$ will most likely be contained in the $\ep$-neighborhood of $\tau^{-1}$ so that the high-frequency resonant modes $\wete(k_\lad,x)$, $\wetm(k_\lad,x)$ for the same value of $n$ and $m$ will be excited simultaneously.  Via the integral operator $\td$, these resonant modes carrying the subwavelength information of the embedded sources can propagate into the far field. In this subsection, instead of considering the vector fields $\td[\wete(k_\lad,\dd)](x)$ and $\td[\wetm(k_\lad,\dd)](x)$, we consider their tangential component measurements for ease of exposition, which can be explicitly represented by
\begin{equation*} 
    \h{x} \t \td[\wete(k_\lad,\dd)](x)  = ik^3 \sqrt{n(n+1)}h^{(1)}_n(k|x|)\unm(\h{x}) \int_0^1 j_n(kr) j_n(k_\lad r)r^2 dr\,,
\end{equation*}
and 
\begin{equation*}
    \h{x} \t \td[\wetm(k_\lad,\dd)](x) = - \frac{k\sqrt{n(n+1)}}{k_\lad |x|}\hh_n(k|x|)\vnm(\h{x})\int_0^1 \jj_n(kr) \jj(k_\lad r) + n(n+1) j_n(kr)j_n(k_\lad r) dr\,,
\end{equation*}
for $|x| > 1$, by \eqref{eq:tancomm_1} and \eqref{eq:tancomm_2} in Appendix \ref{app:C_3}. These formulas motivate us to define the following two propagating functions, respectively, responsible for the propagation of vector spherical harmonics $\unm$ and $\vnm$: 
\begin{equation} \label{def:prog_1}
   \vp_n^{\lad, 1}(kt) := \left\{ \begin{array}{ll}
   \sqrt{n(n+1)} \lad j_n(k_\lad t)      &   0 < t\le 1\\
   ik^3 \sqrt{n(n+1)}h^{(1)}_n(k t)\int_0^1 j_n(kr) j_n(k_\lad r)r^2 dr      & t >1
    \end{array} \right.\,,
\end{equation}
and 
\begin{equation} \label{def:prog_2}
     \vp_n^{\lad, 2}(kt) := \left\{ \begin{array}{ll}
   \frac{i\lad \sqrt{n(n+1)}}{k_\lad t}\jj_n(k_\lad t)    &   0 < t\le 1\\
 - \frac{k \sqrt{n(n+1)}}{k_\lad t} \hh_n(k t)\int_0^1 \jj_n(kr) \jj(k_\lad r) + n(n+1) j_n(kr)j_n(k_\lad r) dr     & t >1
    \end{array} \right.\,.
\end{equation}
Here, to define $\vp_n^{\lad,i}$ inside the domain for $i = 1,2$,  we have used the fact that $\wete$ and $\wetm$ are eigenfunctions of $\td$ with eigenvalue $\lad$. From the definitions \eqref{def:prog_1} and \eqref{def:prog_2}, we readily see that when $t> 1$, $\vp_n^{\lad,1}$(resp., $\vp_n^{\lad,2}$) is proportional to $h_n^{(1)}(kt)$ (resp., $\hh_n(kt)$), and thereby has the same asymptotic behavior as $h_n^{(1)}(kt)$ (resp., $\hh_n(kt)$) as $t \to + \infty$. To understand the roles played by $\vp_n^{\lad,i}$ for different orders $n$ in the far-field measurement,  we give the result about their asymptotics for large order $n$.  The detailed calculations and estimates are included in Appendix \ref{app:C_3}. 
\begin{proposition} \label{prop:asylarge}
The following asymptotic estimates uniformly hold for $t$ in a compact subset of $(1, +\infty)$,
\begin{equation} \label{eq:aympfvp}
    \vp_n^{\lad , 1}(t) = O\left( \left(\frac{e}{2 t}\right)^{n+1} \frac{k_\lad^{n-1}}{(n+1)^n}\right)\,,\q \vp_n^{\lad , 2}(t) = O\left( \left(\frac{e k}{2 t}\right)^{n-1} \frac{k_\lad^{n-2}}{(n-1)^{n-3}}\right) \ \text{as} \ n \to \infty\,,
\end{equation}
where we remind that the big-$O$ terms are bounded by constants independent of $n$ but depending on other parameters: the wave 
number $k$, the eigenvalue $\lad$, and the compact set for variable $t$.
\end{proposition}
In view of the exponential decay of propagating functions $\vp_n^{\lad,i}$ in \eqref{eq:aympfvp} when $n$ tends to infinity, we have theoretically justified the previously mentioned fact in the introduction that the evanescent part of the radiating EM wave with the fine-detail information of the objects, i.e., the remainder term of the infinite sum in \eqref{eq:repoute} from large enough $n$, is almost negligible in the measured far-field data. 
It is the low-frequency component:  
\begin{equation*}
    E_{low}(x) = \sum^N_{n = 1} \sum^n_{m = -n}
    \gnm \ete(k,x) + \enm \etm(k,x)\,,\q \gnm, \enm \in \C\,, \ |x| \gg 1
\end{equation*}
that dominates the far-field behavior of the radiating wave $E$, where $N$ is a given small positive integer . We plot both real and imaginary parts of $\vp_n^{\lad , i}$ in Figures \ref{fig:eigenf5} and  \ref{fig:eigenf9} for different values of $n$ and $k=1$, from which we can clearly observe that the higher the resonant mode oscillates, the smaller the amplitude is. 

\begin{figure}[!htbp]
    \centering
    \subfigure[$\vp_5^{\lad,1}(t)$. Left: $t\in(0,1)$; right: t$\in(1,6)$.]{\label{fig:eigenf51}
    \begin{minipage}{0.48\textwidth}
       \centering
       \includegraphics[clip,width=0.48\textwidth]{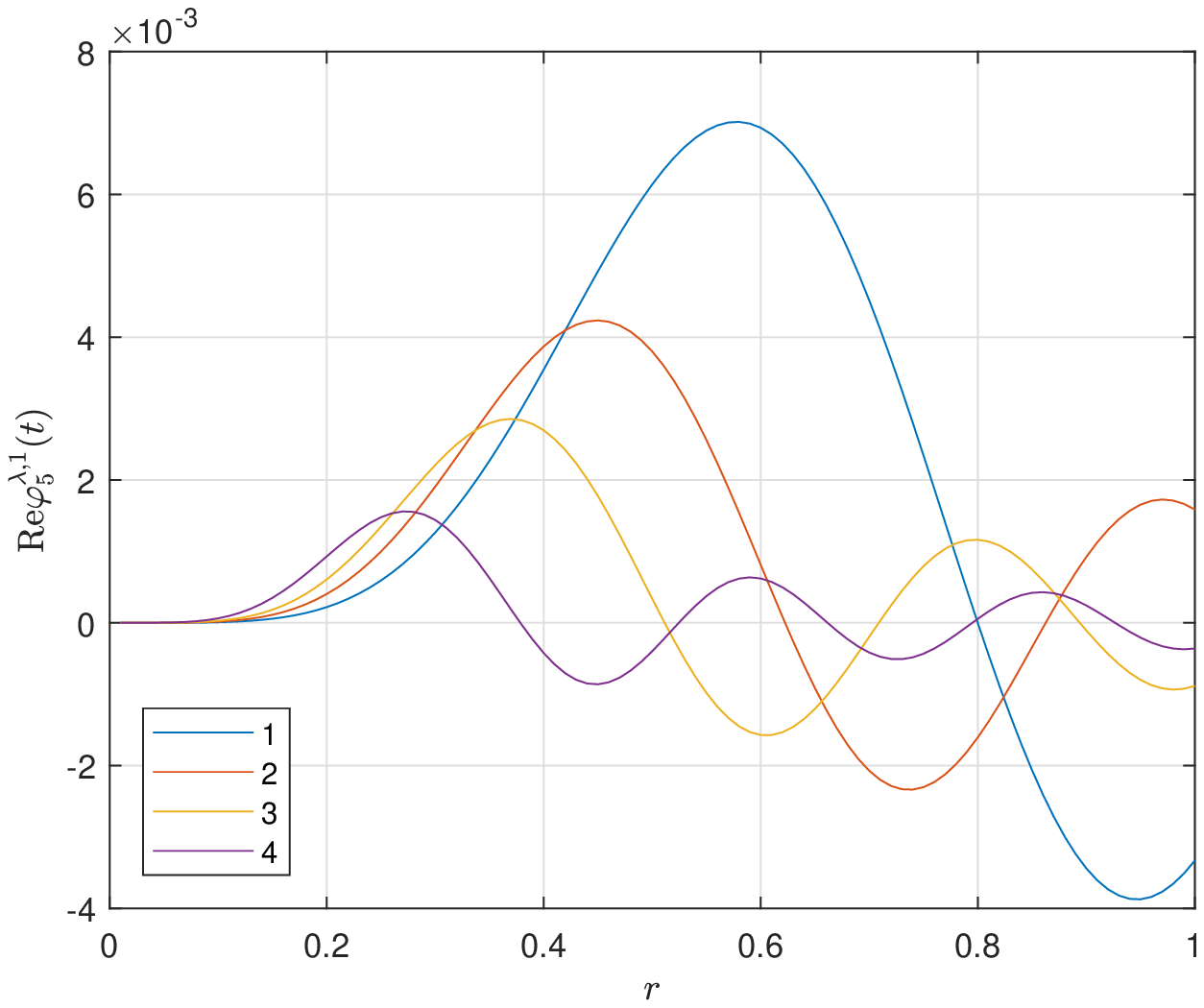} \hskip -0.1cm
\includegraphics[clip,width=0.48\textwidth]{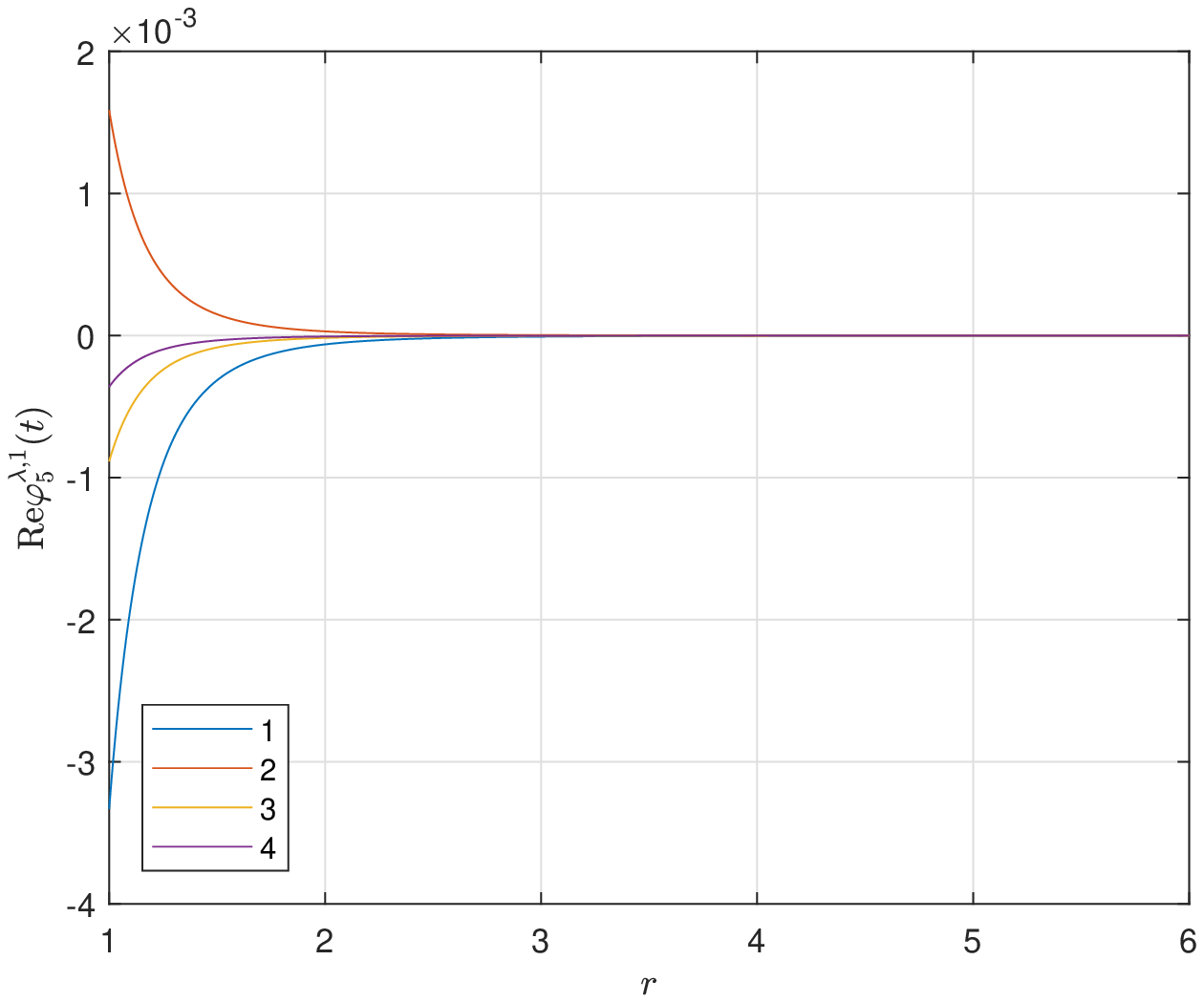} \\
\includegraphics[clip,width=0.48\textwidth]{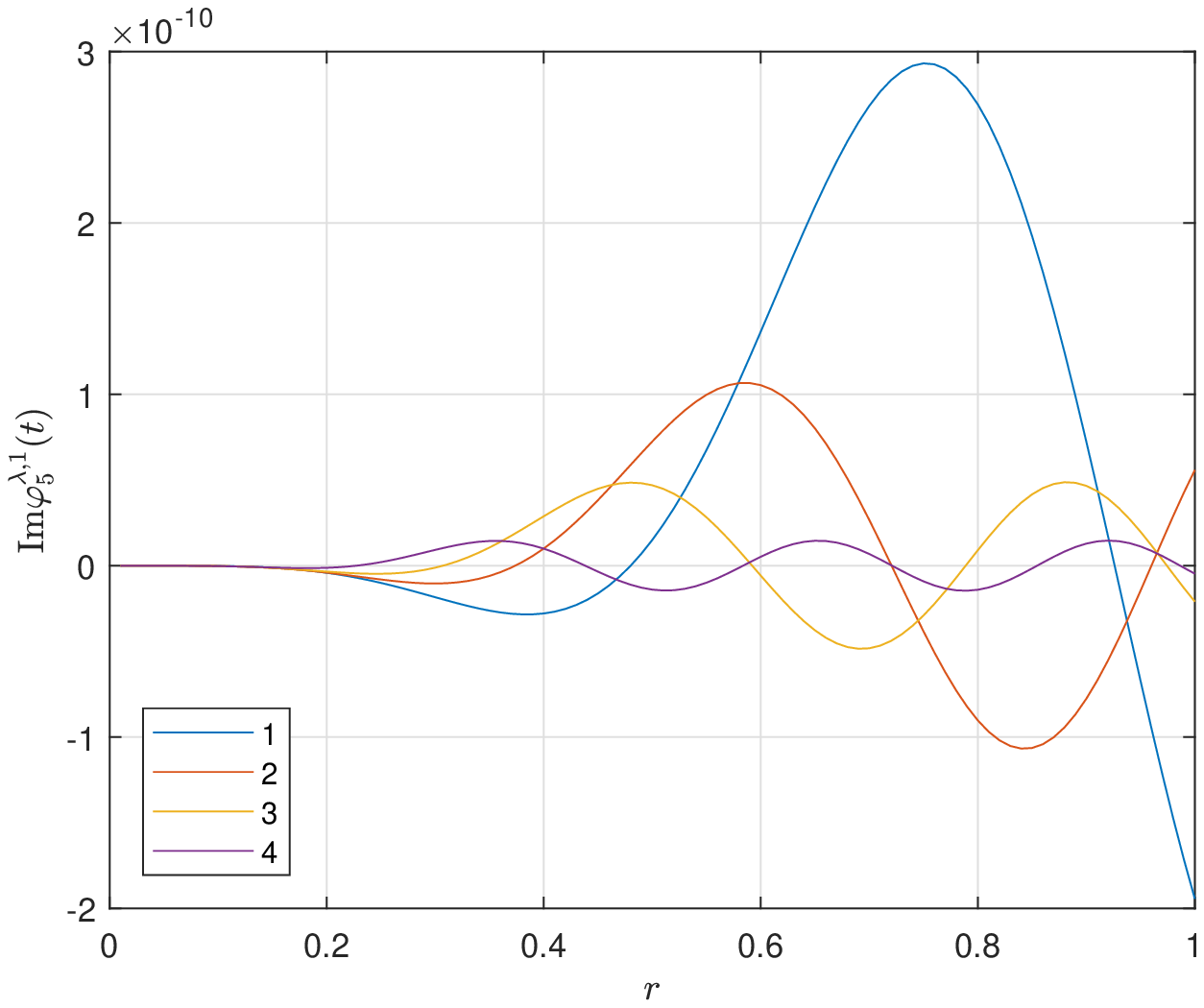} \hskip -0.1cm
\includegraphics[clip,width=0.48\textwidth]{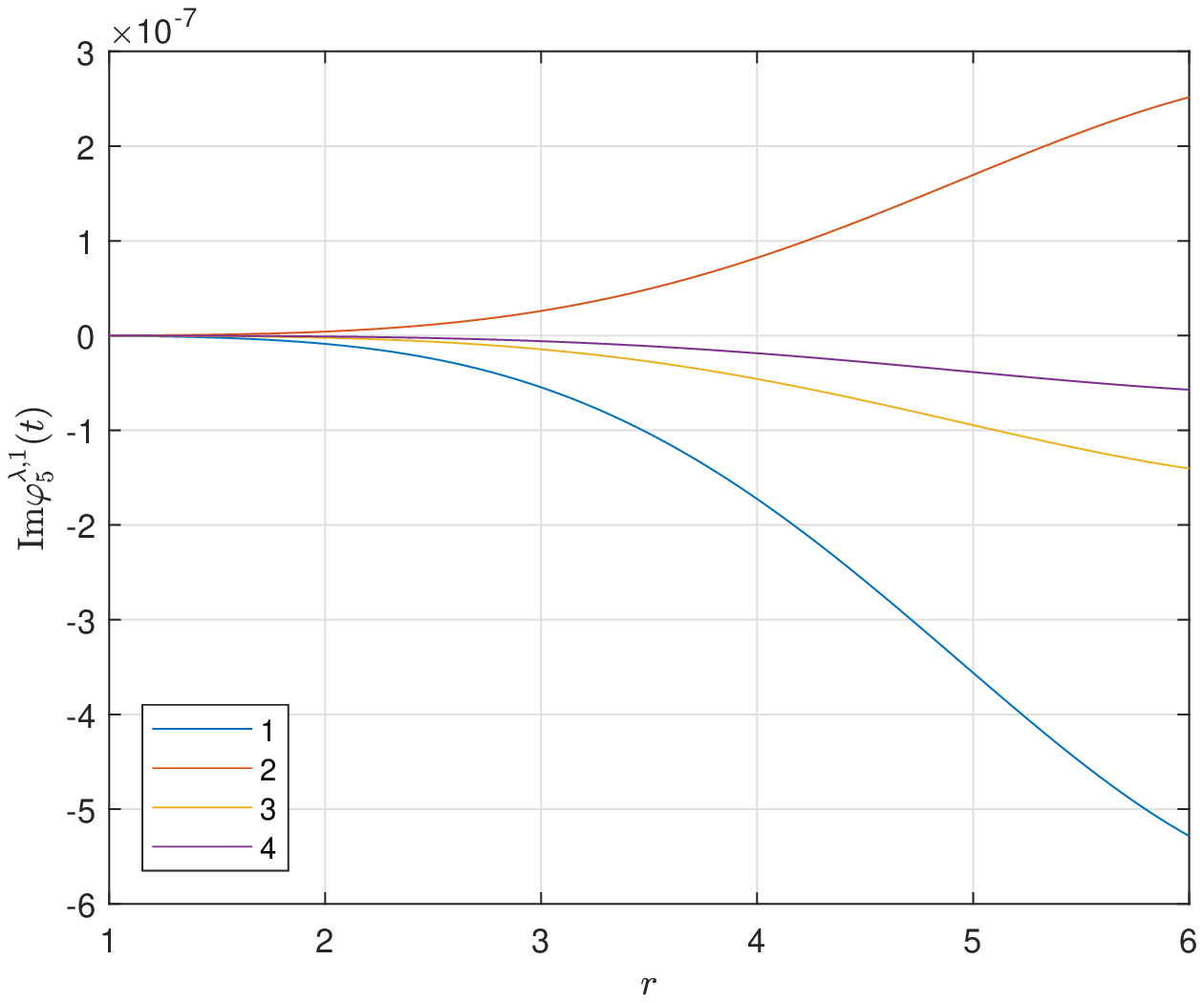} 
    \end{minipage}}%
    \hskip -0.2cm
  \subfigure[$\vp_5^{\lad,2}(t)$. Left: $t\in(0,1)$; right: t$\in(1,6)$.]{ \label{fig:eigenf52}
    \begin{minipage}{0.48\textwidth}
       \centering
     \includegraphics[clip,width=0.48\textwidth]{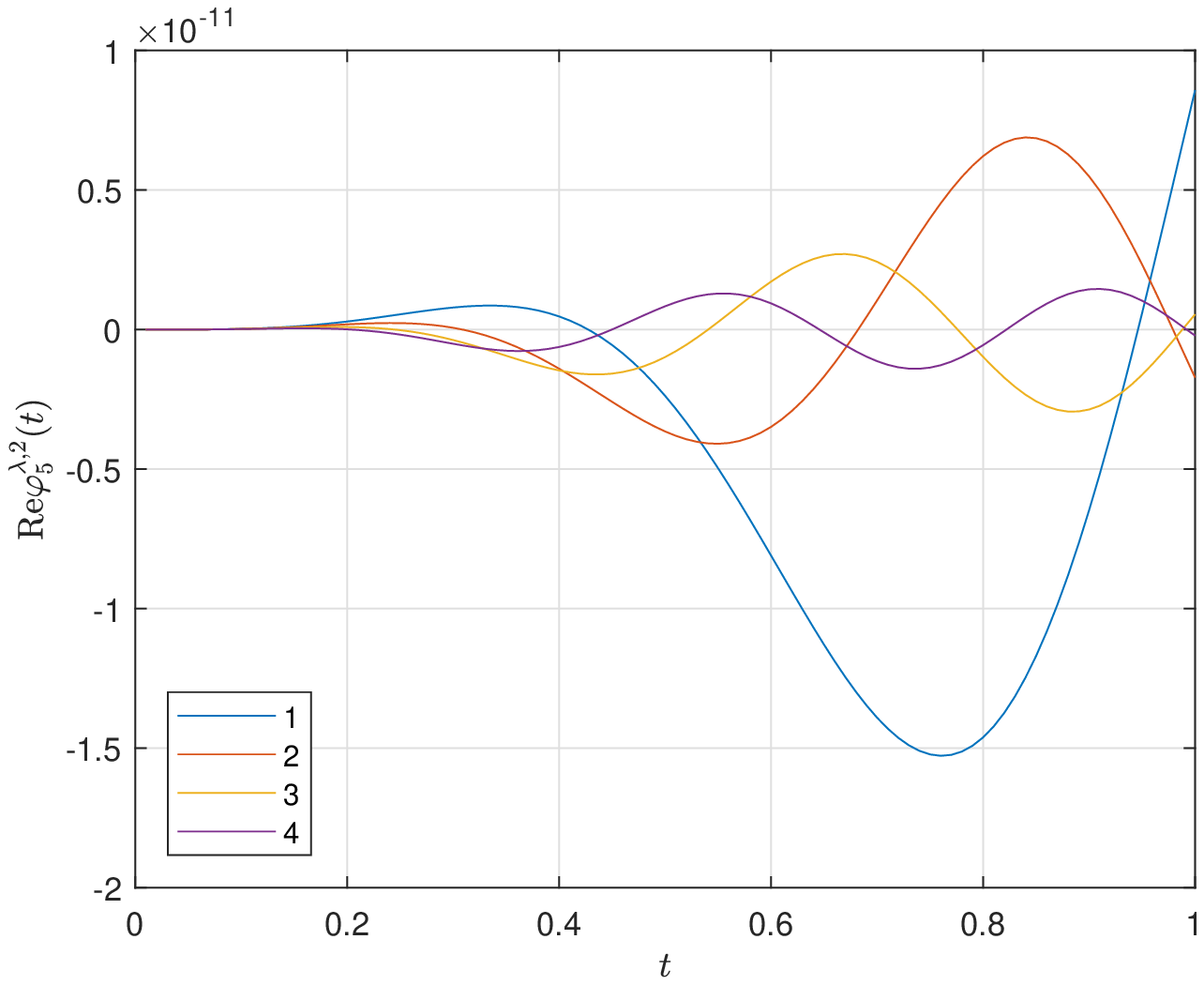} \hskip -0.1cm
\includegraphics[clip,width=0.48\textwidth]{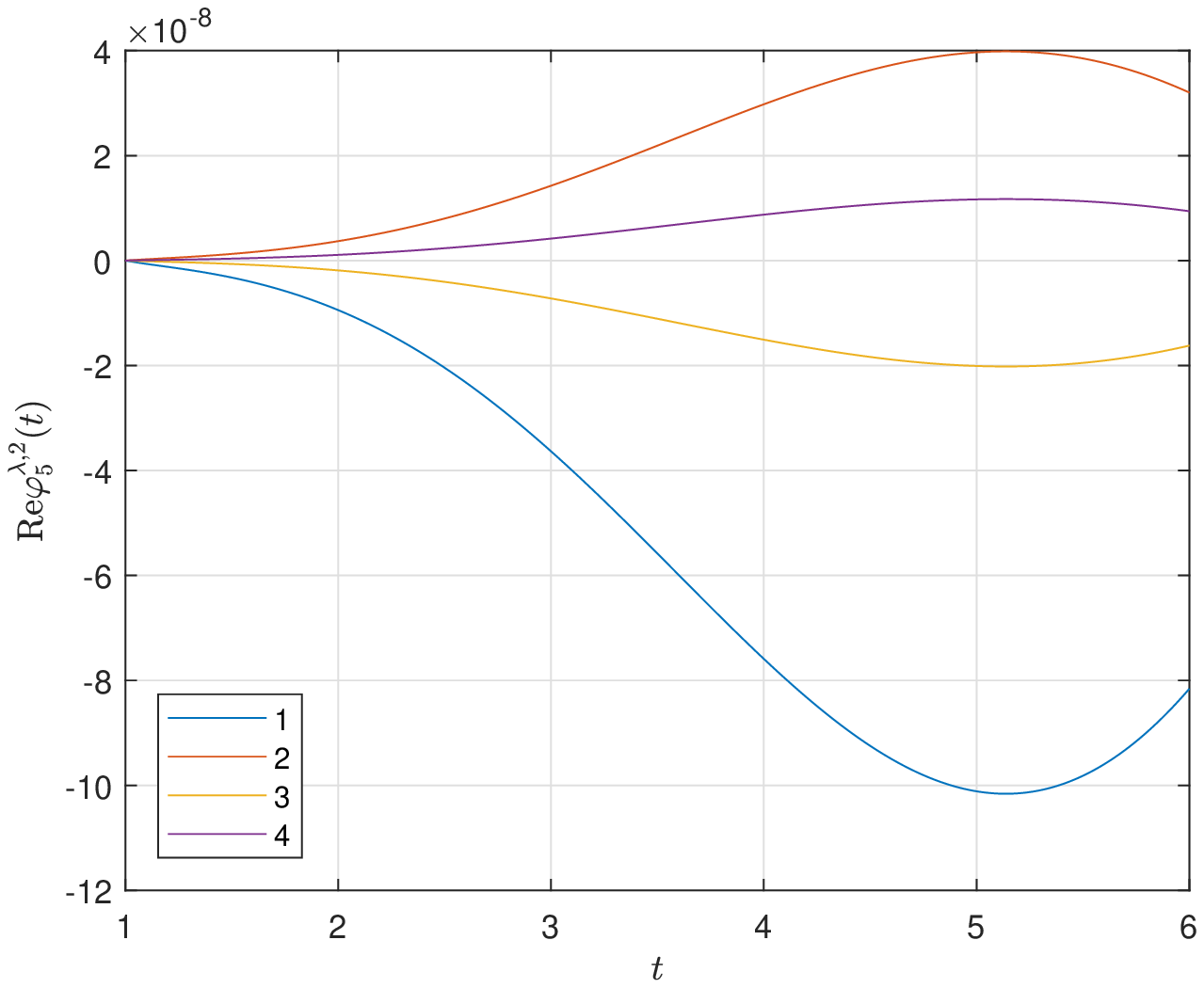}  \\
\includegraphics[clip,width=0.48\textwidth]{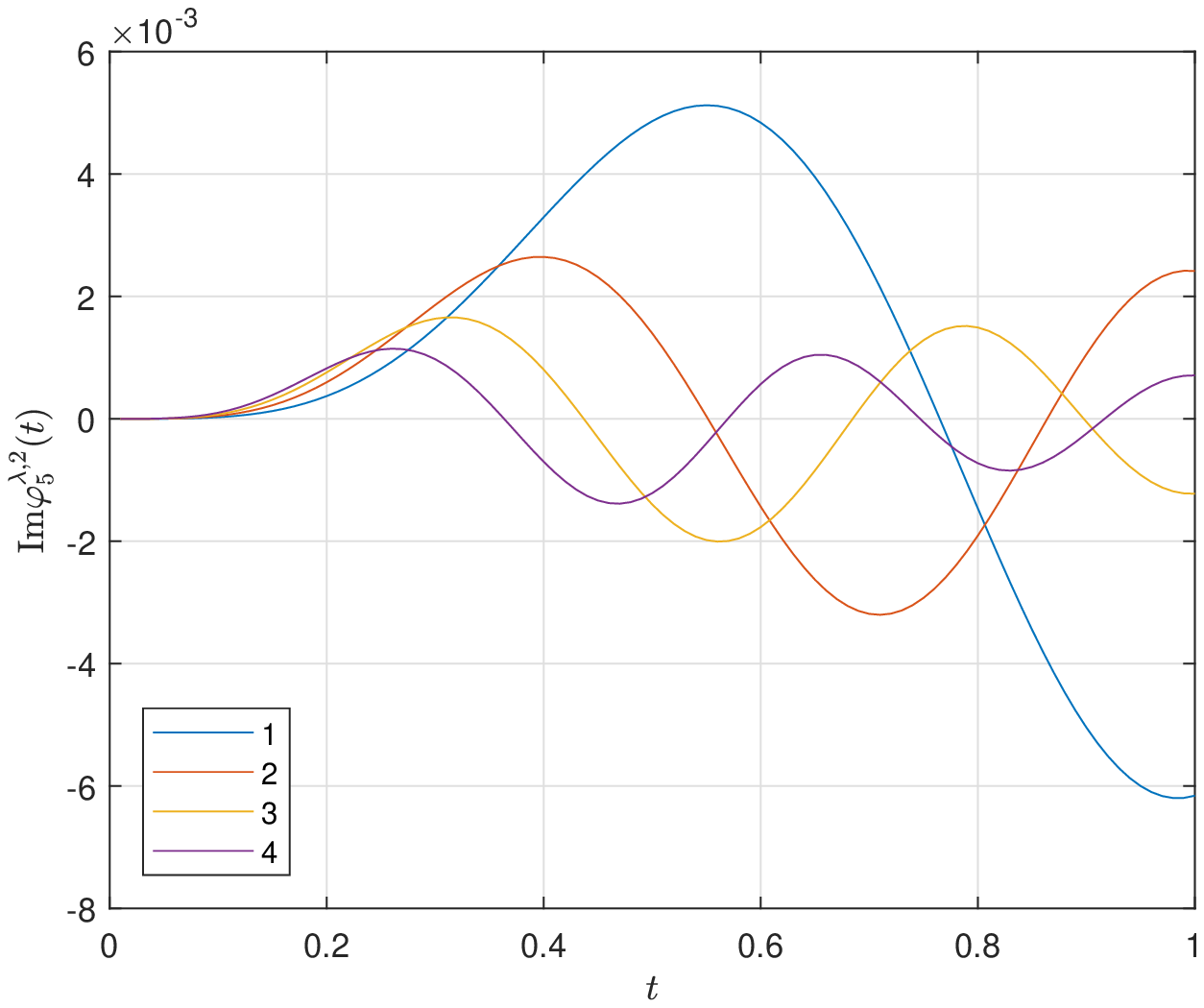} \hskip -0.1cm
\includegraphics[clip,width=0.48\textwidth]{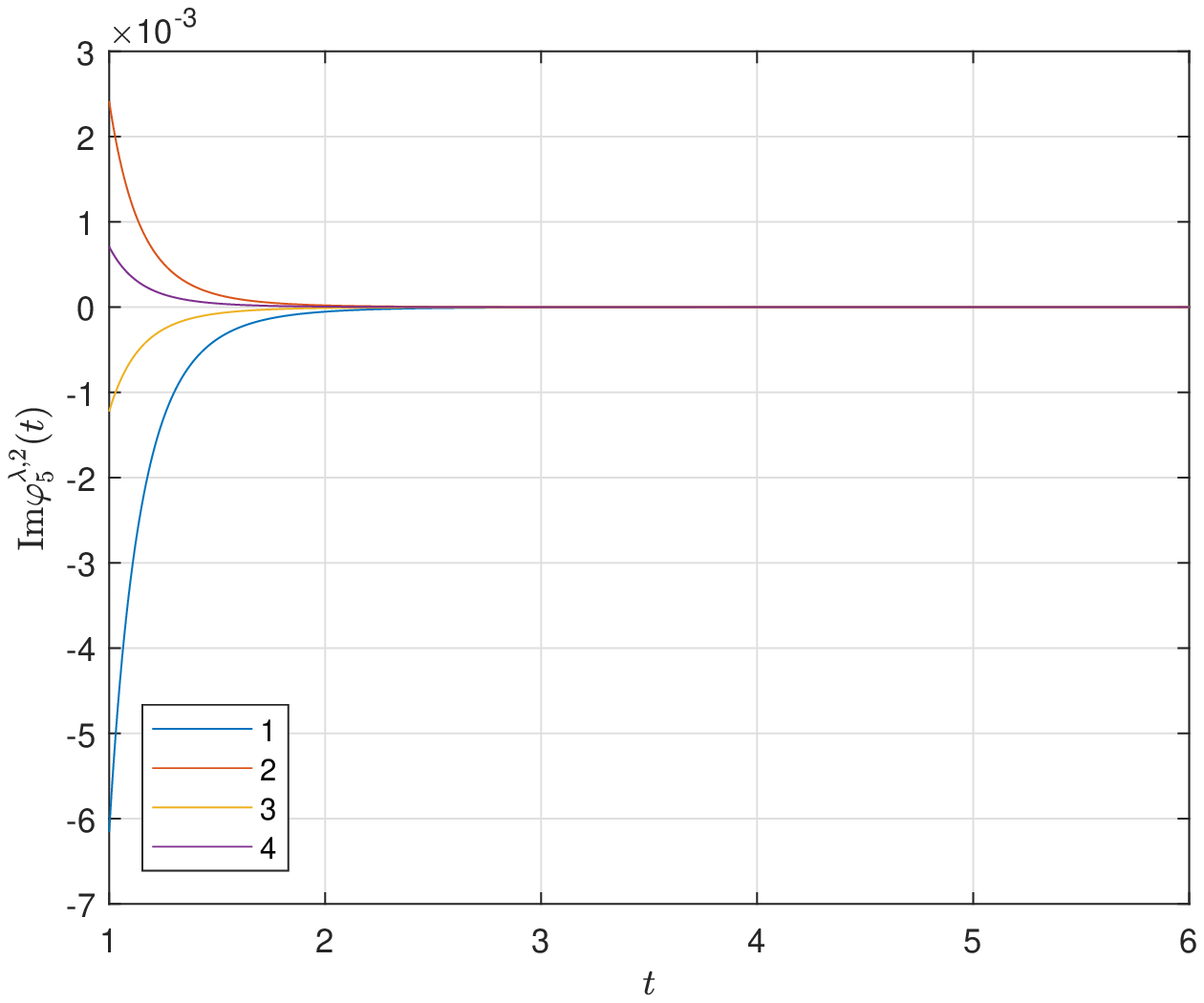}
    \end{minipage}}
    \caption{ Propagating function $\vp_5^{\lad,i}$ for the first four $\lad$ from $\sigma_5^i$, $i=1,2$.  First row: real part of $\vp_5^{\lad,i}$; second row: imaginary part of $\vp_5^{\lad,i}$, for $i=1,2$. }
    \label{fig:eigenf5}
\end{figure}

We also note from Figures \ref{fig:eigenf51} and \ref{fig:eigenf91} that the imaginary parts of $\vp_n^{\lad,1}$ for different $n$ have very small amplitudes inside and outside the domain, while for the case $\vp_n^{\lad,2}$, it is the real part. However, it is not a surprising fact, if we recall from Theorem \ref{thm:asyeigenvalue} that the eigenvalues $\lad$ of $\td$ is near the real axis and there is an additional factor $i$ in the definition of $\vp_n^{\lad,2}$, compared to $\vp_n^{\lad,1}$(cf.\,\eqref{eq:enete} and \eqref{eq:enetm}).

\begin{figure}[!htbp]
    \centering
    \subfigure[$\vp_9^{\lad,1}(t)$. Left: $t\in(0,1)$; right: t$\in(1,6)$.]{\label{fig:eigenf91}
    \begin{minipage}{0.48\textwidth} 
       \centering
       \includegraphics[clip,width=0.48\textwidth]{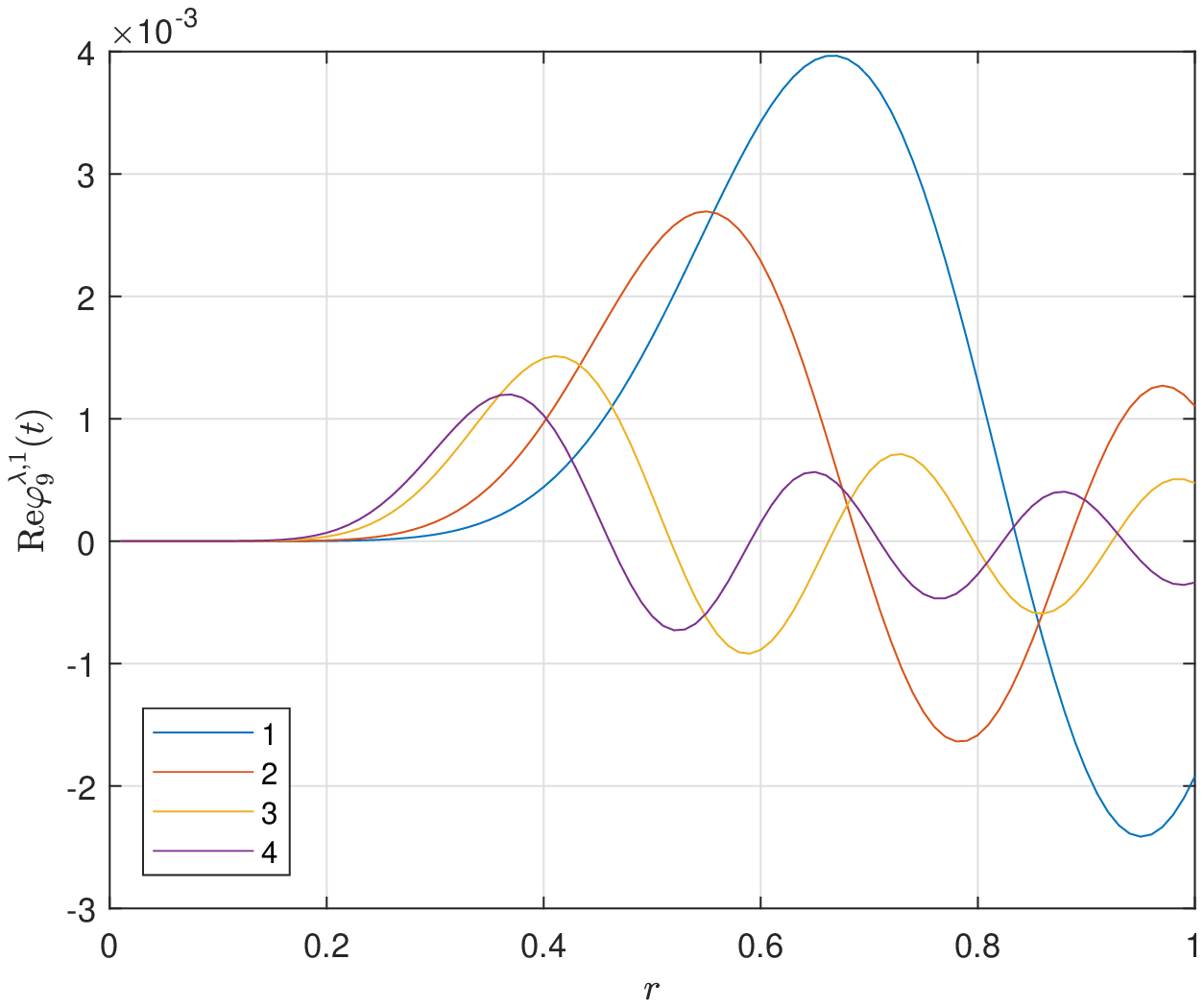} \hskip -0.1cm
\includegraphics[clip,width=0.48\textwidth]{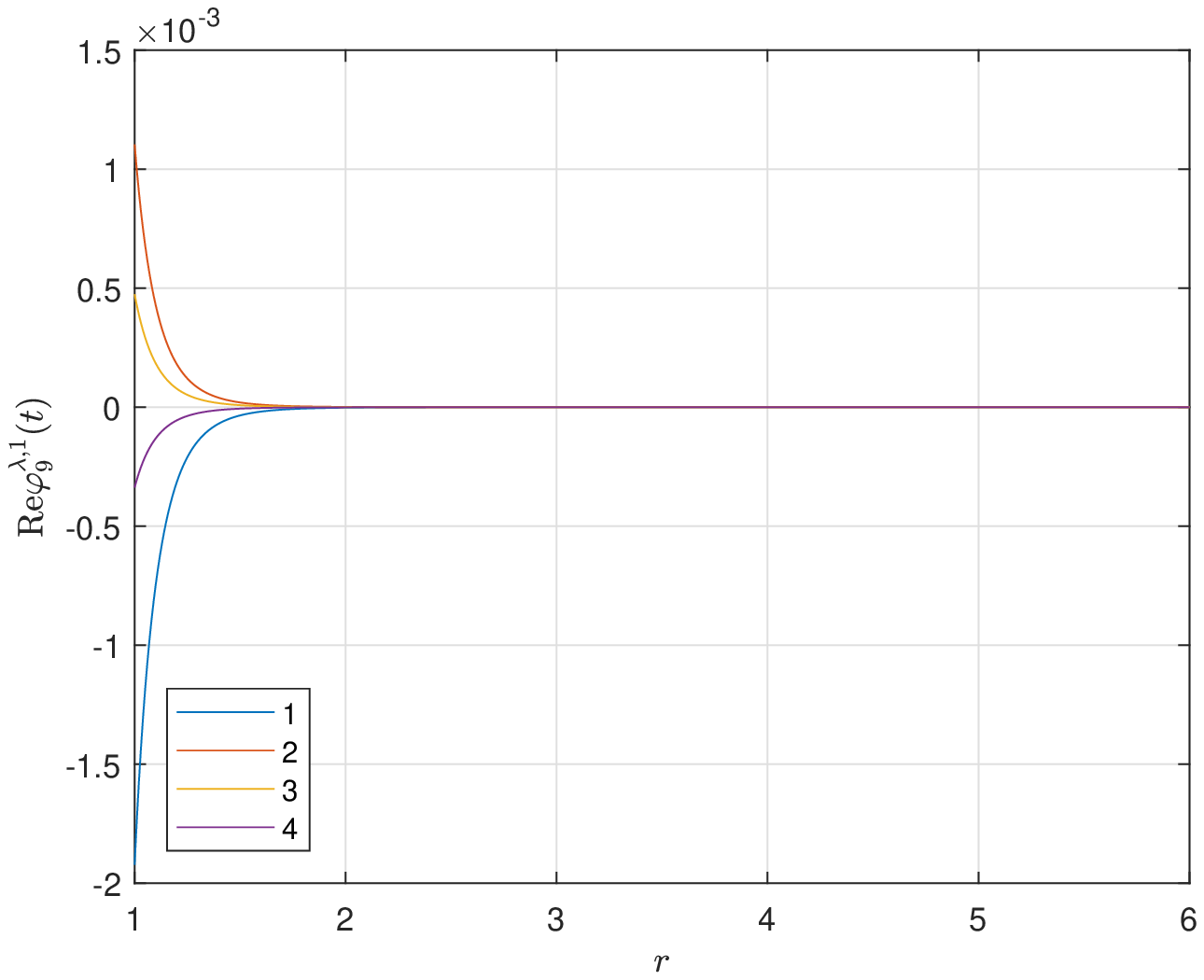} \\
\includegraphics[clip,width=0.48\textwidth]{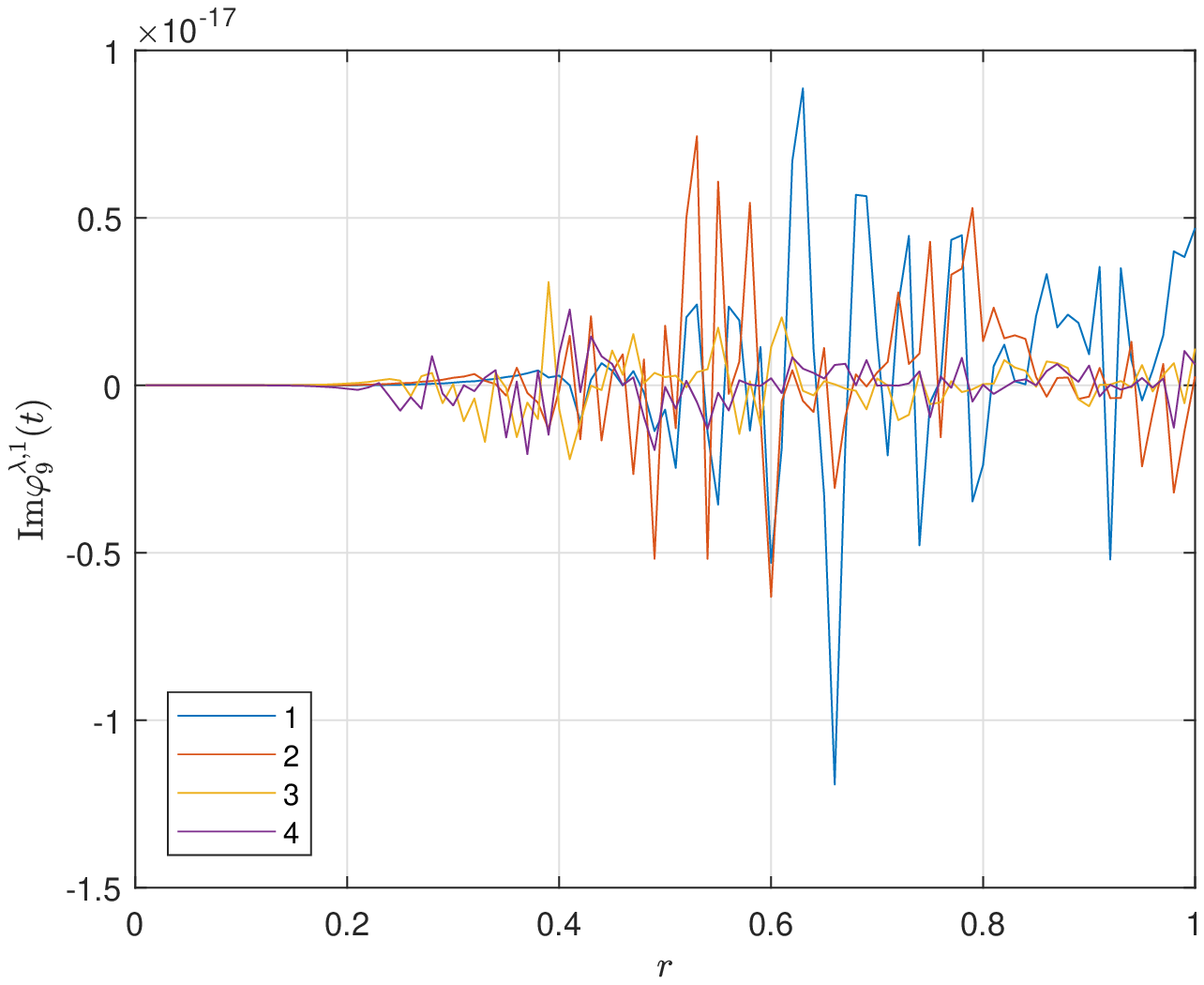} \hskip -0.1cm
\includegraphics[clip,width=0.48\textwidth]{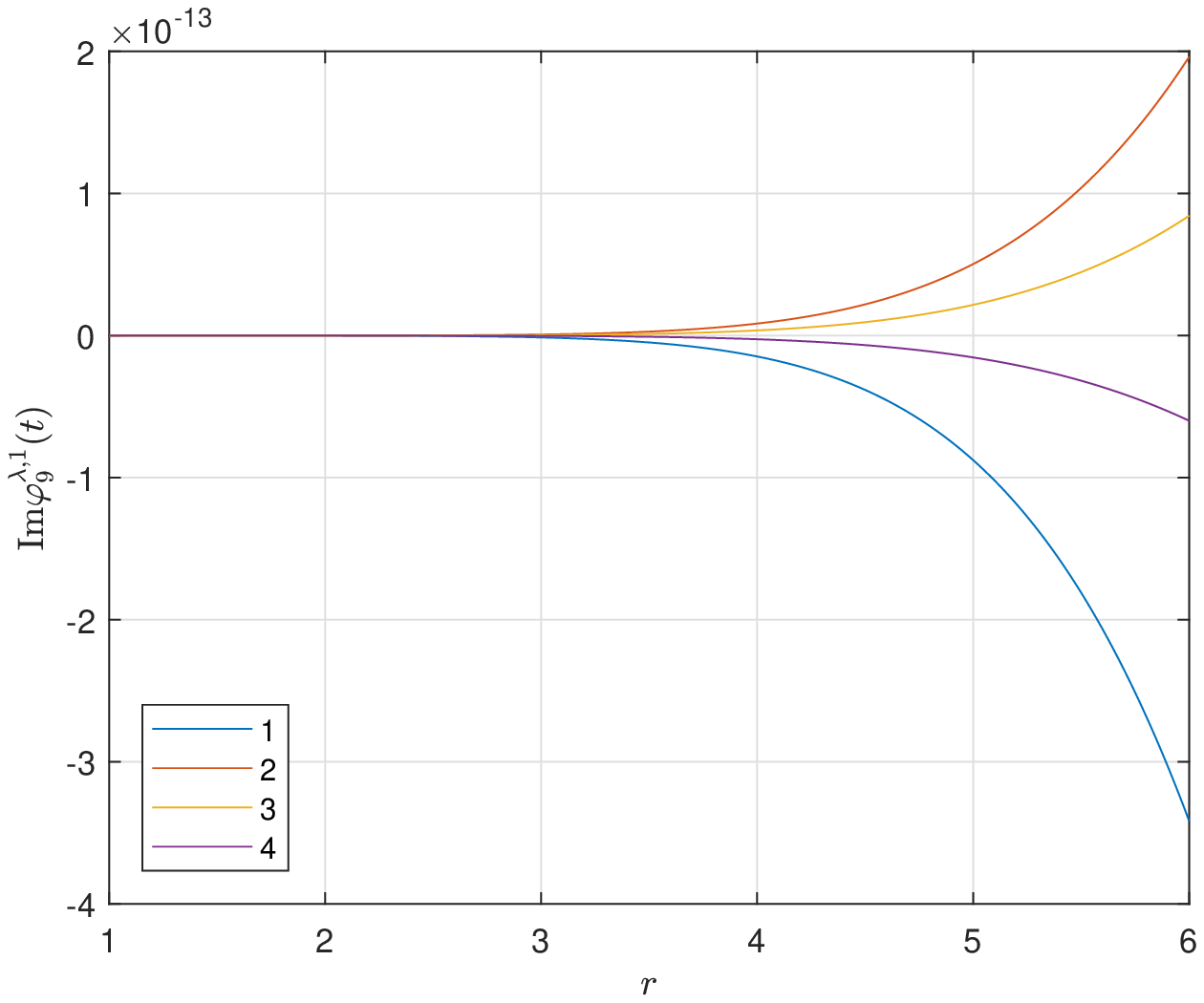} 
    \end{minipage}}%
    \hskip -0.2cm
  \subfigure[$\vp_9^{\lad,2}(t)$. Left: $t\in(0,1)$; right: t$\in(1,6)$.]{ \label{fig:eigenf92}
    \begin{minipage}{0.48\textwidth}
       \centering
     \includegraphics[clip,width=0.48\textwidth]{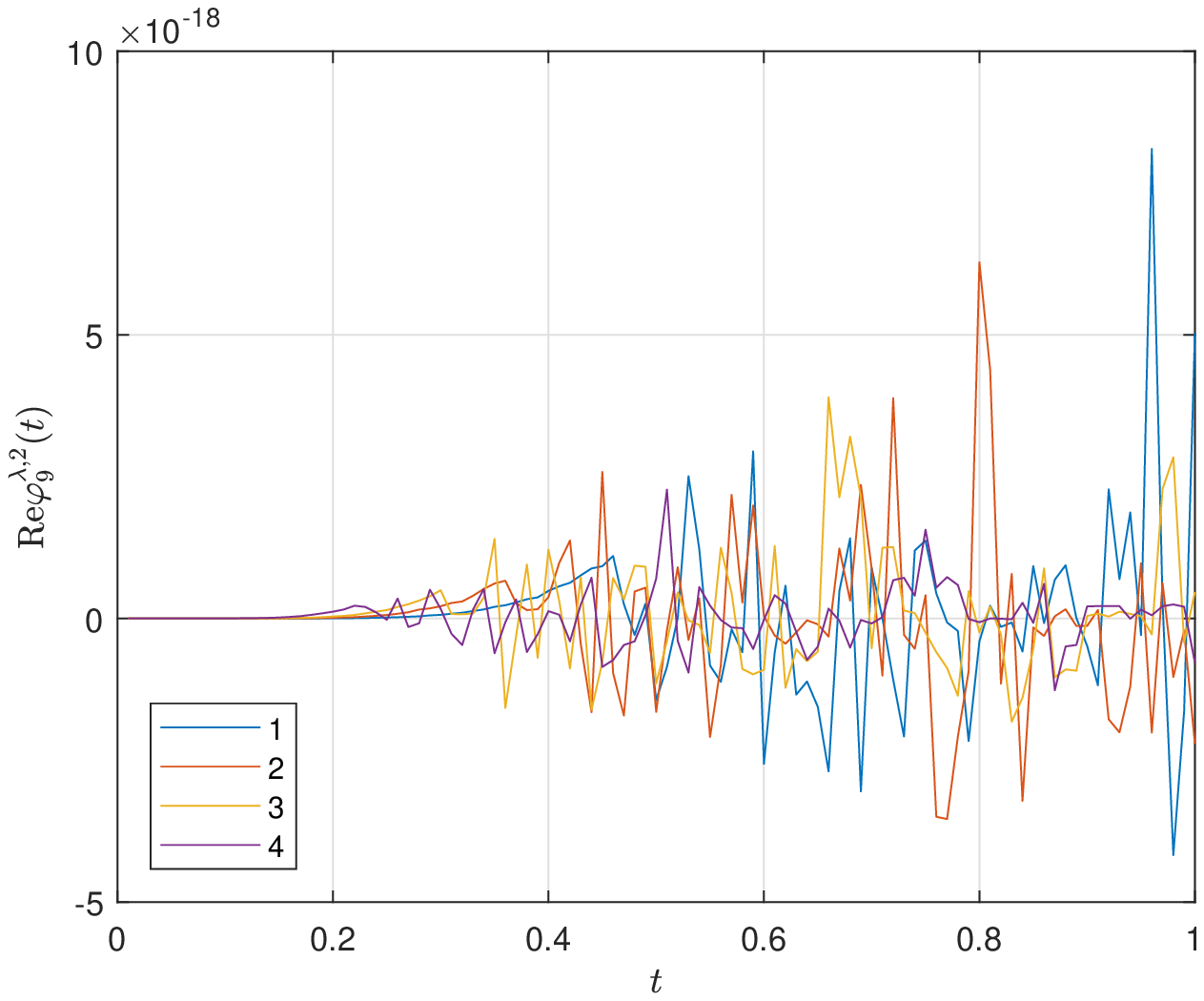} \hskip -0.1cm
\includegraphics[clip,width=0.48\textwidth]{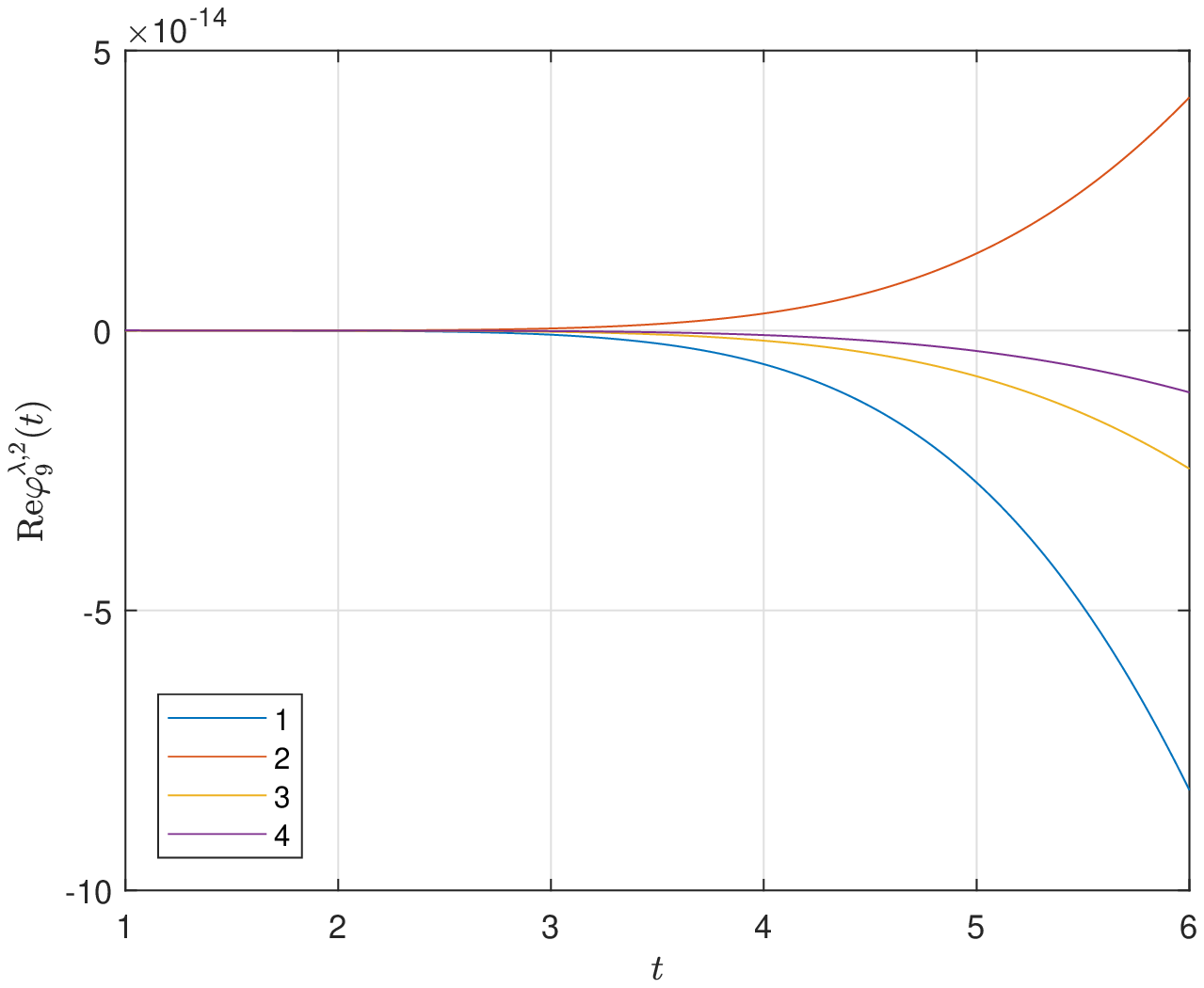}  \\
\includegraphics[clip,width=0.48\textwidth]{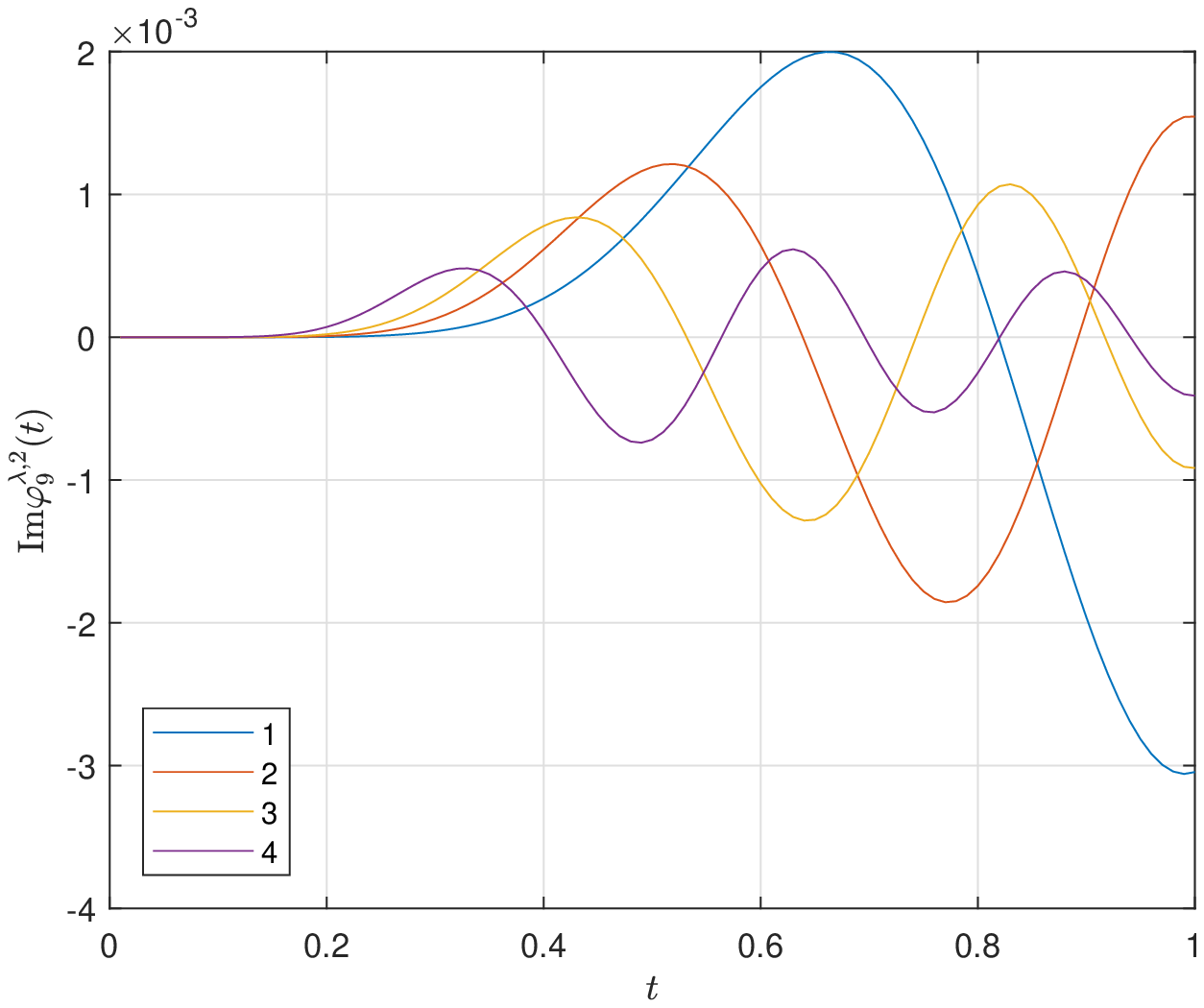} \hskip -0.1cm
\includegraphics[clip,width=0.48\textwidth]{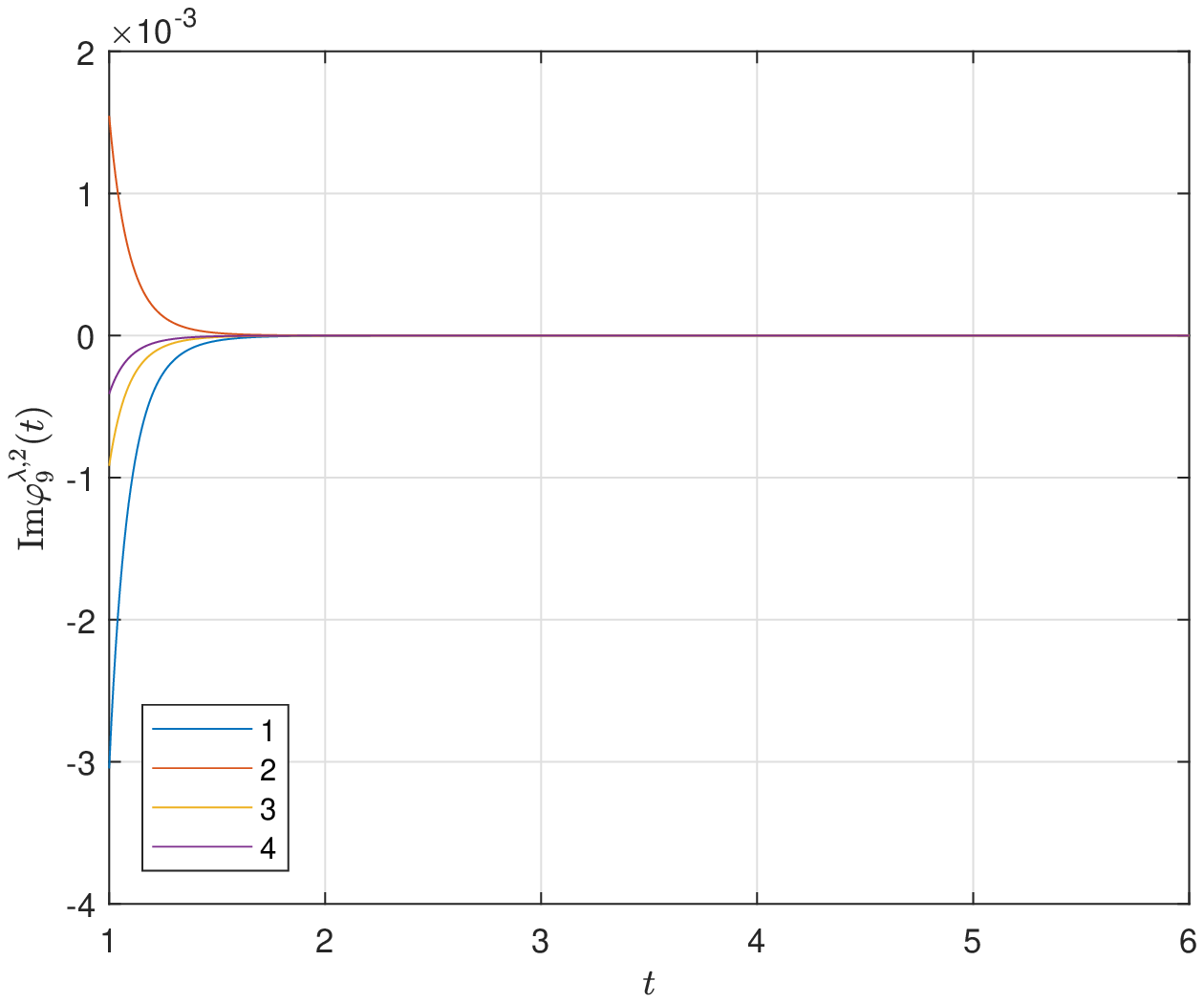}
    \end{minipage}}
    \caption{ Propagating function $\vp_9^{\lad,i}$ for the first four $\lad$ from $\sigma_9^i$, $i=1,2$.  First row: real part of $\vp_9^{\lad,i}$; second row: imaginary part of $\vp_9^{\lad,i}$, for $i=1,2$. }
    \label{fig:eigenf9}
\end{figure}

We next consider the behaviors of the propagating functions $ \vp_n^{ \lad^{\scriptscriptstyle i}_{\scriptscriptstyle n,l},i}(kt)$ for $i=1,2$ inside the domain. \zz{To simplify the notations, we
re-denote them as follows:}
\begin{equation} \label{def:progfun_1}
\vp_{n,l}^1(kt) :=\vp_n^{ \lad^{\scriptscriptstyle 1}_{\scriptscriptstyle n,l},1}(kt) = \sqrt{n(n+1)} \lad_{n,l}^1 j_n(z^1_{n,l} t) \q \text{for}\  t \in [0,1]\,, 
\end{equation}
and 
\begin{equation} \label{def:progfun_2}
    \q \vp^2_{n,l}(kt) := \vp_n^{ \lad^{\scriptscriptstyle 2}_{\scriptscriptstyle n,l},2}(kt) = \frac{i\lad_{n,l}^2 \sqrt{n(n+1)}}{z^2_{n,l} t}\jj_n(z^2_{n,l} t)\q \text{for}\  t \in [0,1]\,.
\end{equation}
By estimates \eqref{eq:est_zeros1} and \eqref{eq:est_zeros2}, the zeros $z_{n,l}^i$, $i = 1,2$ have very small imaginary parts when $l$ is large enough (for the case $n = 5$, $\im z_{n,l}^i \sim 10^{-8}$ by numerical simulation, see Figure \ref{fig:zeros}). 
This indicates that $\vp^1_{n,l}$ is almost a real function while $\vp^2_{5,l}$ is almost purely imaginary (for the case $n = 5$, $\im \vp^1_{n,l} \sim 10^{-10}$ and $\re \vp^2_{n,l} \sim 10^{-11}$  by numerical simulation, see Figure \ref{fig:eigenf5}). We plot in Figure \ref{fig:local} the normalized real parts of propagating function $\vp_{n,l}^1(k|x|)$:
\begin{equation*}
    \widetilde{\re \vp_{n,l}^1}(k|x|) = \frac{\re \vp_{n,l}^1(k|x|)}{\max_{0\le |x|\le 1}\re \vp_{n,l}^1(k|x|)}\,,
\end{equation*}
and the normalized imaginary parts of propagating function $\vp_{n,l}^2(k|x|)$:
\begin{equation*}
     \widetilde{\im \vp_{n,l}^2}(k|x|) = \frac{\im \vp_{n,l}^2(k|x|)}{\max_{0\le |x|\le 1}\re \vp_{n,l}^2(k|x|)}
\end{equation*}
on a two-dimensional cross-sectional plane of the ball $B(0,1)$ passing through the origin, for $k=1$, $n =5$, and different values of $l$.   And we readily see from Figure \ref{fig:local} that for a fixed $n$, when $l$ tends to infinity, both  $\widetilde{\re \vp_{\sss 5,l}^{\sss 1}}(|x|)$ and $\widetilde{\im \vp_{\sss 5,l}^{\sss 2}}(|x|)$ present a remarkable localization pattern in the sense that they are highly oscillating, essentially distributed in a small neighborhood of the origin and rapidly attenuated towards the boundary.

\begin{figure}[!htbp]
    \centering
    \subfigure[$l = 1,5,20,50$ ($\re z^1_{5,l} = 
11.6952, 24.7230, 75.2638, 216.7232$) from left to right.]{
      \centering
      \includegraphics[clip,width=0.25\textwidth]{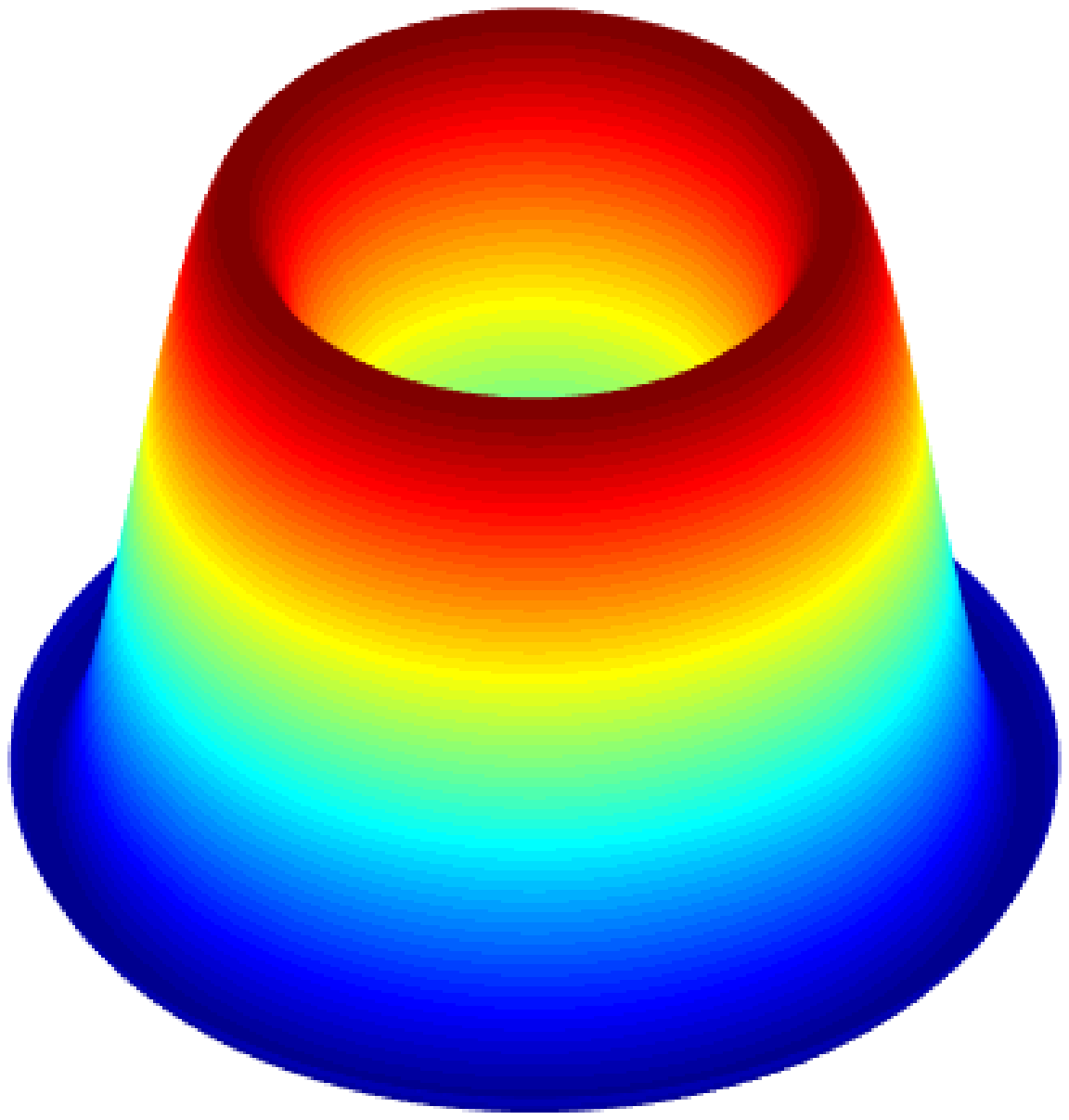}
       \hskip -0.6 cm
\includegraphics[clip,width = 0.25\textwidth]{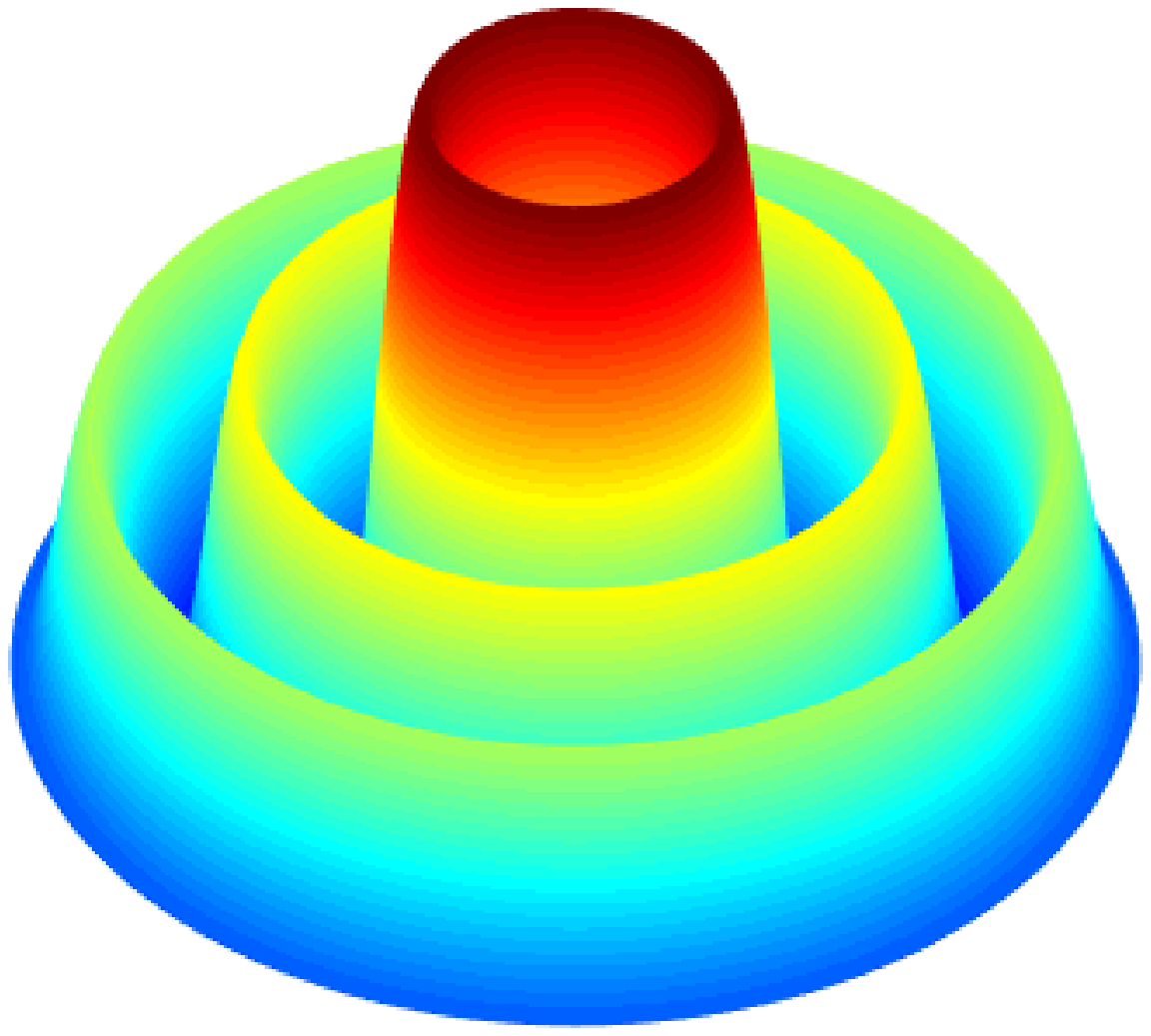}  \hskip -0.6 cm 
\includegraphics[clip,width=0.25\textwidth]{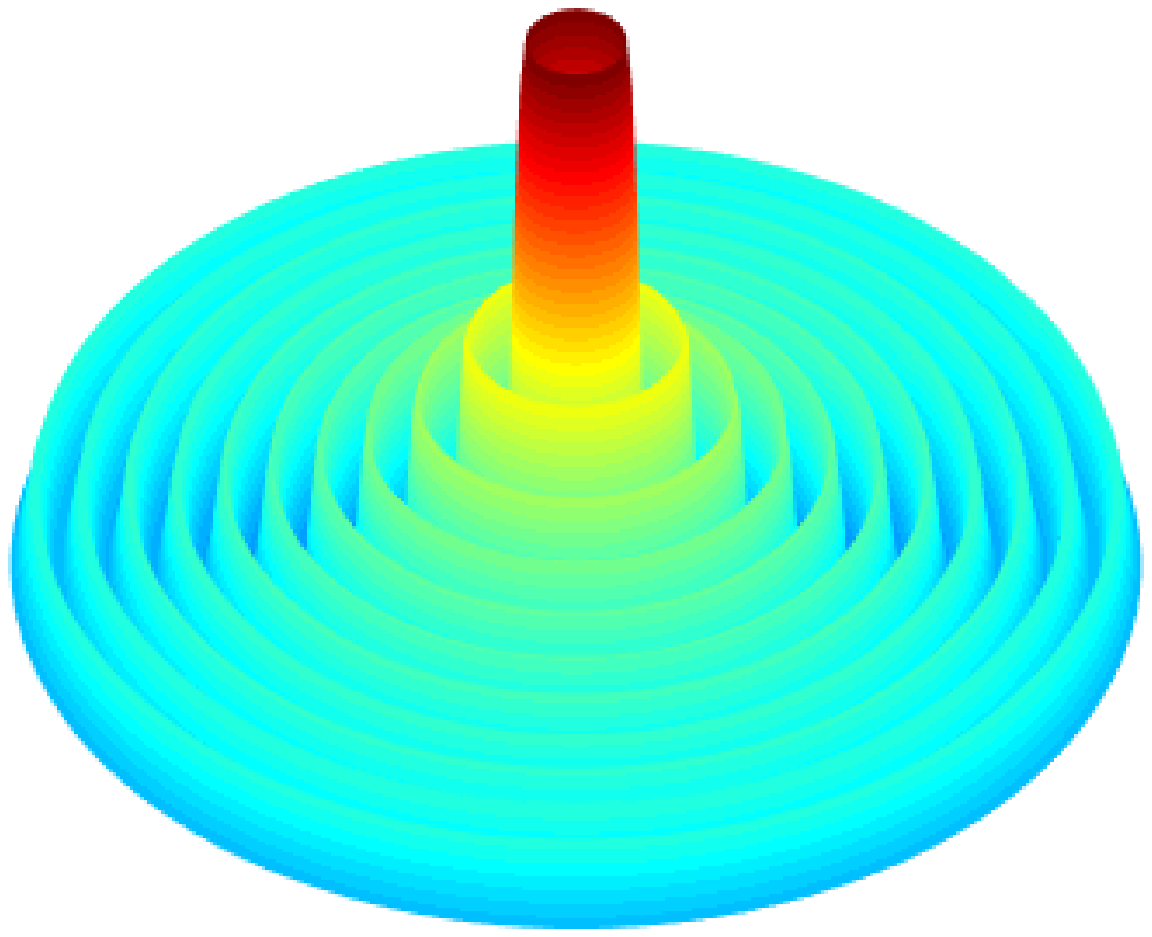}   \hskip -0.6 cm
\includegraphics[clip,width=0.25\textwidth]{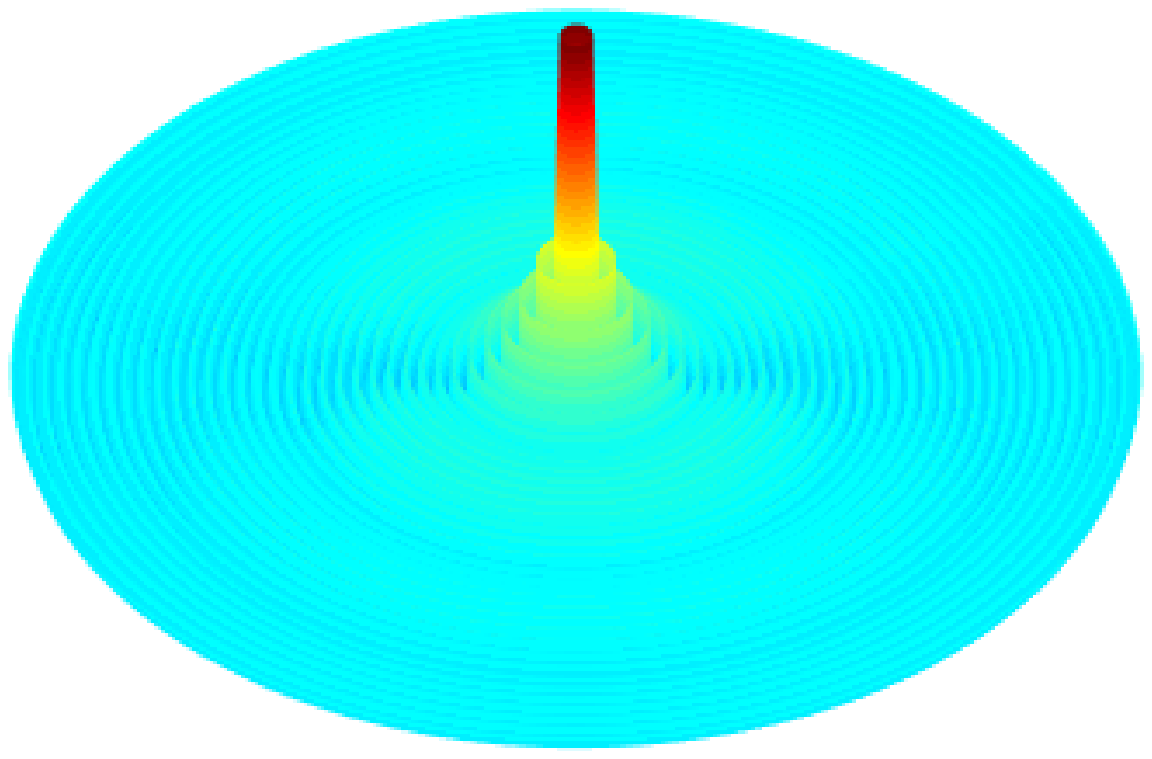}
}
 \\
     \subfigure[$l = 1,5,20,50$ ($ \re z^2_{5,l}=
9.3339, 22.8956, 70.4700, 164.8413$) from left to right.]{
      \centering
      \includegraphics[clip,width=0.25\textwidth]{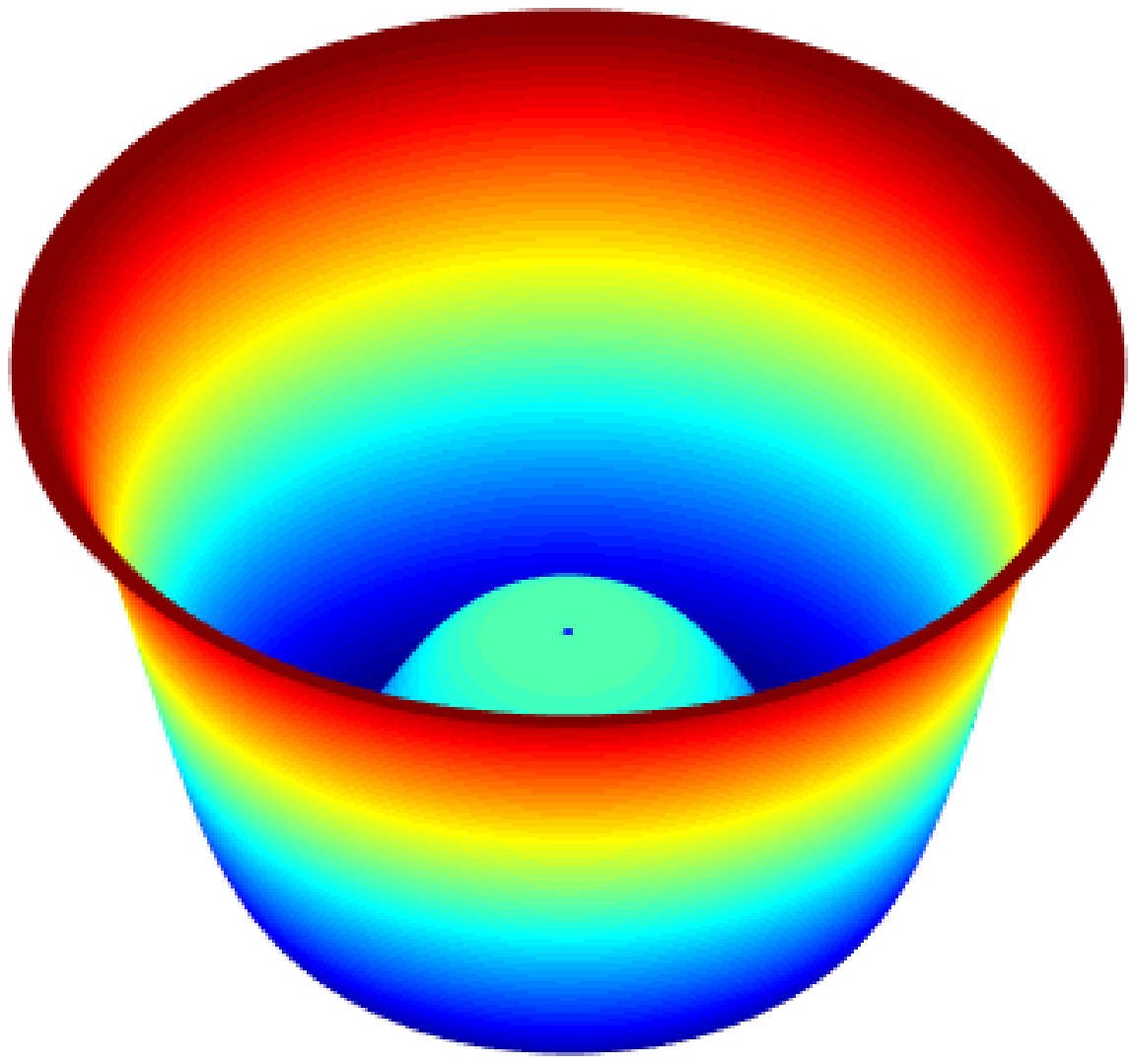}
       \hskip -0.6 cm

     \includegraphics[clip,width = 0.25\textwidth]{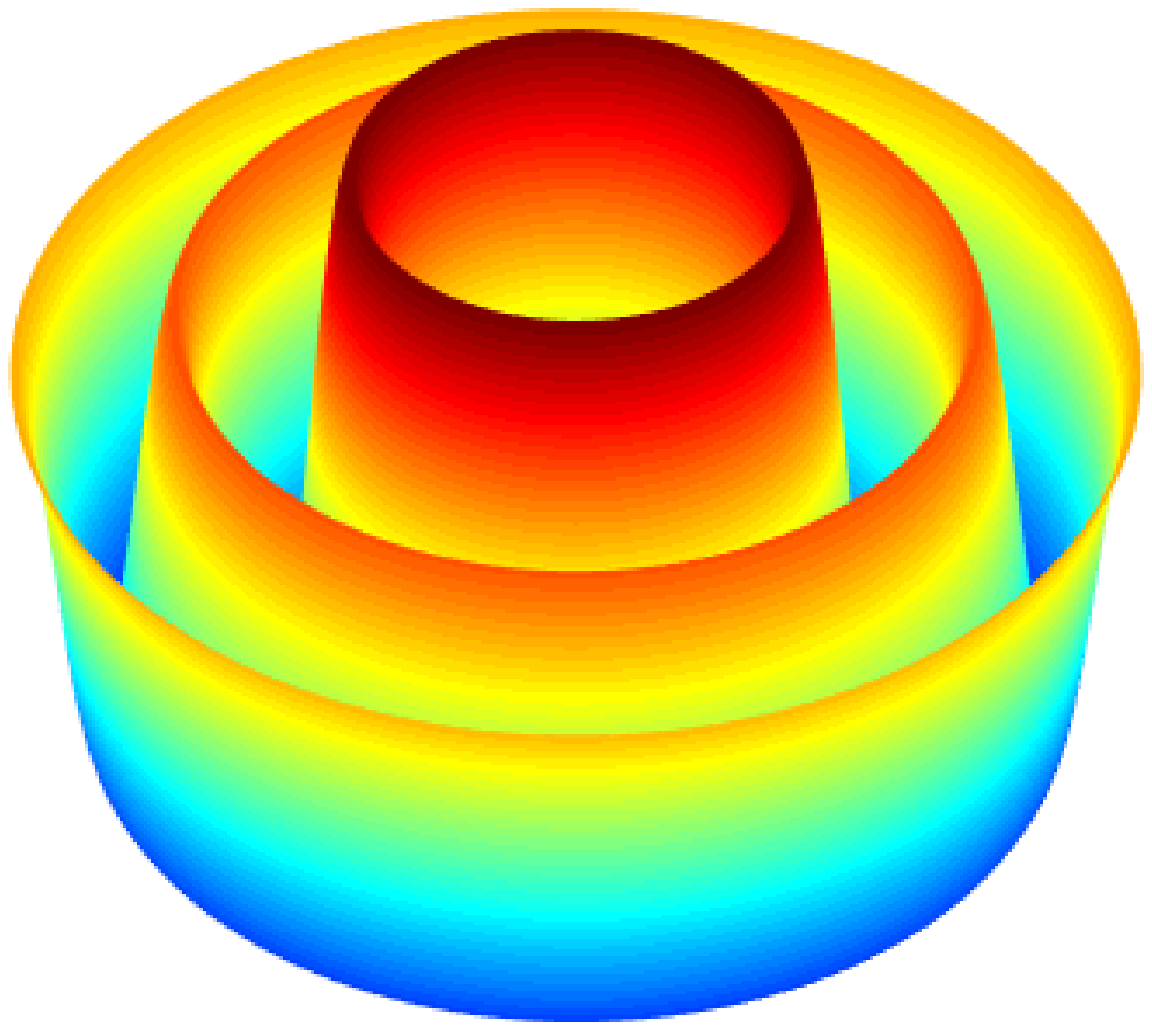} 
 \hskip -0.6 cm
 
\includegraphics[clip,width=0.25\textwidth]{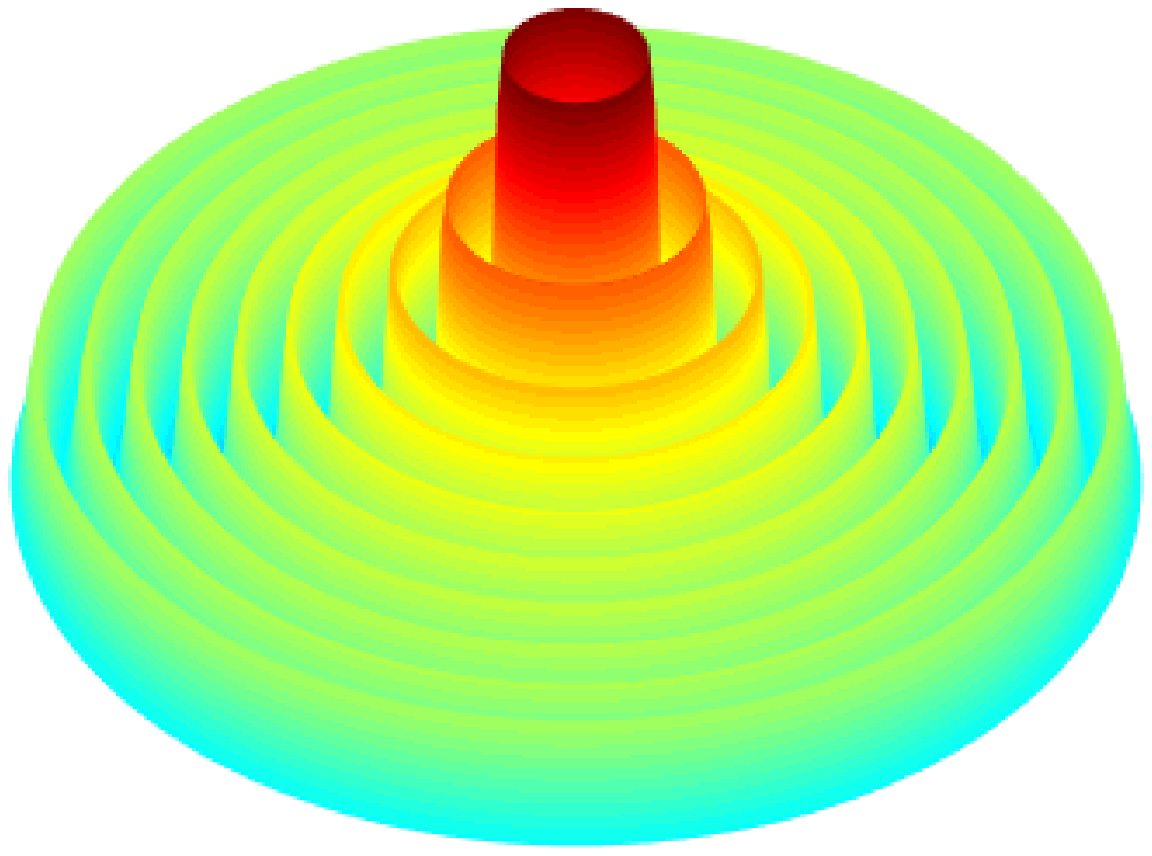} 
 \hskip -0.6 cm
\includegraphics[clip,width=0.25\textwidth]{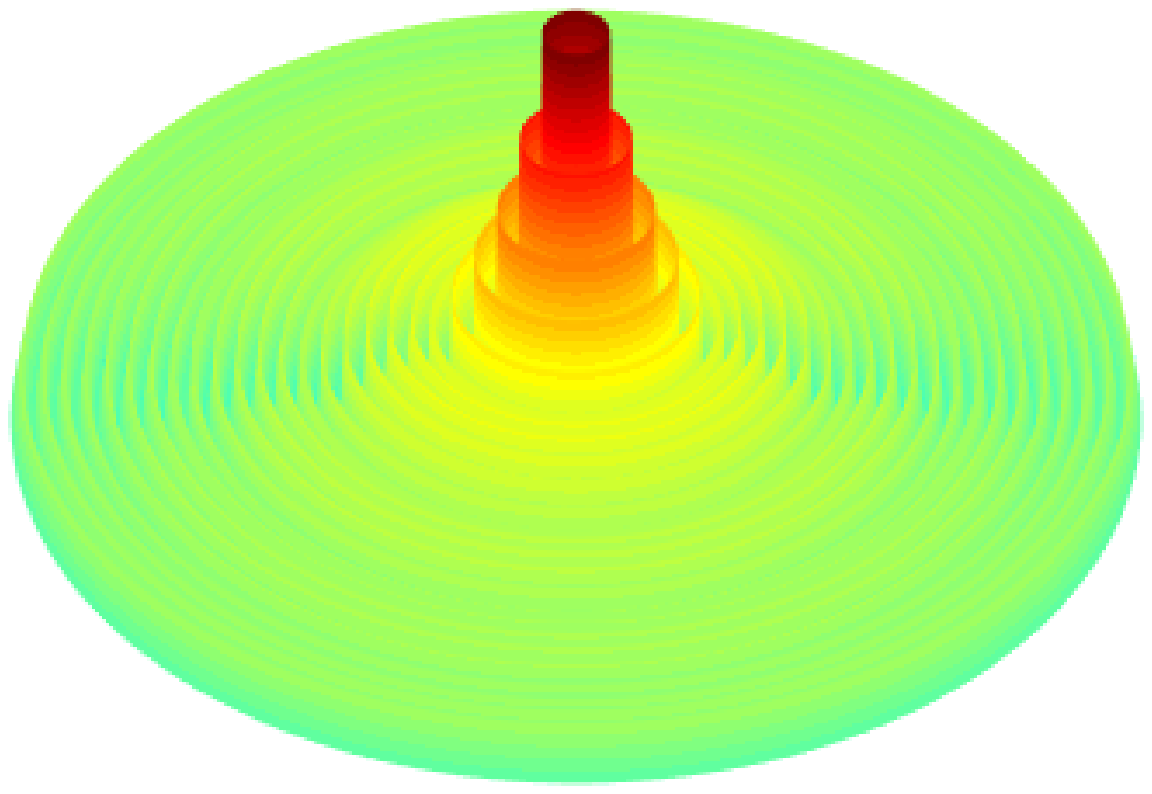}} 
\caption{(a) Normalized real part of  $\vp^1_{5,l}(|x|)$; (b) normalized  imaginary part of $\vp^2_{5,l}(|x|)$ for different values of $l$ on the cross-sectional plane: $|x| \le 1$ with $x_3 = 0$. } \label{fig:local}
\end{figure}

We now give a qualitative mathematical result to illustrate this localization phenomenon.
\begin{theorem}\label{thm:loceig}
Let $\{\vp^i_{n,l}\}$, $i = 1,2$ be the sequences of propagating functions defined by \eqref{def:progfun_1} and \eqref{def:progfun_2}. Then the following asymptotics hold,
\begin{equation} \label{lim;loceig}
    \frac{\max_{t\in [a,1]}|\vp_{n,l}^1(kt)|}{\max_{t\in [0,1]}|\vp_{n,l}^1(kt)|} = O(l^{-1})\,,\q  \frac{\max_{t\in [a,1]}|\vp_{n,l}^2(kt)|}{\max_{t\in [0,1]}|\vp_{n,l}^2(kt)|} = O(l^{-1}) \q \text{as} \ l \to \infty\,,
\end{equation}
\zz{where $a$ is a positive real number from $(0,1)$.}
\end{theorem}
    \begin{proof} 
    The proof is direct and simple based on two lemmas in Appendix \ref{app:C}. We only give the argument for the first estimate in \eqref{lim;loceig}. The analysis for the second one can be conducted by the same idea.  In fact, by Lemma \ref{lem:localforeig} and the asymptotic expansion \eqref{eq:aymjn}, we have
    \begin{align*}
        \frac{\max_{t\in [a,1]}|\vp_{n,l}^1(kt)|}{\max_{t\in [0,1]}|\vp_{n,l}^1(kt)|} =  \frac{\max_{t\in [a,1]}|j_n(z^1_{n,l} t)|}{\max_{t\in [0,1]}|j_n(z^1_{n,l} t)|} \le \frac{C_1|z^1_{n,l}|^{-1}}{\max_{t\in [0,1]}|j_n(\w{z}^1_{n,l}t)|-C_2|l|^{-1}}\,,
    \end{align*}
    where $C_1$ and $C_2$ are some generic constants depending on $n$. Note that letting $l$ tends to infinity, both $\{\w{z}_{n,l}^1\}$ and $\{z_{n,l}^1\}$ vanish with the rate $l^{-1}$. Then the result directly follows from Lemma \ref{lem:localforeig0}.
    \end{proof}
\begin{remark}
It is also possible to obtain more subtle estimates for the localization speed under various $L^p$-norm ($p \ge 1$) in a similar manner as in \cite{nguyen2013localization}, where the authors considered the high-frequency localization of Laplacian eigenfunctions under various boundary conditions and norms. However, the detailed discussions are beyond the scope of this work. We intend to investigate this interesting topic in our future work.
\end{remark}

\section{Applications to super-resolutions in high contrast media}  \label{sec:super-resolu} 
We have established the main mathematical results in this work concerning the spectral properties of $\td$ and the behavior of the resolvent $(\lad - \td)^{-1}$ in the high contrast regime, as well as the asymptotic estimates for the eigenvalues and eigenfunctions for a spherical domain. \zz{In this section,} we shall derive the resonance expansions for the Green's tensor $G$ and its imaginary part $\im G$,  by Theorems \ref{thm:resoinhogr_2} and \ref{thm:ppe}, and use it to explain the expected super-resolution phenomenon when imaging the source $f$ embedded in the high contrast medium. We shall also provide the numerical experiments for the case of a spherical region to show the existence of the possible subwavelength peaks of the imaginary part of the Green's tensor. 
\subsection{Resonance expansion of Green's tensor}
To write the resonance expansion for the Green's tensor $G$, we  directly substitute the pole-pencil decomposition in \eqref{for:poledecom} into the representation of $G$ in \eqref{eq:resoinhogr_2} with a polarization $p \in S^2$ and then obtain
\begin{align}  \label{eq:resexpG}
    G(z,z',k)p = & \frac{1}{k_\tau^2} \nabla_z \ddiv_z (\w{g}(z,z',k)p) + \frac{1}{\tau + 1} \mathbb{P}_0\w{G}(z,z',k)p \notag \\
    & + \frac{1}{\tau}\sum_{i \in I}\sum^{N_i}_{j = 1} \bvi (z) \dd (\tau^{-1} - J^j_{\lad_i})^{-1}(P^j_{\lad_i}\w{G}(\dd,z',k)p)_{\sss\bvi} + (1- \tau T_\zeta^k)^{-1}[P_\zeta \w{G}(\dd,z',k)p](z)\,,
\end{align}
for $z \in D$ and $z'\in D'$; see Theorem \ref{thm:resoinhogr_2} for the definitions of $\w{g}$ and $\w{G}$ here. To derive the resonance expansion of $\im G$, we first recall the explicit form of $\mathbb{P}_0$: $-\na \S \ddiv$ and formula \eqref{def:gp}, and then have
\begin{align}
  \im \mathbb{P}_0\w{G}(z,z',k)p = \mathbb{P}_0\im \w{G}(z,z',k)p & = -\na_z \S \ddiv_z\im G_0(z,z_0,k)p + \na_z \S \ddiv_z \frac{1}{k^2} \nabla_z \ddiv_z \im \w{g}(z,z_0,k)p \notag\\
  & = - \frac{1}{k^2} \na_z \ddiv_z (\im \w{g}(z,z_0,k)p)\,, \label{eq:aux:resoim}
\end{align}
by noting that $\ddiv_z \im G_0(z,z_0,k)p = 0$ and $\S$ is the inverse of $-\Delta$ in the
variational sense (cf.\,\eqref{mod:dsp}). In view of \eqref{eq:aux:resoim}, taking imaginary part of both sides of \eqref{eq:resexpG} gives us the following resonance expansion of $\im G$:
\begin{align}  \label{eq:resoexphc}
    \im G(z,z',k)p = &\im \frac{1}{\tau}\sum_{i \in I}\sum^{N_i}_{j = 1} \bvi (z) \dd (\tau^{-1} - J^j_{\lad_i})^{-1}(P^j_{\lad_i}\w{G}(\dd,z',k)p)_{\sss\bvi} \notag\\ &+ \im (1- \tau T_\zeta^k)^{-1}[P_\zeta \w{G}(\dd,z',k)p](z)\,,\q z\in D,\ z' \in D'\,,
\end{align}
which has a more concise expression than \eqref{eq:resexpG}. Note that the counterpart of the expansion \eqref{eq:resoexphc} for the imaginary part of the free space Green's tensor $\im G_0$ can be derived from \eqref{def:gp} and \eqref{eq:aux:resoim}:
\begin{align} 
    \im G_0(z,z',k)p & = \frac{1}{k^2}\na \ddiv (\im \w{g}(z,z',k)p) + \im (\mathbb{P}_0 + P_\sigma + P_\xi) \w{G}(z,z',k)p \notag\\
    & = \im \sum_{i \in I}\sum^{N_i}_{j = 1} \bvi (z) \dd (P^j_{\lad_i}\w{G}(\dd,z',k)p)_{\sss\bvi} + \im  P_\xi \w{G}(z,z',k)p\,,\q z\in D,\ z' \in D'\,, \label{eq:expimg}
\end{align}
where we have used the fact that $\mathbb{P}_0 + P_\sigma + P_\xi$ is the identity operator, and the definition of $P_\sigma$ in \eqref{def:rp} and the expression \eqref{eq:expbasis}. The first term in 
the above expansion may be viewed as the high-frequency part of $\im G_0$ that can encode the subwavelength information of the sources due to the super-oscillatory nature of the generalized eigenfunctions in the Jordan chains $\bvi$, see Figures \ref{fig:eigenf9} and \ref{fig:local}. Comparing it with \eqref{eq:resoexphc}, we can find that this high-frequency part is amplified by the resolvents of Jordan matrices: $(1-\tau J^j_{\lad_i})^{-1}$ when $\tau^{-1}$ is approaching the eigenvalues $\lad_i, \ i \in I$. Therefore, the imaginary part of $G$ may display a sharper peak than the one of $G_0$ for some specified high contrast parameters, and thus help us more accurately resolve subwavelength details. 

\subsection{Numerical illustrations}  
In this subsection, we numerically study the imaginary part of the Green's tensor $G(x,y,k)$ \zz{corresponding to the spherical medium $B(0,1)$ with the 
high contrast $\tau$, as a complement of the analysis and the illustration for the super-resolution provided in the previous subsection.}
For the sake of simplicity, we let $y = 0$ \zz{and write $G(x,k)$ (resp., $G_0(x,k)$) for $G(x,0,k)$ (resp., $G_0(x,0,k)$).} By the addition formula in \eqref{eq:additiongreen}  for $G_0$ and noting that $\wete(k,0) = 0$ for $n \ge 1$ and $\wetm(k,0) = 0$ for $n \ge 2$, we have 
\begin{equation*}
    G_0(x,k) = \frac{ik}{2}\sum_{m = -1}^1 E_m(k,x) \otimes \overline{\w{E}_m}(k,0)\,, \q x \in \R^3 \backslash \{0\}\,.
\end{equation*}
Here and throughout this subsection, we simply denote $E^{TM}_{1,m}$
 (resp., $\w{E}^{TM}_{n,m}$) by  $E_m(k,x)$ (resp., $\w{E}_m$), for $m =-1,0,1$. As in Section \ref{subsec:sphere}, via the vector wave functions, we assume that the Green's tensor $G$ with a real polarization vector $p \in \R^3$ has the following ansatz: 
\begin{equation} \label{eq:ansaimg_1}
    G(x,k)p = \left\{ \begin{array}{ll}
    G_0(x,k_\tau)p +  \sum^1_{m = -1} a_m \w{E}_m(k_\tau,x)  &  |x| \le 1\,,  \\
     \sum^1_{m = -1} b_m E_m(k, x)     &  |x| \ge 1\,,
    \end{array}\right.
\end{equation}
where $a_m$ and $b_m$ for $m = -1,0,1$ are complex constants to be determined
and linearly depending on $p$. To proceed, we note that, from \eqref{eq:trramp} and \eqref{eq:trenmp}, it follows that 
\begin{equation} \label{eq:trag0_1}
\begin{cases}
\h{x} \t G_0(x,k_\tau) p =  -\frac{1}{\sqrt{2}|x|}  \hh_1(k_\tau|x|) \sum_{m = -1}^1 V_1^m(\h{x}) \overline{\w{E}_m}(k_\tau,0)^t\dd p\,, \q x \in \R^3 \backslash \{0\}\,,   \\
  \h{x} \t \curl G_0(x,k_\tau)p = -\frac{k_\tau^2}{\sqrt{2}} h_1^{(1)}(k_\tau|x|)\sum_{m = -1}^1 U_1^m(\h{x}) \overline{\w{E}_m}(k_\tau,0)^t \dd p\,, \q x \in \R^3 \backslash \{0\}\,.
\end{cases}
\end{equation}
To avoid calculating the three coefficients $a_m (m=-1,0,1)$, we choose a special real polarization vector $p$:
\begin{equation} 
    p = \frac{\w{p}}{\norm{\scriptstyle \w{E}_0(k_\tau,0)}^2}\in \R^3\,,\q   \w{p} = \w{E}_0(k_\tau,0)/i\,,
\end{equation}
\zz{according to} two easily verified observations that $\w{E}_m(k,0)$, $m = -1,0,1$ are  orthogonal vectors with the same $l^2$-norms (cf.\,\eqref{eq:tancomm_2}), and $\w{E}_0(k_\tau,0)$ has purely imaginary components since $Y_1^0(\h{x})$ is a real vector function on $S^2$. With this specially chosen $p$, we can simplify \eqref{eq:trag0_1} as follows:
\begin{equation}\label{eq:trag0_2}
\begin{cases} 
  \h{x} \t G_0(x,k_\tau) p =  \frac{i}{\sqrt{2}|x|}  \hh_1(k_\tau|x|)  V_1^0(\h{x}) \,, \q x \in \R^3 \backslash \{0\}\,,  \\
   \h{x} \t \curl G_0(x,k_\tau)p = \frac{i k_\tau^2}{\sqrt{2}}h_1^{(1)}(k_\tau|x|)  U_1^0(\h{x}) \,, \q x \in \R^3 \backslash \{0\}\,.
\end{cases}
\end{equation}
 Matching the Cauchy data of the field in \eqref{eq:ansaimg_1} inside and outside the domain on the boundary $\p B(0,1)$, we obtain, by using \eqref{eq:trramp} and \eqref{eq:trenmp}, that $a_{-1} = a_1 = 0$ and $b_{-1} = b_1 = 0$, and the following equation for $(a_0,b_0)$:
\begin{equation*} 
\mm \frac{1}{ik_\tau}\jj_1(k_\tau) & - \frac{1}{i k}\hh_1(k) \\ -ik_\tau j_1(k_\tau) & i k h_1^{(1)}(k)
\nn \mm a_0\\b_0 \nn = \mm \frac{i}{2} \hh_1(k_\tau)   \\   \frac{i k_\tau^2}{2}  h_1^{(1)}(k_\tau) \nn.
\end{equation*} 
Then the solution $a_0$ to the above equation readily follows (we only need $a_0$ to investigate the behavior of $G$ inside the domain):
\begin{align*}
    a_0 
     = \frac{ -\frac{k^2}{2k_\tau}\hh_1(k_\tau)h_1^{(1)}(k) + \frac{k_\tau}{2}\hh_1(k)h_1^{(1)}(k_\tau)  }{\frac{k^2}{k^2_\tau}\jj_1(k_\tau)h_1^{(1)}(k) -j_1(k_\tau)\hh_1(k)}\,.
\end{align*}
We regard $a_0$ as a function of the real variable $k_\tau$ and plot its absolute value in Figure \ref{fig:abscoeff} for $k = 1$, from which we clearly see that it blows up when $k_\tau$ hits the real parts of the discrete zeros $z^2_{1,l}$ of $f^2_n(z)$.
\begin{figure}[!htbp]
    \centering
      \includegraphics[clip,width=0.5\textwidth]{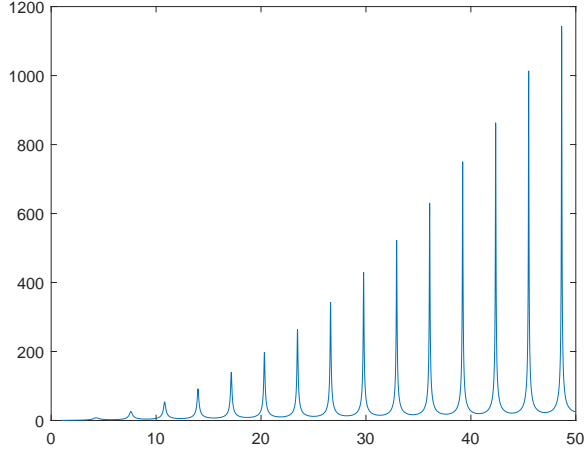}
      \caption{$|a_0(k_\tau)|$ as a function of $k_\tau$, $k_\tau \in [1,50]$.}
    \label{fig:abscoeff}
\end{figure}

Since the spherical harmonics has nothing to do with the contrast $\tau$, in the following, we shall pay attention to the imaginary part of the radial part:
\begin{equation*}
    \phi(k_\tau,t) = \frac{i}{\sqrt{2}t}\hh_1(k_\tau t) - a_0 \frac{\sqrt{2}}{ik_\tau t} \jj_n(k_\tau t), \quad t \in[-1,1]\,,
\end{equation*}
of the tangential component $\h{x} \t G(x,k)p$ of $G(x,k)p$:
\begin{equation*}
    \h{x} \t G(x,k)p = \frac{i}{\sqrt{2}|x|}\hh_1(k_\tau |x|) V_1^0(\h{x}) - a_0 \frac{\sqrt{2}}{ik_\tau |x|} \jj_n(k_\tau |x|) V_1^0(\h{x})\,.
\end{equation*}
We remark that $\phi(k_\tau,t)$ is a one-dimensional function but keeping all the main features of $\im G(x,k)p$ we are interested in; and the radial part of the normal component $\h{x} \dd G(x,k)p$ has a very similar behavior as $\phi(k_\tau,t)$. 
\begin{figure}[!htbp]
    \centering 
    \subfigure[$k_\tau =1, \re(z^2_{1,l})$ for $l = 2,3,4,5$, i.e., $k_\tau$ =1, 7.5944, 10.8119,   13.9949, 17.1626.]{\label{fig:imaggre_1}
       \includegraphics[clip,width=0.5\textwidth]{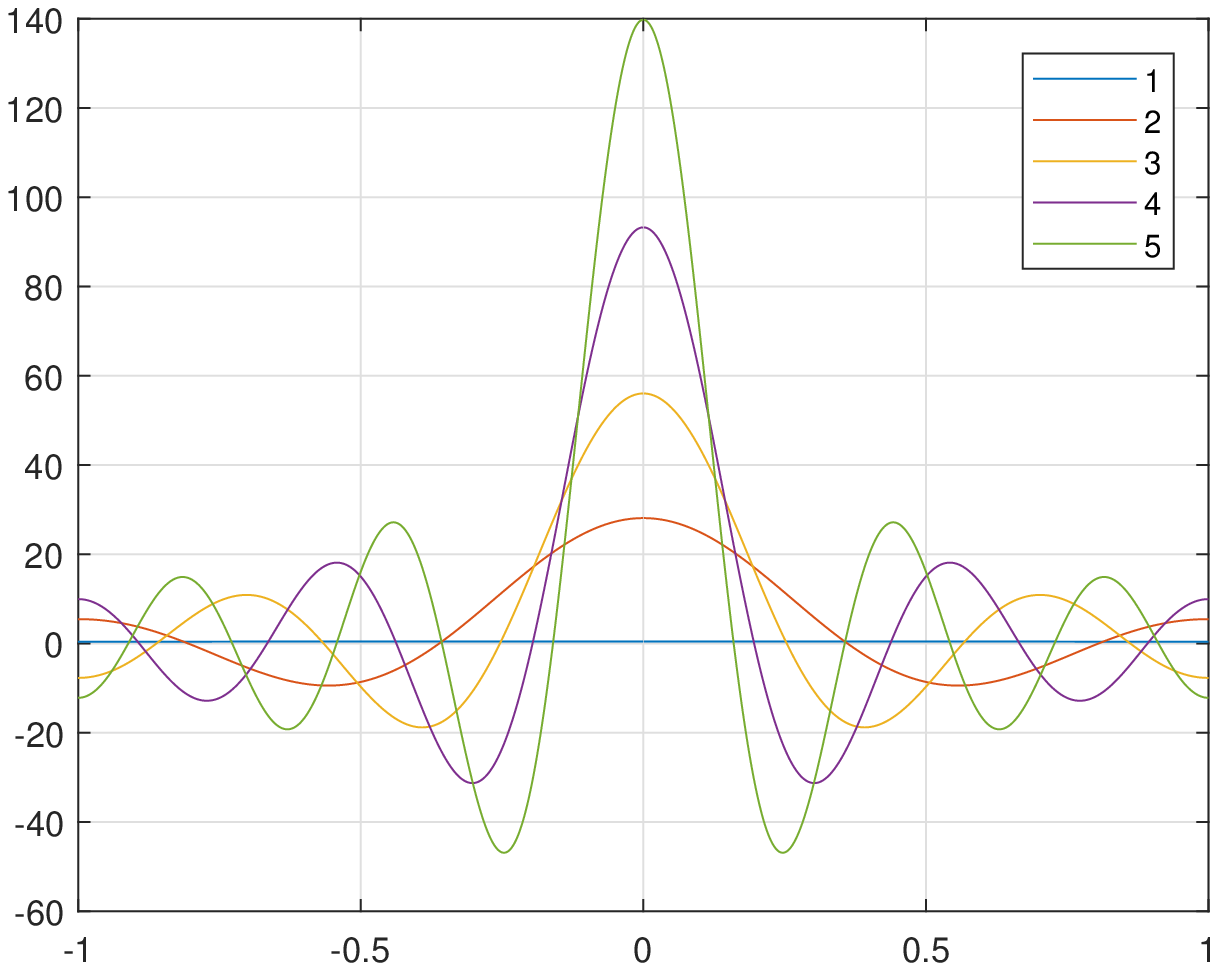}}  \hskip -0.6 cm
        \subfigure[$k_\tau$ =1, 15, 25.]{\label{fig:imaggre_2}
       \includegraphics[clip,width=0.5\textwidth]{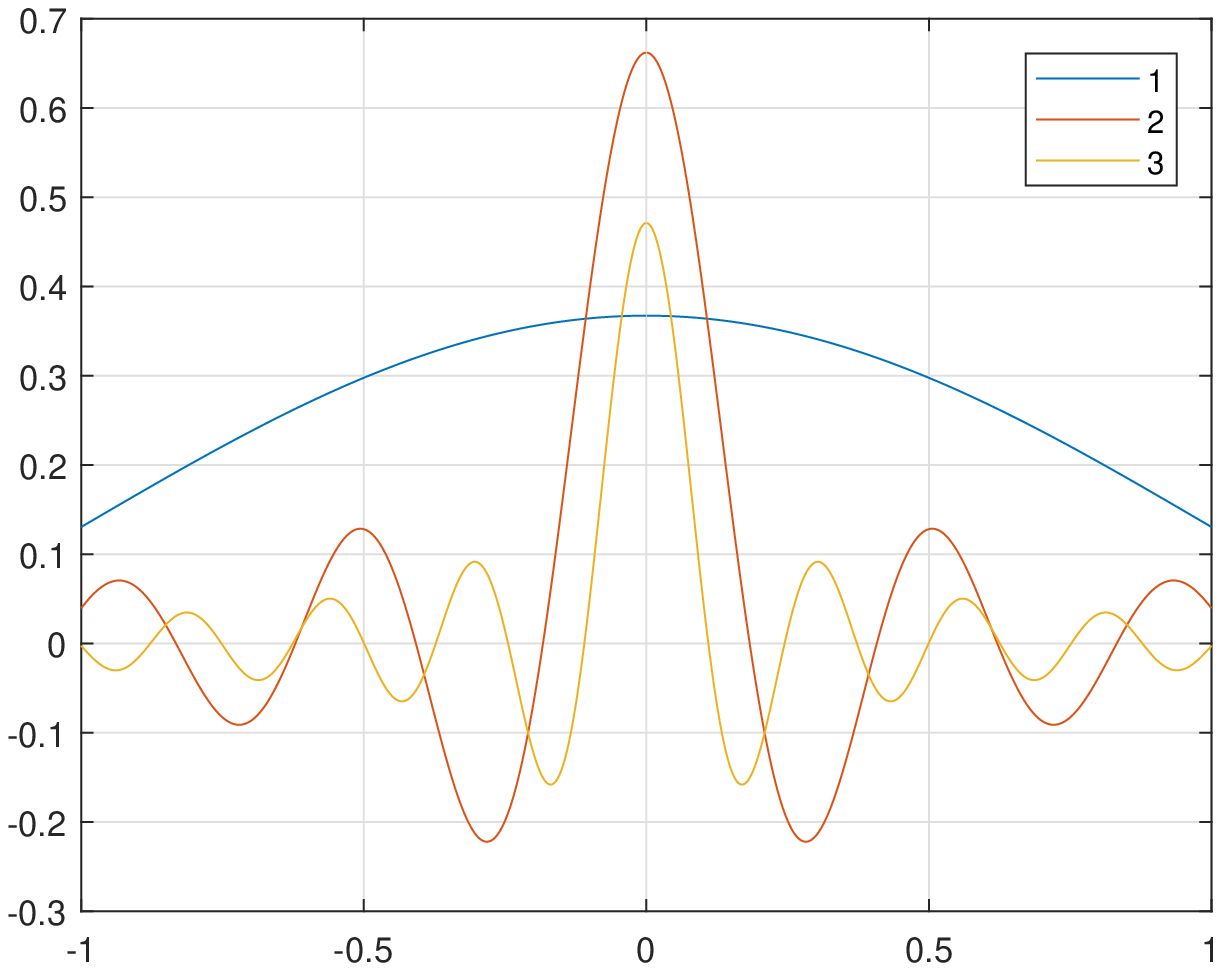}} 
      \caption{Imaginary part of  $\phi(k_\tau,t)$ for various $k_\tau$.}
    \label{fig:imaggre}
\end{figure}

From Figure \ref{fig:imaggre_1}, where we present $\im \phi(k_\tau,t)$ for different values of $k_\tau$,  we see that when $k_\tau$ increases and hits the real parts of $z_{1,l}^2$, the imaginary part of Green's tensor become highly oscillating and exhibit a subwavelength peak, and hence the super-resolution can be achieved \zz{with the increasing likelihood. When $\tau$ tends to infinity, we can even expect the infinite resolvability of the imaging system, by Theorems \ref{thm:asyeigenvalue} and \ref{thm:loceig}.} However, we would like to stress that the super-resolution phenomenon can only be expected for discrete values of $\tau$. For those $\tau$ taking high values but not near the resonant values, the magnitude of $\im G(x,k)p$ will not be significantly enhanced and have almost the same order of $\im G_0(x,k)p$, although it is more oscillatory than the one in the homogeneous space; see Figure \ref{fig:imaggre_2}.

\section{Concluding remarks}
In this work, we have considered the time-reversal reconstruction of EM sources embedded in an inhomogeneous background, and tied its anisotropic resolution to the resolvent of a certain type of integral operators $\td$ via a newly derived
Lippmann-Schwinger representation that reveals the close relation between the medium (shape and refractive indices) 
and its associated EM Green's tensor. We have then investigated the spectral structure of $\td$ for a bounded smooth domain with a very general geometry and found that all the poles of its resolvent in $\C\backslash \sigma_{ess}(\td)$ are eigenvalues of finite type and lie in the upper-half plane \zz{with $\sigma_{ess}(\td)$ being all its possible accumulation points}. With these new findings, we have derived the pole-decomposition for the resolvent of $\td$ and \zz{obtained the local resonance expansion for the Green's tensor associated with the high contrast medium}. More quantitative results about the asymptotic behaviors of eigenvalues and eigenfunctions have been also provided for the case of a spherical domain. \zz{As a by-product of our spectral analysis, we have given a characterization and discussion about the EM nonradiating sources, see Remarks \ref{rem:nonradi_1} and \ref{rem:eig-1ball}. Some further interesting spectral results about the operator $\td$ based on the fact that $\td$ is a quasi-Hermitian operator have been included in Appendix \ref{app:B}.}
\zz{In Section \ref{sec:super-resolu},} we have applied our new theoretical results to explain the expected super-resolution in the inverse electromagnetic source problem at some discrete characteristic values. \zz{It turns out that both eigenvalues and eigenfunctions are responsible for the super-resolution phenomenon in the sense that the eigenfunctions are super-oscillatory and can encode the subwavelength information of the sources; while the eigenvalues serve as an amplifier when they nearly hit the reciprocal of the contrast so that these subwavelength information can be measurable in the far field. We finally remark that our analysis and results can be naturally extended to the Lipschitz domain by noting the facts that the Helmholtz decomposition in Appendix \ref{app:A} still holds \cite{amrouche1998vector} and that for a selfadjoint operator on a Hilbert space, the essential spectrum is a compact
subset of the real line \cite{gohberg1990classes}.}

\titleformat{\section}{\bfseries}{\appendixname~\thesection .}{0.5em}{}
\titleformat{\subsection}{\normalfont\itshape}{\thesubsection.}{0.5em}{}
\appendices

\section{Helmholtz decomposition of \texorpdfstring{$L^2$-vector fields}{}}\label{app:A}
In this section we give a complete review of the Helmholtz decomposition of $L^2$-vector fields in a unified manner due to its great significance to our main analysis in the work. For a vector field $u$, the Helmholtz decomposition provides us a procedure to separate its divergence, curl, and the normal trace information. In the following, we show how to extract these information from a field $u$ by solving some sub-variational problems. Let us first give a more precise description about the geometry of the domain $D$. We denote by $\Gamma_j\,,0\le j \le J$, the connected component of $\p D$, in which $\Gamma_0$ is the boundary of the unbounded connected component of $\R^3 \backslash \bar{D}$. And the genus $L$ of $\p D$ may be nontrivial, i.e., $L \ge 0$ (for $L \ge 1$, we can construct interior cuts: $\Sigma_i\,,1\le i \le L$ contained in $D$ such that $D\backslash\cup^L_{i=1}\Sigma_i$ is simple connected; see \cite[Section 3.7]{monk2003finite}).
 A typical example of $D$ with $L = 1$ and $J=1$ is a torus with a ball hole.

Denote by $\S: H^{-1}(D) \to H_0^1(D)$ the solution operator of the Dirichlet source problem, namely, for $l \in H^{-1}(D)$, $\S l \in H_0^1(D)$ solves the variational problem:
\begin{equation} \label{mod:dsp}
   \text{Find\ }\psi\in H_0^1(D)\ \text{such that}\  \l l, \varphi \r_{H_0^1(D)} = (\nabla \psi, \nabla \varphi)_{L^2(D)}\,, \quad \forall\  \varphi \in H^1_0(D)\,.
\end{equation}
We remark that $\S$ is an isomorphism between $H^{-1}(D)$ and $H_0^1(D)$. Note that $\ddiv :  L^2(D,\R^3) \to H^{-1}(D)$ is the adjoint operator of $-\nabla: H_0^1(D) \to  L^2(D,\R^3)$. For $u \in L^2(D,\R^3)$, we consider \eqref{mod:dsp} with 
\begin{equation*}
    \l l, \varphi \r_{H_0^1(D)}: = (u,\nabla \varphi)_{L^2(D)}, \quad \forall\ \varphi \in H^1_0(D)\,.
\end{equation*}
Then there exists a unique solution $\psi_1 := - \S \ddiv u \in H^1_0(D)$ satisfying \eqref{mod:dsp}, from which it follows that  $u - \na \psi_1$ is divergence-free in the distribution sense, and the normal trace $\gamma_n$ is well-defined. 

To obtain the $\ccurl$ part of $u$, we need to solve a magnetostatics problem. To do so, we introduce the Hilbert space $X_N := H_0(\ccurl,D) \bigcap H(\ddiv, D)$ 
with the graph norm $\norm{\dd}_{X_N}:=\norm{\dd}_{L^2(D)} + \norm{\ddiv \dd}_{L^2(D)} + \norm{\ccurl \dd}_{L^2(D)}$, and its subspace $X_N^0 := H_0(\ccurl,D) \bigcap H(\ddiv0, D)$. By the well-known de Rham diagram (cf.\,\cite[Section 3.7]{monk2003finite}), we see that the kernel space of the $\ccurl$ operator in $H_0(\ccurl,D)$, i.e., $H_0(\ccurl 0,D)$, has the following orthogonal decomposition:
\begin{equation} \label{eq:kerofcurl}
    H_0(\ccurl 0,D) = \na H^1_0(D) \oplus_\perp K_N(D),
\end{equation}
where $K_N(D)$ is the normal cohomology space
with the dimension $J$, given by 
\begin{equation*}
    K_N(D)=\{u \in H_0(\ccurl,D)\,;\ \curl u = 0,\ \ddiv u = 0\ \text{in} \ D\}.   
\end{equation*}
Moreover, we have the following characterization of $K_N(D)$ from \cite[Theorem 3.42]{monk2003finite}.
\begin{lemma} \label{lem:charkn}
$K_N(D)$ is spanned by $\na p_j$, $1\le j \le J$, where $p_j \in H^1(D)$ satisfies
\begin{equation*}
    \Delta p_j = 0\ \text{in}\ D,\q  \text{and} \ \  p_j = \d_{j,s} \q \text{on} \ \Gamma_s, \ 0 \le s \le J.
\end{equation*}
In addition $\l\frac{\p p_j}{\p \n},1\r_{H^{1/2}(\Gamma_s)} = \d_{j,s}$, $1\le j \le J$, and $\l\frac{\p p_j}{\p \n}, 1\r_{H^{1/2}(\Gamma_0)} = -1$. 
\end{lemma}

By Friedrich's inequality (cf.\,\cite[Corollary.3.19]{amrouche1998vector}), on the space $X_N$, the seminorm
\begin{equation*}
|\dd|_{X_N}:=\norm{\ccurl \dd}_{L^2(D)} + \norm{\ddiv \dd}_{L^2(D)} + \sum^J_{j=1}\left|\l \gamma_n \dd,1\r_{H^{1/2}(\Gamma_j})\right|    
\end{equation*}
is equivalent to the graph norm $\norm{\dd}_{X_N}$. We now define the following quotient space:
\begin{equation*}
    \w{X}_N := X_N/K_N(D)
\end{equation*}
with the standard quotient norm $\norm{[u]}_{\w{X}_N} := \inf_{v \in K_N(D)}|u+v|_{X_N}$, where $[u] \in \w{X}_N$ denotes the equivalent class of $u$. It is easy to see that the quotient norm has an explicit form:
\begin{equation} \label{eq:equinorm}
    \norm{[u]}_{\w{X}_N} = \norm{\ccurl [u]}_{L^2(D)} + \norm{\ddiv [u]}_{L^2(D)}\,,
\end{equation}
where $\ccurl [u]$ and $\ddiv [u]$ are well-defined. Indeed, we can choose $$v = - \sum^J_{j=1} \l\gamma_n u ,1 \r_{{H^{1/2}(\Gamma_j})} \na p_j \in K_N(D)$$
such that for the representation element $u + v$ of $[u]$, the term $\sum^J_{j=1}|\l \gamma_n \dd,1\r_{{H^{1/2}(\Gamma_j})}|$ vanishes, which directly leads us to \eqref{eq:equinorm}. Moreover, on the subspace $\w{X}^0_N:= X_N^0/K_N(D)$ the quotient norm reduces to $\norm{\ccurl \dd}_{L^2(D)}$. We are now ready to consider the following magnetostatic field problem:   
for $f \in L^2(D,\R^3)$, find $\psi \in \w{X}_N$ such that
\begin{equation} \label{sys:cc}
  \left\{ \begin{array}{lc}
  \ccurl \ccurl \psi = \ccurl f &  \text{in} \ D\,, \\
   \ddiv \psi = 0      &  \text{in} \ D\,, \\
   \n \t \psi = 0 & \text{on} \ \p D\,, 
    \end{array}  \right.
\end{equation}
which shall be seen to have a unique solution. Its variational formulation is given by the next lemma.
\begin{lemma}
The system \eqref{sys:cc} is equivalent to the following variational problem: find $\psi \in \w{X}_N$ such that it holds, for all  $\phi \in \w{X}_N$,  that
\begin{equation} \label{sys:var:cc}
    (f, \ccurl \phi)_{L^2(D)} = (\ccurl \psi, \ccurl \phi)_{L^2(D)} + (\ddiv \psi, \ddiv \phi)_{L^2(D)}\,.
\end{equation}
\end{lemma}

\begin{proof}
If $\psi$ is a solution of \eqref{sys:cc}, by the first equation in \eqref{sys:cc}, then it holds for all $\phi \in H_0(\ccurl,D)$ that
\begin{equation*}
    (f,\ccurl \phi)_{L^2(D)} = (\ccurl \psi, \ccurl \phi)_{L^2(D)}.
\end{equation*}
Therefore, by combining it with the fact that $\ddiv \psi = 0$,  we can directly see that  \eqref{sys:var:cc} holds. Conversely, if \eqref{sys:var:cc} holds, it suffices to prove that $\ddiv \psi = 0$ to conclude the lemma. Recalling \eqref{eq:kerofcurl}, we have 
   \begin{equation} \label{eq:seteqi}
       H_0(\ccurl0,D) \bigcap H(\ddiv,D) = \{\nabla \vp\,;\ \vp \in H_0^1(D) \ \text{with}\ \Delta \vp \in L^2(D)\} \oplus_{\perp} K_N(D).
   \end{equation}
   Denoting the space defined in \eqref{eq:seteqi} by $X$, we then obtain $L^2(D) = \ddiv (X/K_N(D))$ since for all $v\in L^2(D)$, we can find $\varphi \in H_0^1(D)$ such that $\Delta \varphi = v$ in the variational sense. By choosing
 $\phi \in X/K_N(D)$ in \eqref{sys:var:cc}, we readily see $\ddiv \psi = 0$, and hence the proof is complete.
\end{proof}

To show the existence and uniqueness of a solution, we introduce the isomorphism  $\T:\w{X}_N' \to \w{X}_N$ such that for $l \in \w{X}_N'$,  $Tl$ satisfies
\begin{equation*}
   \l l, \phi \r_{\w{X}_N} = (\ccurl \T l, \ccurl \phi)_{L^2(D)} + (\ddiv \T l, \ddiv \phi)_{L^2(D)}\,, \quad \forall\  \varphi \in \w{X}_N\,,
\end{equation*} 
by \eqref{eq:equinorm} and Riesz representation theorem.
We note that $\ccurl$ can be regarded as
 a continuous mapping from $L^2(D,\R^3)$ to $\w{X}_N'$, by setting 
 \begin{equation} \label{def:discurll2}
     \l\ccurl u, \phi\r_{\w{X}_N}: = (u, \ccurl \phi)_{L^2(D)},
 \end{equation}
  which is well-defined since $\ccurl \phi$ is independent of the choice of the representative element of $[\phi]$. Then for $u \in L^2(D,\R^3)$, there is a unique $\psi_2 := \T \ccurl u \in \w{X}_N^0$ solving \eqref{sys:var:cc} or \eqref{sys:cc} with $f = u$. By the above constructions, we can see that the remaining $v$ of $u \in L^2(D,\R^3)$:
\begin{equation} \label{eq:remainvec}
  v :=  u - \nabla \psi_1 - \ccurl \psi_2 = u + \nabla \S \ddiv u - \ccurl \T\ccurl u \in L^2(D,\R^3)\,,
\end{equation}
is an irrational and divergence-free vector field, i.e., $\ddiv v = \ccurl v = 0$.  

The last step regarding the normal trace is relatively simple by noting the fact that the restriction of normal trace mapping $\w{\gamma}_n: = \gamma_n|_W$ on $W$ is 
an isomorphism from $W$ to $H_0^{-1/2}(\p D)$. To be precise, for $\phi \in H_0^{-1/2}(\p D)$, $\wg \phi$ is the gradient, which is unique, of a solution to the following Neumann problem:
\begin{equation*}
   \left\{ \begin{array}{ll}
        \Delta p = 0 &  \ \text{in}\ D\,,\\
        \frac{\p p}{\p \n} = \phi & \ \text{on}\ \p D \,.
    \end{array} \right.
\end{equation*}
By setting $\phi = \gamma_nv$, where $v$ is introduced in \eqref{eq:remainvec}, we can find an element $\wg\gamma_n v$ from $W$ to characterize the normal trace information of $v$ (and also $u$).

However, after we remove the divergence, curl and normal trace component of $u$, the remaining part: 
\begin{equation*}
    u - \nabla \psi_1 - \ccurl \psi_2 - \wg \gamma_n v 
\end{equation*}
is still nontrivial if the genus $L \ge 1$, and it is located in the so-called tangential cohomology space $K_T(D)$, defined  by
\begin{equation*}
K_T(D)=\{u \in H_0(\ddiv,D)\,;\ \curl u = 0,\ \ddiv u = 0\ \text{in} \ D\}\,,
\end{equation*}
which has dimension $L$. We remark that there exists a similar characterization as in Lemma \ref{lem:charkn} for $K_T(D)$. We now summarize
the above constructions in the following result, where the $L^2$-orthogonal relation can be verified directly. 
\begin{theorem} \label{lem:decoofl2}
$L^2(D,\R^3)$ has the following $L^2$-orthogonal decomposition:
\begin{align*}
    L^2(D,\R^3) 
    = \nabla H^1_0(D) \oplus_\perp \ccurl \w{X}_N^0\oplus_\perp W \oplus_\perp K_T(D)\,,
\end{align*}
where $\nabla H^1_0(D)$, $\ccurl \w{X}_N^0$, and $W$ are uniquely determined by $\ddiv u$, $\ccurl u$, and $\gamma_n(u + \na \S \ddiv u)$, respectively. Here, the operator $\S$ is given by \eqref{mod:dsp}.
\end{theorem}

\if \commentflag = \ct
For a vector field $u$, the Helmholtz decomposition provides us a procedure to extract its divergence information, curl information and the normal trace part. We next show how to construct the $\nabla H^1_0(D)$ and $W$ components by utilizing $\ddiv u$ and $\gamma_n u$ for a $L^2$-field $u$ for our subsequent use. The structure of space $\hzz$ is a little bit more complicated than the aforementioned two parts, which can be uniquely determined by the $\ccurl u$ on a quotient space module over a finite-dimensional cohomology subspace. We do not intend to give a comprehensive illustration of this point here since it won't be used directly, although it is very important. Denote by $\S: H^{-1}(D) \to H_0^1(D)$, which is also an isomorphism, the solution operator of Dirichlet source problem, namely, for $l \in H^{-1}(D)$, $\S l \in H_0^1(D)$ solve the following variational problem:
\begin{equation} \label{mod:dsp}
   \text{find\ }\psi\in H_0^1(D)\ \text{such that}\  \l l, \varphi \r_{H_0^1} = (\nabla \psi, \nabla \varphi)_{L^2(D)}, \quad \forall \varphi \in H^1_0(D)\,.
\end{equation}
Note that $\ddiv :  L^2(D,\R^3) \to H^{-1}(D)$ is the adjoint operator of $-\nabla: H_0^1(D) \to  L^2(D,\R^3)$. For $u \in L^2(D,\R^3)$, we consider \eqref{mod:dsp} with 
\begin{equation*}
    \l l, \varphi \r_{H_0^1}: = (u,\nabla \varphi)_{L^2(D)}, \quad \forall \varphi \in H^1_0(D).
\end{equation*}
Then there exists a unique solution $\psi_1 := - \S \ddiv u \in H^1_0(D)$ satisfying \eqref{mod:dsp}, from which it follows that  $u - \na \psi_1$ is divergence-free in the distribution sense, and the normal trace $\gamma_n$ is well-defined. To find the $W$ component, we first note that the restriction of normal trace operator $\gamma_n$ on $W$, denoted by $\w{\gamma}_n$, is an isomorphism from $W$ to $\hbb$. More precisely, for $\phi \in \hbb$, $\wg \phi$ is the gradient of a solution to the following Neumann problem:
\begin{equation} \label{eq:neumannpw}
   \left\{ \begin{array}{ll}
        \Delta p = 0 &  \ \text{in}\ D\\
        \frac{\p p}{\p \n} = \phi & \ \text{on}\ \p D\,.
    \end{array} \right.
\end{equation}
By setting $\phi = \gamma_n(u-\na \psi_1)$ in \eqref{eq:neumannpw}, we then have $\pw u = \na p$, which implies $u - \na \psi_1 - \na p \in \hzz$. 
\fi

\section{\texorpdfstring{$\td$}{} as a quasi-Hermitian operator} \label{app:B}
\subsection{A global resolvent estimate}  \label{app:B_1}
In this subsection, we provide a resolvent estimate for $(\lad-\td)^{-1}$ on $\rho(\td)$ by applying a general spectral result from \cite{gil2003operator}. To do this, We first introduce some notions. We consider the bounded linear operator $A$ acting on a separable Hilbert space $H$. The imaginary Hermitian component $A_I$ and the real Hermitian componet $A_R$ are defined as follows:
\begin{equation*}
    A_I = \frac{A-A^*}{2 i}\,, \q A_R = \frac{A + A^*}{2}\,,
\end{equation*}
where $A^*$ is the adjoint operator of $A$ in the Hilbert sense. Moreover, we say that an operator $A$ is quasi-Hermitian operator if it is a sum of a selfadjoint operator and a compact one. For such kind of operators, we have a general resolvent bound under the condition (cf.\, \cite[Thm.7.7.1]{gil2003operator}):
\begin{equation}\label{con:resol}
    A_I \ \text{is a Hilbert-Schmidt operator.}
\end{equation}
\begin{theorem} \label{est:resol}
Under condition \eqref{con:resol}, the following bound for the norm of $(\lad-A)^{-1}$ holds,
\begin{equation}
    \norm{(\lambda - A)^{-1}} \le \frac{\sqrt{2}}{{\rm dist}(\lambda,\sigma(A))}\exp\left(\frac{g_I^2(A)}{{\rm dist}^2(\lambda,\sigma(A))}\right)\,,
\end{equation}
where the quantity $g_I(A)$ is given by
\begin{equation} \label{eqapp:gi}
    g_I(A) = \sqrt{2}\Big[\norm{A_I}^2_{HS} - \sum_{k = 0}^\infty(\im \lambda_k(A))^2\Big]^{\frac{1}{2}},
\end{equation}
where $\lambda_k(A)$ are the eigenvalues of $A$ counting multiplicity and $\norm{\dd}_{HS}$ denotes the Hilbert-Schmidt norm.
\end{theorem}
For our purpose, we write $\td$ as the sum of $T_D$ and $N_D^k := \td - T_D$, where $T_D$ is known to be a selfadjoint operator. We consider the kernel $K_N$ of the integral operator $N_D^k$:
\begin{equation*}
  K_N(x,y):= (k^2 + \nabla_x \ddiv_x)(g(x,y,k) - g(x,y,0))\,.
\end{equation*}
It is easy to see that when $x$ approaches $y$, the kernel has following singularity:
\begin{equation*}
    K_N(x,y) = O\left(\frac{1}{|x-y|}\right)\,.
\end{equation*}
It directly follows that $N_D^k$ and its imaginary Hermitian component $N_{D,I}^k$ 
are Hilbert-Schmidt operators. We further note the relation:
\begin{equation*}
    T^k_{D,I} = \frac{\td - T_D^{k,*}}{2 i} = \frac{N_D^k - N_D^{k,*}}{2 i} = N_{D,I}^k\,,
\end{equation*}
which helps us to conclude that $\td$ is a quasi-Hermitian operator satisfying condition \eqref{con:resol}, and thus Theorem \eqref{est:resol} can be applied. 

\subsection{Decay property and bound of the imaginary parts of eigenvalues} \label{app:B_2}
Formula \eqref{eqapp:gi} has suggested us that $\{\im \lad_k(A)\}$ is a bounded sequence and tends to zero when $k \to \infty$. Its detailed proof can be found in \cite[pp.106-107]{gil2003operator}. Here we provide a sketch  of the main argument for the sake of completeness. For a quasi-Hermitian operator $A$ satisfying condition \eqref{con:resol}, we have the following triangular representation:
\begin{equation*}
    A = D + V\,,
\end{equation*}
such that $\sigma(D) = \sigma(A)$, where $D$ is a normal operator and $V$ is a compact operator with $\sigma(V) = \{0\}$ and 
\begin{equation*}
    \norm{A_I}_{HS}^2 = \norm{D_I}_{HS}^2 + \norm{V_I}_{HS}^2< + \infty\,. 
\end{equation*}
Then, by using $\sigma(A) = \sigma(D)$ and the fact that $D$ is a normal operator, we can obtain
\begin{equation*}
    \norm{D_I}^2_{HS} = \sum_{k = 0}^\infty(\im \lambda_k(A))^2 < + \infty\,.
\end{equation*}
We end this appendix with the corresponding result for $\td$.
\begin{theorem} \label{thm:eigfree_region}
For the integral operator $\td$ defined in \eqref{def:td}, its spectrum $\sigma(\td)$ is contained in a strip in the complex plane:
\begin{equation*}
    \sigma(\td) \subset \{z\in \C\,;\q |\im z|\le C\}\q \text{for some}\ C\,,
\end{equation*}
and the imaginary parts of the eigenvalues in the spectrum consists of a $2$-power summable sequence, i.e., 
\begin{equation*}
    \sum_{i = 0}^\infty\left|\im \lambda_i(\td)\right|^2 < + \infty\,, \q \lad_i \in \sigma_f(\td)\,. 
\end{equation*}
\end{theorem}

\section{Some definitions, calculations and auxiliary  results for Section \ref{subsec:sphere}}
\label{app:C}
\subsection{Vector wave functions} \label{app:C_1}
Let $\ynm(\h{x}), n =0, 1,2,\cdots,\ m=-n,\cdots,n,$ be the spherical harmonics on $S^2$. The  vector spherical harmonics, which form a complete orthonormal system of $L_T^2(S^2)$ \cite[Theorem 6.25]{colton2012inverse}, are introduced as follows:
\begin{equation*}
    \unm = \frac{1}{\sqrt{n(n+1)}} \sgd \ynm\,,\q \vnm = \h{x} \t \unm\,, \ n = 1,2,\cdots,\ m=-n,\cdots,n\,.
\end{equation*}
Define the radiating electric multipole fields in $\R^3\backslash\{0\}$ for $n = 1,2,\cdots$ and $m=-m,\cdots,n$ \cite{monk2003finite}: 
\begin{align}
   \ete(k,x) & = \curl \{x h_n^{(1)}(k |x|)\ynm(\h{x})\} \notag\\ & =  - \sqrt{n(n+1)} h_n^{(1)}(k|x|)\vnm(\h{x}), \label{eq:raete} \\
   \etm(k,x)  & = - \frac{1}{i k } \curl \ete(k,x)  \notag\\ & = -\frac{\sqrt{n(n+1)}}{ik|x|}\hh_n(k|x|)\unm(\h{x}) - \frac{n(n+1)}{ik|x|}h_n^{(1)}(k|x|)\ynm(\h{x})\h{x}\,, \label{eq:raetm}
\end{align}
where $h_n^{(1)}(t)$ is the spherical Hankel function of the first kind and order $n$ and $\hh_n(t) := h_n^{(1)}(t) + t(h_n^{(1)})'(t)$. The entire electric multipole fields $\w{E}^{TE}_{n,m}(k,x)$ and  $\w{E}^{TM}_{n,m}(k,x)$ can be similarly introduced \cite{monk2003finite}:
\begin{align}
   \wete(k,x) & = \curl \{x j_n(k |x|)\ynm(\h{x})\} \notag\\ & =  - \sqrt{n(n+1)} j_n(k|x|)\vnm(\h{x}), \label{eq:enete} \\
   \wetm(k,x)  & = - \frac{1}{i k } \curl \wete(k,x)\notag\\ &  = -\frac{\sqrt{n(n+1)}}{ik|x|}\jj_n(k|x|)\unm(\h{x}) - \frac{n(n+1)}{ik|x|}j_n(k|x|)\ynm(\h{x})\h{x}\,, \label{eq:enetm}
\end{align}
where $j_n(t)$ is the spherical Bessel function of the first kind and order $n$ and $\jj_n$ is given by $\jj_n(t) := j_n(t) + tj'_n(t)$.  Then, a direct 
calculation gives us the tangential traces of the multipole fields:
\begin{equation} \label{eq:trramp}
    \begin{cases}
    \h{x} \t \ete(k,x) =  \sqrt{n(n+1)} h_n^{(1)}(k |x|)\unm(\h{x}) \\
    \h{x} \t \etm(k,x) = - \frac{\sqrt{n(n+1)}}{i k |x|} \hh_n(k|x|)\vnm(\h{x})
    \end{cases},
\end{equation}
and
\begin{equation}\label{eq:trenmp}
\begin{cases}
    \h{x} \t \wete(k_\lad,x) =  \sqrt{n(n+1)} j_n(k|x|)\unm(\h{x})\\
    \h{x} \t \wetm(k_\lad,x) = - \frac{\sqrt{n(n+1)}}{i k |x|}\jj_n(k|x|)\vnm(\h{x})
\end{cases}.
\end{equation}

We end this section with the addition formula of the Green's tensor $G_0(x,y,k)$ \cite[Theorem 6.29]{colton2012inverse}:
\begin{align} \label{eq:additiongreen}
    G_0(x,y,k) = & \sum^\infty_{n = 1}\frac{i k}{n(n+1)} \sum^n_{m=-n} \etm(x)\otimes\overline{\wetm}(y) \notag \\
    & + \sum^\infty_{n=1} \frac{i k}{n(n+1)} \sum^n_{m=-n}\ete(x) \otimes\overline{\wete}(y) 
     \q \text{for}\ |x|>|y|\,.
\end{align}

\subsection{Asymptotic expansions for spherical Bessel functions} \label{app:C_2}
We collect some standard results about asymptotic expansions for $j_n(z),\ n\ge 0$. For the complex variable $z$ with $|\arg(z)|< \pi$, the following asymptotics holds \cite[p.199]{watson1995treatise},
\begin{align} \label{eq:aymjn}
    j_n(z) = \frac{1}{z} \cos\left(z - \frac{n \pi}{2} - \frac{\pi}{2}\right) + e^{|\im z|}O\left(\frac{1}{|z|^2}\right)\q \text{as} \ |z|\to \infty\,.
\end{align}
Combining \eqref{eq:aymjn} with the following recurrence relations of Bessel functions \cite{olver2010nist,watson1995treatise}:
\begin{equation*}
    n j_{n-1}(z) - (n+1)j_{n+1}(z) = (2n+1)j_n'(z),
\end{equation*}
we see the asymptotic form of $j_n'(z)$:
\begin{equation} \label{eq:aymjnd}
    j_n'(z) = \frac{1}{z} \cos\left(z-\frac{n\pi}{2}\right)+ e^{|\im z|}O\left(\frac{1}{|z|^2}\right)\q \text{as} \ |z|\to \infty\,.
\end{equation}
By definition of $\jj_n(z)$, \eqref{eq:aymjn} and \eqref{eq:aymjnd}, it holds that
\begin{align} \label{eq:aymjjn}
    \jj_n(z) & = \frac{1}{z} \cos\left(z - \frac{n \pi}{2} - \frac{\pi}{2}\right) +  \cos\left(z-\frac{n\pi}{2}\right) + e^{|\im z|}O\left(\frac{1}{|z|}\right)\notag\\
    & = \cos\left(z-\frac{n\pi}{2}\right) + e^{|\im z|}O\left(\frac{1}{|z|}\right)\q \text{as} \ |z|\to \infty\,,
\end{align}
where we have also used the  observation:
\begin{equation} \label{eq:obserrem}
 \frac{e^{|\im z|} - 1}{2} \le |\cos(z)| = \left|\frac{e^{i\re z - \im z} + e^{-i\re z + \im z}}{2}\right| \le \frac{1 + e^{|\im z|}}{2}\,.
\end{equation}


\subsection{Auxiliary results for propagating functions} \label{app:C_3}
In this section, we first calculate the tangential traces of $\h{x} \t \td[\wete(k_\lad,\dd)](x)$ and $\h{x} \t \td[\wetm(k_\lad,\dd)](x)$ on the sphere $\p B(0,|x|)$ with radius $|x| > 1$, where $D = B(0,1)$. By the addition formula for the Green's tensor \eqref{eq:additiongreen} and the definition of $\td$, we have, by using the orthogonality of $\{\unm\}$ and $\{\vnm\}$, 
\begin{align} \label{eq:tancomm_1}
   \h{x} \t \td[\wete(k_\lad,\dd)](x) & = \frac{ik^3}{n(n+1)}\h{x}\t \ete(k,x) \int_{B(0,1)}\overline{\wete}(k,x)^t \dd \wete(k_\lad,x)dx \notag \\
    & = ik^3 \h{x} \t \ete(k,x) \int_0^1 j_n(kr) j_n(k_\lad r)r^2 dr \notag \\
    & = ik^3 \sqrt{n(n+1)}h^{(1)}_n(k|x|)\unm(\h{x}) \int_0^1 j_n(kr) j_n(k_\lad r)r^2 dr\,,
\end{align}
and   
\begin{align} \label{eq:tancomm_2}
   \h{x} \t \td[\wetm(k_\lad,\dd)](x)  & =  \frac{ik^3}{n(n+1)}\h{x} \t \etm(k,x) \int_{B(0,1)}\overline{\wetm}(k,x)^t \dd \wetm(k_\lad,x)dx \notag \\
  & = \frac{i k^3}{k k_\lad} \h{x} \t \etm(k,x) \int_0^1 \jj_n(kr) \jj_n(k_\lad r) + n(n+1) j_n(kr)j_n(k_\lad r) dr \notag\\
  & = - \frac{k\sqrt{n(n+1)}}{k_\lad |x|}\hh_n(k|x|)\vnm(\h{x})\int_0^1 \jj_n(kr) \jj(k_\lad r) + n(n+1) j_n(kr)j_n(k_\lad r) dr\,.
\end{align}
The integrals involved in \eqref{eq:tancomm_1} and \eqref{eq:tancomm_2} can
be explicitly calculated by the Lommel's integrals \cite{watson1995treatise} for $n \ge 1$:
\begin{equation} \label{eq:lommel_1}
    \int_0^1 j_n(kr)j_n(k_\lad r)r^2 dr = \frac{1}{k^2-k_\lad^2}\left[k_\lad j_n(k)j_{n-1}(k_\lad)-kj_{n-1}(k)j_n(k_\lad)\right]\,,
\end{equation}
and
 \begin{align} \label{eq:lommel_2}
& \int^1_0 n(n+1) j_n(kr)j_n(k_\lad r) + \jj_n(kr)\jj_n(k_\lad r) dr \notag \\ = & \frac{k k_\lad}{2n+1}\left((n+1)\int_0^1 j_{n-1}(k r)j_{n-1}(k_\lad r)r^2dr + n\int_0^1 j_{n+1}(k r)j_{n+1}(k_\lad r)r^2dr\right)\,.
\end{align}

We next provide the calculations and estimates for Proposition \ref{prop:asylarge}. We recall the following asymptotic forms of  $j_n$ and
$h_n^{(1)}$ for large $n$ that uniformly hold for $z$ in a compact subset of $\C$ away from the origin:
\begin{equation} \label{eq:asyjh}
    j_n(z) = O\left(\left(\frac{e|z|}{2(n+1)}\right)^{n+1}\right)\,,\q h_n^{(1)}(z) = O\left(\left(\frac{2n}{e|z|}\right)^{n}\right) \quad \text{as}\ n \to \infty\,, 
\end{equation}
as a result of series expansions of $j_n$ and $h_n^{(1)}$ and Stirling's formula (cf.\,\cite[p.30]{colton2012inverse}). For the propagating function $\vp_n^{\lad,1}(kt)$, by \eqref{eq:tancomm_1} and \eqref{eq:lommel_1}, a direct application of \eqref{eq:asyjh} gives us, for $t$ from a compact subset of $(1,+\infty)$,
\begin{align*}
    \vp_n^{\lad,1}(kt) & = O\left(n  \left(\frac{2n}{e k t}\right)^{n} \frac{1}{|k_\lad|^2}\left[|k_\lad| \left(\frac{e k }{2(n+1)}\right)^{n+1}\left(\frac{e |k_\lad|}{2 n}\right)^{n} +\left(\frac{e k}{2n}\right)^{n}\left(\frac{e|k_\lad|}{2(n+1)}\right)^{n+1}\right]\right) \\
    & = O\left(n  \frac{1}{t^n} \frac{1}{|k_\lad|^2}\left[|k_\lad| \left(\frac{1}{2(n+1)}\right)^{n+1}\left(e |k_\lad|\right)^n +\left(\frac{e|k_\lad|}{2(n+1)}\right)^{n+1}\right]\right)  \\
    & = O\left(n  \frac{1}{t^n} \frac{1}{|k_\lad|^2} \left(\frac{e|k_\lad|}{2(n+1)}\right)^{n+1}\right) = O\left(\left(\frac{e}{2 t}\right)^{n+1} \frac{|k_\lad|^{n-1}}{(n+1)^n}\right)\,.
\end{align*}
A very similar but more complicated calculation yields the second  estimate in \eqref{eq:aympfvp}. We omit the details here.

The following two lemmas were used for Theorem \ref{thm:loceig}.
\begin{lemma} \label{lem:localforeig0}
Suppose that $f(x)$ is a continuous function on $[0,+\infty)$ with $f(x)\to 0$ as $x \to + \infty$. We have 
\begin{equation*}
    \max_{x \in [0,a]}|f(x)| = \max_{x \in [0,+\infty)}|f(x)|\,
\end{equation*}
\zz{for any $a \in \R$ larger than some fixed $a_0 >0$}. Moreover, let $\{a_n\}$ be a sequence such that $a_n \to + \infty$ when $n \to + \infty$, 
then $\{f(a_n x)\}$ are localized near the origin in the sense that 
\begin{equation*}
    \lim_{n \to +\infty}\frac{\max_{x\in[a,1]}|f(a_n x)|}{\max_{x\in[0,1]}|f(a_n x)|} = 0\,.
\end{equation*}
\end{lemma}

\begin{lemma} \label{lem:localforeig}
For $j_n(z)$ and $\jj_n(z)/z$, the following estimates uniformly hold for $t \in [0,1]$, 
\begin{equation}\label{eq:est:j_n1}
     \left|j_n(z^1_{n,l}t) - j_n(\w{z}^1_{n,l}t)\right| = O(l^{-1})\,, \quad 
     \left|\frac{\jj_n(z^2_{n,l}t)} {z^2_{n,l}t} - \frac{\jj_n(\w{z}^2_{n,l}t)}{\w{z}^2_{n,l}t}\right|  =  O(l^{-1})\,,
\end{equation}
when $l$ tends to infinity. Here, $\{z^i_{n,l}\}$ and $\{\w{z}_{n,l}^i\}$, $i=1,2$, are the same as the ones in \eqref{eq:est_zeros1} and \eqref{eq:est_zeros2}.
\end{lemma}
\begin{proof}
For the first estimate, we first observe from \eqref{eq:aymjnd} and \eqref{eq:obserrem} that $|j_n'(z)|$ is bounded by a constant $M$ on the strip:
\begin{equation} \label{eq:aux:strip}
    \{z\in \C\,;\ |\im z| \le C,\ -\frac{\pi}{2}< \arg(z)\le \frac{\pi}{2}\}\,,
\end{equation}
where the constant $C \in \R$  is chosen such that $\{z^1_{n,l}\}_{l \in \NN}$ lie in \eqref{eq:aux:strip}. Then we have, by using the analyticity of $j_n(z)$ and the contour integral,
\begin{equation*}
    |j_n(z_{n,l}t) - j_n(\w{z}_{n,l}t)| = \left|\int_{\gamma} j_n'(\xi) d\xi \right| \le M |z_{n,l} - \w{z}_{n,l}||t|\,, \q t \in [0,1]\,,
\end{equation*}
where $\gamma$ is the segment connecting $z^1_{n,l}t$ with $\w{z}^1_{n,l}t$. Combining the above estimate with \eqref{eq:est_zeros1}, we can conclude that the first estimate in \eqref{eq:est:j_n1} holds. For the second estimate, it suffices to note that the derivative of $\jj_n(z)/z$ is entire and satisfies the following asymptotic form:
\begin{equation*}
    \left(\frac{\jj_n(z)}{z}\right)' = \frac{j_n'(z)z-j_n(z)}{z^2} + j_n''(z) =  \frac{1}{z} \cos\left(z - \frac{n \pi}{2} - \frac{\pi}{2}\right) + e^{|\im z|}O\left(\frac{1}{|z|^2}\right)\q \text{as} \ |z|\to \infty\,,
\end{equation*}
and can also bounded on a strip of the form \eqref{eq:aux:strip} with  a different constant $C$ such that it contains the zeros $\{z_{n,l}^2\}$ of $f_n^2(z)$. Then, in view of \eqref{eq:est_zeros2}, the same argument as the previous one allows us to complete the proof. 
\end{proof}

\if \commentflag = \ct
By \eqref{eq:farexgre}, we have
\begin{align} \label{eq:farexccg}
    \curl \curl (g(x,y,k)\I) = k^2 \frac{e^{ik|x|}}{4 \pi |x|}\{e^{-ik\h{x}\dd y}(\I - \h{x}\otimes\h{x})+O(\frac{1}{|x|})\},
\end{align}
while \eqref{eq:additiongreen} gives us
\begin{align}
     \curl \curl (g(x,y,k)\I) = & \sum^\infty_{n = 1}\frac{i k^3}{n(n+1)} \sum^n_{m=-n} \etm(x)\otimes\overline{\wetm}(y) \notag\\
    & + \sum^\infty_{n=1} \frac{i k^3}{n(n+1)} \sum^n_{m=-n}\ete(x)\otimes \overline{\wete}(y).\label{eq:addccg}
\end{align}
Recall the asymptotics of spherical hankel function of the first kind and its derivative(cf.\,\cite[p.31]{colton2012inverse}): as $t \to \infty$,
\bb
h_n^{(1)}(t) \sim \frac{e^{it}}{t}e^{-i\frac{n+1}{2}\pi}  \\
(h_n^{1})'(t) \sim \frac{e^{it}}{t}e^{-i\frac{n}{2}\pi}.
\ee
Then a direct calculation gives us the far field behavior of radiating multiple electric fields \eqref{eq:raete}-\eqref{eq:raetm}:
\bb \label{eq:farexprme}
\ete(k,x) \sim \frac{e^{ik|x|}}{k|x|} e^{-i\frac{n+1}{2}\pi} (-\sqrt{n(n+1)}\vnm(\h{x})) \\
\etm(k,x) \sim  \frac{e^{ik|x|}}{k|x|} e^{-i\frac{n+1}{2}\pi} (-\sqrt{n(n+1)}\unm(\h{x}).
\ee 
 Substituting \eqref{eq:farexprme} into \eqref{eq:addccg} and comparing it with \eqref{eq:farexccg}, we can write  Jacobi-anger expansion of $e^{-ik\h{x}\dd y}(\I - \h{x}\otimes \h{x})$:
 \begin{align} \label{eq:jaexp}
    e^{-ik\h{x}\dd y}(\I - \h{x}\otimes\h{x})  = - \sum_{n=1}^\infty \frac{ 4 \pi(-i)^n}{\sqrt{n(n+1)}} \sum_{m=-n}^n   \vnm(\h{x})\otimes \overline{\wete}(y)+ \unm(\h{x})\otimes \overline{\wetm}(y)\,,\ \text{for}\ y \in D.
\end{align}
By noting that $\sum^n_{m = -n} \ynm(\h{x})\overline{\ynm}(\h{y})$ is real, we can find the addition formula for the imaginary part of $ G_0(x,y,k)$ via \eqref{eq:additiongreen} directly:\fn{add more details}
\begin{align} \label{eq:expimgreen}
  \im G_0(x,y,k) = \sum^\infty_{n = 1} \frac{k}{n(n+1)} \sum^n_{m = -n} \wete(x) \otimes \overline{\wete}(y)
      + \wetm(x)\otimes \overline{\wetm}(y).
\end{align}
for all $x,y$ with $|x|\neq |y|$.
\fi

\if \commentflag = \ct
\newpage
we can readily write the corresponding series expansion (Fourier expansion) of $\Psi_\infty(\h{x},z)$:
\begin{equation*}
    \Psi_\infty(\h{x},z) = \sum_{n =1}^\infty \sum_{m=-n}^n \aanm(z)\unm(\h{x}) + \bbnm(z)\vnm(\h{x})\,,
\end{equation*}
where the Fourier coefficients $\{\aanm\}$ and $\{\bbnm\}$ are given by 
\begin{equation} \label{def:coeffiint}
    \aanm(z) = \int_{S^2} \overline{\vnm}^t(\h{x}) \dd \Psi_\infty(\h{x},z)d\sigma(\h{x})\,,\q  
    \bbnm(z) = \int_{S^2} \overline{\unm}^t(\h{x}) \dd \Psi_\infty(\h{x},z)d\sigma(\h{x})\,.
\end{equation}
Then the series expansion of $\h{I}_2(z)$ follows from the Parseval's theorem,
\begin{equation*}
    \h{I}_2(z) = k \sum_{n =1}^\infty \sum^n_{m = - n} |\aanm(z)|^2 + |\bbnm(z)|^2\,.
\end{equation*}
It is worth noting that the expansion allows us to separate the sampling position $z$ and the observation direction $\h{x}$ of $\Psi_\infty(\h{x},z)$. We hence expect that the investigation of $\{\aanm\}$ and $\{\bbnm\}$ can provide us more precise and detailed information on the resolvability of $\h{I}_2$. 

To reveal the relations between these coefficients and the high contrast parameter $\tau$, similar to the previous subsection,  we first write the 
the following Lippmann-Schwinger equation for $\Psi$:
\begin{equation}  \label{eq:lippsch}
\Psi = \tau \td[\Psi + \im G_0p],
\end{equation}
which can be further written as,
\begin{equation} \label{eq:lippsch_2}
    \Psi + \im G_0p  = (1 - \tau \td)^{-1}[\im G_0p],
\end{equation}
by the same derivation as the one for \eqref{eq:reprevp}. Then, by asymptotic expansion of Green's tensor $G_0$: 
\begin{equation} \label{eq:farexgre}
G_0(x,y,k) = (\I-\h{x}\otimes \h{x})\frac{e^{ik|x|}}{4 \pi |x|}e^{-ik \h{x}\dd y}
+ O(\frac{1}{|x|^2})\,, \ \text{for} \ y \in D, \ \text{as}\ |x| \to \infty\,,
\end{equation}
along with \eqref{eq:lippsch}, we have
\begin{equation} \label{eq:exppsi}
    \Psi(x,z) = \tau k^2 (\I-\h{x}\otimes \h{x})\frac{e^{ik|x|}}{4 \pi |x|} \int_D e^{-ik \h{x}\dd y} (\Psi(y,z) + \im G_0(y,z,k))dy\,.
\end{equation}
The integral representation of the far-field pattern $\Psi_\infty(\h{x},z)$ hence follows from its definition \eqref{def:ffp} and \eqref{eq:lippsch_2},
\begin{align} \label{eq:farfieldpsi}
    \Psi_\infty(\h{x},z) & = \frac{k^2}{4 \pi} (\I - \h{x}\otimes\h{x}) \int_{D} e^{-ik\h{x}\dd y}(\tau^{-1} - \td)^{-1}[\im G_0(\dd,z,k)p](y)dy\,.
\end{align} 
Furthermore, the Fourier coefficients $\aanm(z) $ and $\bbnm(z)$ can be computed by \eqref{def:coeffiint} and  \eqref{eq:farfieldpsi}, as well as the Jacobi-Anger expansion \eqref{eq:jaexp} of $e^{-ik\h{x}\dd y}(\I - \h{x}\otimes \h{x})$(see Appendix \ref{app:A}. We summarize the above discussion in the following result.
\begin{theorem} \label{thm:seriesexpI2}
The resolution  $\h{I}_2$ of the imaging functional $I_2$ has the following series expansion: 
\begin{equation}
    \h{I}_2(z) = k \sum_{n =1}^\infty \sum^n_{m = - n} |\aanm(z)|^2 + |\bbnm(z)|^2\,,
\end{equation}
where $\{\aanm(z)\}$ and $\{\bbnm(z)\}$ are the Fourier coefficients of $\Psi_\infty(\h{x},z)$ with respect to the vector spherical harmonics, given by
\begin{equation}
    \begin{cases}
    \aanm(z) = k^2 \frac{-(-i)^n}{\sqrt{n(n+1)}}\int_D\overline{\wete}^t(y)\dd(\tau^{-1} - \td)^{-1}[\im G_0(\dd,z,k)p](y)dy\,, \\
    \bbnm(z) = k^2 \frac{-(-i)^n }{\sqrt{n(n+1)}}\int_D\overline{\wetm}^t(y)\dd(\tau^{-1} - \td)^{-1}[\im G_0(\dd,z,k)p](y)dy\,.
    \end{cases}
\end{equation}
\end{theorem}

the Funk-Hecke theorem, that is,
\begin{equation*}
    \int_{S^2}e^{i k  \h{x}\dd (\xi-y)}d\sigma(\h{x})  = \frac{4 \pi  sin(k  |\xi-y|)}{k|\xi-y|}= \frac{16 \pi^2}{k} \im g(\xi,y,k),
\end{equation*}
and an observation:
\begin{equation*}
     k^2(\I - \h{x}\otimes\h{x})e^{ik \h{x}\dd(\xi -y)} = (k^2\I + \na_y \ddiv_y)e^{ik \h{x}\dd(\xi -y)},
\end{equation*}
we can further derive 
\begin{align*}
    \h{I}(z) = &
  k \frac{\tau^2}{(4 \pi)^2}k^4 \int_{S^2}\int_{D\t D}(1 - \tau T_D^{-k})^{-1}[\im G_0p](\xi)^t \dd (\I - \h{x}\otimes\h{x})(1 - \tau \td)^{-1}[\im G_0p](y) e^{ik \h{x}\dd(\xi-y)}dyd\xi d \sigma(\h{x}) \\
    = &k^4 \tau^2 \int_{D\t D}(1 - \tau T_D^{-k})^{-1}[\im G_0p](\xi)^t \dd \im G_0(\xi,y,k)(1 - \tau \td)^{-1}[\im G_0p](y) dyd\xi
\end{align*}
\fi





\begin{thebibliography}{10}

    \bibitem{albanese2006inverse}
    R.~Albanese and P.~B. Monk.
    \newblock The inverse source problem for {M}axwell's equations.
    \newblock {\em Inverse problems}, 22:1023, 2006.
    
    \bibitem{ammari2018super}
    H.~Ammari, Y.~T. Chow, and J.~Zou.
    \newblock Super-resolution in imaging high contrast targets from the
      perspective of scattering coefficients.
    \newblock {\em Journal de Math{\'e}matiques Pures et Appliqu{\'e}es},
      111:191--226, 2018.
    
    \bibitem{ammari2016surface}
    H.~Ammari, Y.~Deng, and P.~Millien.
    \newblock Surface plasmon resonance of nanoparticles and applications in
      imaging.
    \newblock {\em Archive for Rational Mechanics and Analysis}, 220:109--153,
      2016.
    
    \bibitem{ammari2017sub}
    H.~Ammari, B.~Fitzpatrick, D.~Gontier, H.~Lee, and H.~Zhang.
    \newblock Sub-wavelength focusing of acoustic waves in bubbly media.
    \newblock {\em Proceedings of the Royal Society A: Mathematical, Physical and
      Engineering Sciences}, 473:20170469, 2017.
    
    \bibitem{ammari2018minnaert}
    H.~Ammari, B.~Fitzpatrick, D.~Gontier, H.~Lee, and H.~Zhang.
    \newblock Minnaert resonances for acoustic waves in bubbly media.
    \newblock {\em Annales de l'Institut Henri Poincar{\'e} C, Analyse non
      lin{\'e}aire}, 35:1975--1998, 2018.
    
    \bibitem{ammari2018mathematical}
    H.~Ammari, B.~Fitzpatrick, H.~K~and, M.~Ruiz, S.~Yu, and H.~Zhang.
    \newblock {\em Mathematical and computational methods in photonics and
      phononics}, volume 235.
    \newblock American Mathematical Soc., 2018.
    
    \bibitem{ammari2014medium}
    H.~Ammari, J.~Garnier, J.~De~Rosny, and K.~Sølna.
    \newblock Medium induced resolution enhancement for broadband imaging.
    \newblock {\em Inverse problems}, 30:085006, 2014.
    
    \bibitem{ammari2016mathematics}
    H.~Ammari, J.~Garnier, H.~Kang, L.~Nguyen, and L.~Seppecher.
    \newblock {\em Multi-Wave Medical Imaging}, volume~2 of {\em Modelling and
      Simulation in Medical Imaging}.
    \newblock World Scientific, London, 2017.
    
    \bibitem{ammari2013enhancement}
    H.~Ammari, H.~Kang, H.~Lee, M.~Lim, and S.~Yu.
    \newblock Enhancement of near cloaking for the full {M}axwell equations.
    \newblock {\em SIAM Journal on Applied Mathematics}, 73:2055--2076, 2013.
    
    \bibitem{ammari2017mathematicalscalar}
    H.~Ammari, P.~Millien, M.~Ruiz, and H.~Zhang.
    \newblock Mathematical analysis of plasmonic nanoparticles: the scalar case.
    \newblock {\em Archive for Rational Mechanics and Analysis}, 224:597--658,
      2017.
    
    \bibitem{ammari2016plasmaxwell}
    H.~Ammari, M.~Ruiz, S.~Yu, and H.~Zhang.
    \newblock Mathematical analysis of plasmonic resonances for nanoparticles: the
      full {M}axwell equations.
    \newblock {\em Journal of Differential Equations}, 261:3615--3669, 2016.
    
    \bibitem{ammari2015mathematical}
    H.~Ammari and H.~Zhang.
    \newblock A mathematical theory of super-resolution by using a system of
      sub-wavelength {H}elmholtz resonators.
    \newblock {\em Communications in Mathematical Physics}, 337:379--428, 2015.
    
    \bibitem{ammari2015super}
    H.~Ammari and H.~Zhang.
    \newblock Super-resolution in high-contrast media.
    \newblock {\em Proceedings of the Royal Society A: Mathematical, Physical and
      Engineering Sciences}, 471:20140946, 2015.
    
    \bibitem{amrouche1998vector}
    C.~Amrouche, C.~Bernardi, M.~Dauge, and V.~Girault.
    \newblock Vector potentials in three-dimensional non-smooth domains.
    \newblock {\em Mathematical Methods in the Applied Sciences}, 21:823--864,
      1998.
    
    \bibitem{bao2013near}
    G.~Bao and P.~Li.
    \newblock Near-field imaging of infinite rough surfaces.
    \newblock {\em SIAM Journal on Applied Mathematics}, 73:2162--2187, 2013.
    
    \bibitem{bao2014near}
    G.~Bao and P.~Li.
    \newblock Near-field imaging of infinite rough surfaces in dielectric media.
    \newblock {\em SIAM Journal on Imaging Sciences}, 7:867--899, 2014.
    
    \bibitem{bao2019stability}
    G.~Bao, P.~Li, and Y.~Zhao.
    \newblock Stability for the inverse source problems in elastic and
      electromagnetic waves.
    \newblock {\em Journal de Math{\'e}matiques Pures et Appliqu{\'e}es}, 2019.
    
    \bibitem{budko2006spectrum}
    N.~V. Budko and A.~B. Samokhin.
    \newblock Spectrum of the volume integral operator of electromagnetic
      scattering.
    \newblock {\em SIAM Journal on Scientific Computing}, 28:682--700, 2006.
    
    \bibitem{chen2013reverseac}
    J.~Chen, Z.~Chen, and G.~Huang.
    \newblock Reverse time migration for extended obstacles: acoustic waves.
    \newblock {\em Inverse Problems}, 29:085005, 2013.
    
    \bibitem{chen2013reverse}
    J.~Chen, Z.~Chen, and G.~Huang.
    \newblock Reverse time migration for extended obstacles: electromagnetic waves.
    \newblock {\em Inverse Problems}, 29:085006, 2013.
    
    \bibitem{colton2012inverse}
    D.~Colton and R.~Kress.
    \newblock {\em Inverse acoustic and electromagnetic scattering theory},
      volume~93.
    \newblock Springer Science \& Business Media, 2012.
    
    \bibitem{costabel2010volume}
    M.~Costabel, E.~Darrigrand, and E.~H. Kon\'{e}.
    \newblock Volume and surface integral equations for electromagnetic scattering
      by a dielectric body.
    \newblock {\em Journal of Computational and Applied Mathematics},
      234:1817--1825, 2010.
    
    \bibitem{costabel2012essential}
    M.~Costabel, E.~Darrigrand, and H.~Sakly.
    \newblock The essential spectrum of the volume integral operator in
      electromagnetic scattering by a homogeneous body.
    \newblock {\em Comptes Rendus Mathematique}, 350:193--197, 2012.
    
    \bibitem{costabel2015volume}
    M.~Costabel, E.~Darrigrand, and H.~Sakly.
    \newblock Volume integral equations for electromagnetic scattering in two
      dimensions.
    \newblock {\em Computers \& Mathematics with Applications}, 70:2087--2101,
      2015.
    
    \bibitem{davies2007linear}
    E.~B. Davies.
    \newblock {\em Linear operators and their spectra}, volume 106.
    \newblock Cambridge University Press, 2007.
    
    \bibitem{folland1995introduction}
    G.~B. Folland.
    \newblock {\em Introduction to partial differential equations}.
    \newblock Princeton university press, 1995.
    
    \bibitem{gil2003operator}
    M.I. Gil.
    \newblock {\em Operator functions and localization of spectra}.
    \newblock Springer, 2003.
    
    \bibitem{gohberg1990classes}
    I.~Gohberg, S.~Goldberg, and M.A. Kaashoek.
    \newblock {\em Classes of linear operators}, volume~49.
    \newblock Birkh{\"a}user, 1990.
    
    \bibitem{grebenkov2013geometrical}
    D.~S. Grebenkov and B.~T. Nguyen.
    \newblock Geometrical structure of {L}aplacian eigenfunctions.
    \newblock {\em siam REVIEW}, 55:601--667, 2013.
    
    \bibitem{ito2013direct}
    K.~Ito, B.~Jin, and J.~Zou.
    \newblock A direct sampling method for inverse electromagnetic medium
      scattering.
    \newblock {\em Inverse Problems}, 29:095018, 2013.
    
    \bibitem{erxiong1994bounds}
    E.~Jiang.
    \newblock Bounds for the smallest singular value of a {J}ordan block with an
      application to eigenvalue perturbation.
    \newblock {\em Linear Algebra and its Applications}, 197:691--707, 1994.
    
    \bibitem{monk2003finite}
    P.~Monk.
    \newblock {\em Finite element methods for {M}axwell's equations}.
    \newblock Oxford University Press, 2003.
    
    \bibitem{nguyen2013localization}
    B.~T. Nguyen and D.~S. Grebenkov.
    \newblock Localization of {L}aplacian eigenfunctions in circular, spherical,
      and elliptical domains.
    \newblock {\em SIAM Journal on Applied Mathematics}, 73:780--803, 2013.
    
    \bibitem{olver2010nist}
    F.~W. Olver, D.~W. Lozier, R.~F. Boisvert, and C.~W. Clark.
    \newblock {\em NIST handbook of mathematical functions hardback and {CD-ROM}}.
    \newblock Cambridge university press, 2010.
    
    \bibitem{rahola2000eigenvalues}
    J.~Rahola.
    \newblock On the eigenvalues of the volume integral operator of electromagnetic
      scattering.
    \newblock {\em SIAM Journal on Scientific Computing}, 21:1740--1754, 2000.
    
    \bibitem{wahab2014electromlossy}
    A.~Wahab, A.~Rasheed, T.~Hayat, and R.~Nawaz.
    \newblock Electromagnetic time reversal algorithms and source localization in
      lossy dielectric media.
    \newblock {\em Communications in Theoretical Physics}, 62:779, 2014.
    
    \bibitem{wahab2014electromagnetic}
    A.~Wahab, A.~Rasheed, R.~Nawaz, and S.~Anjum.
    \newblock Localization of extended current source with finite frequencies.
    \newblock {\em Comptes Rendus Math{\'e}matique}, 352:917--921, 2014.
    
    \bibitem{watson1995treatise}
    G.~N. Watson.
    \newblock {\em A treatise on the theory of {B}essel functions}.
    \newblock Cambridge university press, 1995.
    
    \end{thebibliography}

\end{document}